\documentclass[11pt,reqno]{amsart}

\usepackage[dvipsnames,table]{xcolor}

\usepackage{amsmath,amssymb,amsthm,mathtools}
\usepackage{calligra}
\usepackage{siunitx}

\usepackage{graphicx}
\usepackage{caption}
\usepackage{subcaption}
\usepackage{wrapfig}
\usepackage{float}
\usepackage[percent]{overpic}
\usepackage{rotating} % Para girar a figura

\usepackage[left=2cm, right=2cm, top=3cm, bottom=2cm]{geometry}

\usepackage{tikz}
\usepackage{tikz-cd}
\usetikzlibrary{shapes, arrows, positioning, fit, backgrounds, matrix, arrows.meta, bending}
\usepackage{pgfplots}
\pgfplotsset{compat=1.18}

\usepackage{comment}
\usepackage{orcidlink}
\usepackage{hyperref} % Recomenda-se que hyperref seja o último pacote carregado

\newcommand{\R}{\mathbb{R}}

\newcommand{\Z}{\mathbb{Z}}
\newcommand{\N}{\mathbb{N}}
\newcommand{\es}{\mathbb{S}}

\theoremstyle{plain}
\newtheorem{theorem}{Theorem}[section]
\newtheorem{lemma}[theorem]{Lemma}
\newtheorem{proposition}[theorem]{Proposition}
\newtheorem{corollary}[theorem]{Corollary}

\theoremstyle{definition}
\newtheorem{definition}[theorem]{Definition}
\newtheorem{example}[theorem]{Example}

\theoremstyle{remark}
\newtheorem{remark}[theorem]{Remark}

\begin{document}

\title[Singularity Collisions]{Singularity Collisions through Homotopical Dynamical cancellation}

% --- Autores e Afiliações (Padrão amsart) ---

\author[D. V. S. Lima]{D. V. S. Lima}
\address{Center for Mathematics, Computing and Cognition, Federal University of ABC, São Paulo, Brazil.}
\email{dahisy.lima@ufabc.edu.br}
\thanks{D.V.S. Lima was partially supported by São Paulo Research Foundation (FAPESP) under Grants 2025/12435-9,  22/16455-6 and  19/21181-0.}

\author[K. A. de Rezende]{K. A. de Rezende}
\address{Institute of Mathematics, Statistics and Scientific Computing, University of Campinas, São Paulo, Brazil.}
\email{ketty@ime.unicamp.br}
\thanks{K. A. de Rezende was partially supported by São Paulo Research Foundation (FAPESP) under Grants  22/16455-6 and 18/13481-0.}

\author[D. Tenório]{D. Tenório}
\address{Center for Mathematics, Computing and Cognition, Federal University of ABC, São Paulo, Brazil.}
\email{denilson.tenorio@ufabc.edu.br}
\thanks{D. Tenório was financed by the Coordenação de Aperfeiçoamento de Pessoal de Nível Superior - Brazil (CAPES) - Finance code 001.}


\begin{abstract}
We introduce collisions of invariant sets and, in particular, consider dynamical homotopical cancellations that preserve the homotopy type of the underlying singular manifold. We develop the theory of homotopical dynamical cancellation for generalized Gutierrez–Sotomayor (GGS) flows defined on GGS manifolds. This framework extends the classical cancellation theory of Morse flows to the singular setting. To effectively capture these homotopical cancellations, we introduce a GGS chain complex, which encodes essential dynamical and algebraic-topological information. Furthermore, we provide a spectral sequence analysis of a filtered GGS chain complex, demonstrating a bijective correspondence between algebraic cancellations of the modules of the spectral sequence and homotopical dynamical cancellations in the GGS flow. Several illustrative examples are presented, highlighting the practical applicability of the proposed framework.
\end{abstract}

\maketitle

\tableofcontents

\section{Introduction} 		   
\label{sec:intro}

The concept of cancellation was first systematically investigated in the context of Morse theory, specifically for a Morse function $f : M \to \mathbb{R}$, where $M$ is a compact, differentiable, n-manifold. Marston Morse devoted several papers to this topic, including \cite{morse1960existence, morse1964elimination, morse1965differential}. In particular, he introduced the idea of cancellation for critical points of index $0$ and $n$ in \cite{morse1960existence}, and later extended the concept to critical points of intermediate index in \cite{morse1964elimination, morse1965differential}.

The main idea is to modify a vector field in an arbitrarily small neighborhood of  two consecutive critical points and the trajectory connecting them, in order to produce a new vector field without these critical points, as illustrated in Figure \ref{fig_intr_00}.

\begin{figure}[!ht]
\centering
\begin{overpic}[width=9cm]{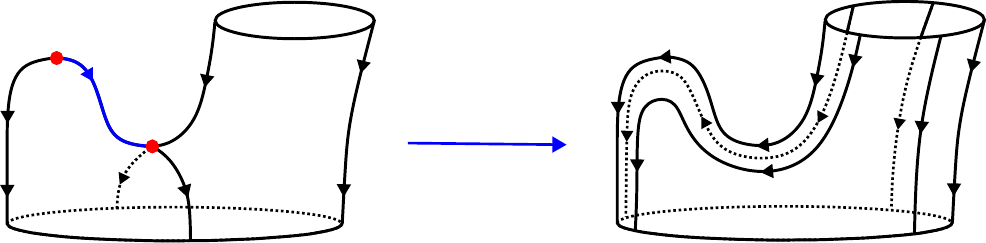}
\put(4.5,20){$x_1$}
\put(13.5,11.5){$x_2$}
\end{overpic}
\caption{Cancellation between non degenerate critical points  $x_1$ and $x_2$,
in dimension $2$.}
\label{fig_intr_00}
\end{figure}

The concept of cancellation became a powerful tool when Smale employed it in 1960 as part of his celebrated proof of the generalized Poincaré conjecture \cite{smale1961generalized}. A more didactic approach to cancellation in the smooth setting can be found in \cite{milnor1965lectures}.

In classical Morse theory, consecutive critical points can be canceled by suitably perturbing the flow. However, this approach fails in the non-smooth, singular context due to the presence of topological singularities. These singularities exhibit a rigidity in their local structure that is intrinsically linked to the topology of the underlying space and, in general, cannot be removed through perturbation. Eliminating them typically compromises essential topological, geometrical, and analytical properties.

Due to this rigidity of the structure, the notion of homotopical dynamical cancellation, or simply homotopical cancellation, was introduced in \cite{2021homotopical} as a generalization of classical cancellation theory to the singular setting, particularly for the class of singular manifolds known as Gutierrez–Sotomayor manifolds. In this framework, the singularities are not eliminated; rather, the manifold is homotopically deformed so that singularities merge, giving rise to a more dynamically intricate singularity. Thus, homotopically cancelling topological singularities entails not only modifications of the vector field but also a homotopical deformation of the underlying manifold itself. This class of singular manifolds was presented by Gutierrez and Sotomayor in \cite{GS}. Their goal was to generalize the theorems on structurally stable vector fields tangent to compact smooth 2-manifolds, originally developed by Peixoto and Peixoto in \cite{Peixoto1, Peixoto}. This class includes, in addition to regular points denoted by $\mathcal{R} = \{(x, y, z) \in \mathbb{R}^3 \mid z = 0\}$, the following types of singular points:
\begin{itemize}
    \item Cone: $\mathcal{C} = \{(x, y, z) \in \mathbb{R}^3 \mid z^2 - x^2 - y^2 = 0\}$
    \item Cross-cap: $\mathcal{W} = \{(x, y, z) \in \mathbb{R}^3 \mid zx^2 - y^2 = 0, z\geq0\}$
    \item Double crossing: $\mathcal{D} = \{(x, y, z) \in \mathbb{R}^3 \mid xy = 0\}$
    \item Triple crossing: $\mathcal{T} = \{(x, y, z) \in \mathbb{R}^3 \mid xyz = 0\}$
\end{itemize}

One of the key differences between homotopical cancellation and classical cancellation lies in the fact that, after cancellation, the resulting manifold preserves only the homotopy type of  $M$ whereas in the classical case it preserves the homeomorphism class of $M$.

This distinction highlights one of the main challenges in the singular setting. When homotopies are carried out along flow lines, even within smooth regions, the process may take us outside the class of Gutierrez–Sotomayor (GS) manifolds. As shown in \cite{2021homotopical}, even when cancellations were restricted to GS singularities of the same type, it was still necessary to introduce new ones, specifically, singularities with $n$-sheets, as illustrated in Figure~\ref{fig_intr_01}.

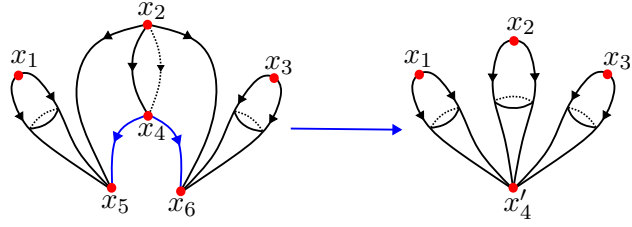
\begin{figure}[!ht]
\centering
\begin{overpic}[unit=1mm, scale=.25, width=8cm]{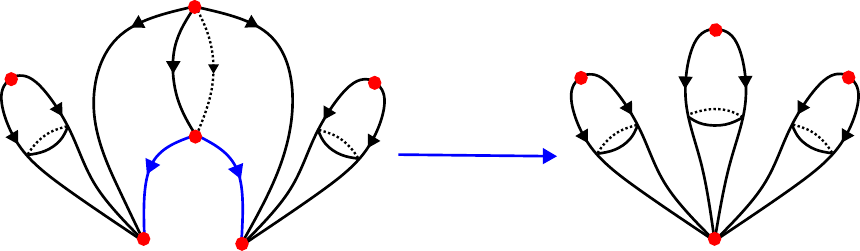}
\put(0,22){$x_1$}
\put(21,30){$x_2$}
\put(42,21){$x_3$}
\put(21,10){$x_4$}
\put(15,-2){$x_5$}
\put(26,-2.5){$x_6$}
\put(65,22){$x_1$}
\put(82,27.5){$x_2$}
\put(98,21.7){$x_3$}
\put(81.5,-2.5){$x'_4$}
\end{overpic}
\caption{Homotopical dynamical cancellation between $x_4$, $x_5$, $x_6$ and the flow lines connecting them generating a $3$-sheet cone singularity $x'_4$.}
\label{fig_intr_01}
\end{figure}

The use of algebraic tools, particularly spectral sequences,  has gained considerable traction in recent years as a powerful framework for extracting information from both smooth and singular dynamical systems, as well as for detecting  cancellation phenomena \cite{BLMdRS, bertolim2017algebraic, LMdRS, 2021homotopical}. Additionally, significant computational advances in spectral sequence analysis have been reported \cite{edelsbrunner2010computational, zomorodian2005topology}.

In this paper, our goal is to perform homotopical deformations of flow lines, and consequently, of the phase space itself. These deformations lead to collisions of singularities, reducing their total number by producing more degenerate ones, while preserving the homotopy type of the phase space. In a certain sense, this process moves in the opposite direction of a resolution of singularities. The end result is a phase space with fewer, yet more intricate, singularities.

In Section \ref{sec:collisions}, we introduce the concept of \textit{collisions of invariant sets}, which provides a  more comprehensive framework than standard notions of cancellation. Within this setting, we present a more precise formulation of \textit{homotopical dynamical cancellation}, recovering the definition given in \cite{2021homotopical} and we proof the Conley index is invariant by homotopical dynamical cancellation. Several illustrative examples of collisions and  homotopical dynamical cancellations are provided in this section. Furthermore, this approach ensures that a finite sequence of homotopical dynamical cancellations, preserves the homotopy type of the singular manifold in which they occur, as illustrated in Figure \ref{fig_intr_02}.

\begin{figure}[!ht]
\centering
\begin{overpic}[unit=1mm, scale=.25, width=9cm]{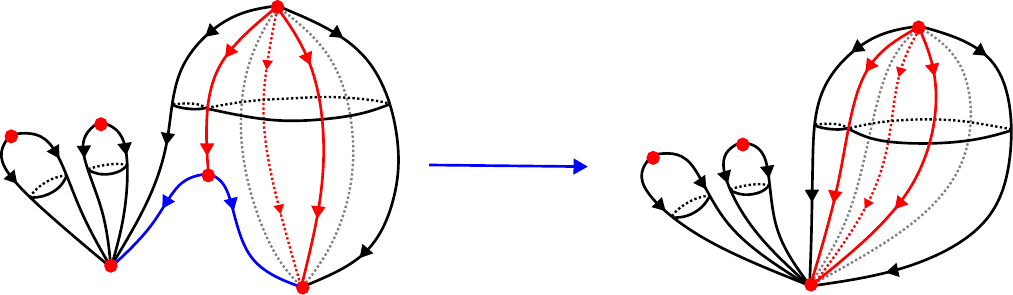}
\put(0,17.5){$x_1$}
\put(8,18.5){$x_2$}
\put(26,30){$x_3$}
\put(18,9){$x_4$}
\put(9.5,0){$x_5$}
\put(28,-2.5){$x_6$}
\put(62.5,15){$x_1$}
\put(71,16.5){$x_2$}
\put(89,28){$x_3$}
\put(78,-2.5){$x'_4$}
\end{overpic}
\caption{Homotopical dynamical cancellation between $x_4$, $x_5$, $x_6$ and the flow lines connecting them results in the new singularity $x'_4$.}
\label{fig_intr_02}
\end{figure}
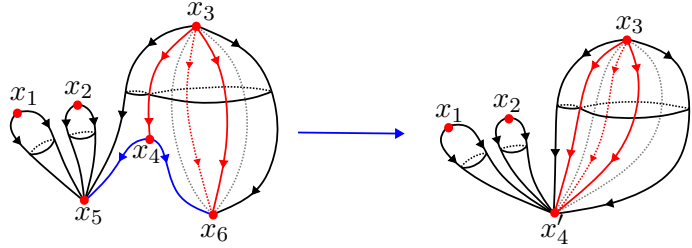

In Section \ref{sec:GS}, we introduce the class of Generalized Gutierrez–Sotomayor (GGS) manifolds and their associated flows. The singular nature of these spaces presents substantial difficulties, as we no longer operate within a differentiable framework, unlike in the case of smooth manifolds. In this context, it becomes essential to carefully examine the structure and {\it nature} of GGS singularities. 

A major challenge lies in defining the intersection number for GGS singularities. In order to achieve this objective we developed in Section \ref{sec:morsification} one of the main theoretical results, namely the existence of a morsification for a GGS flow. Our primary contribution in this section is Theorem \ref{teo:morsificacao}, which addresses the challenges arising from the non-smoothness of GGS flows. To this end, we develop a method that associates to each GGS flow a smooth manifold endowed with a Morse–Smale flow, while preserving the essential dynamical features of the original system. A crucial trade-off in this construction is the unavoidable loss of the one-to-one correspondence between the singularities of the GGS flow and those of the Morse–Smale flow. The proof of Theorem \ref{teo:morsificacao} relies not only on a careful local analysis of each singularity but also on a study of the global dynamical behavior, particularly the connections established through the folds.

In Section \ref{sec:complex}, we define a GGS chain complex by establishing its generators as well as the corresponding differential. Presenting a chain complex in the context of this paper is neither simple nor straightforward. Beyond the previously discussed issue of non-smoothness, one must contend with a range of local singularity types. In the  proof of Theorem \ref{teo:complexo}, we  demonstrate that the composition of the boundary map defined in Definition \ref{def:GS_chain} is always null. This is achieved by analyzing all possible local singularities and employing algebraic manipulations to express the composition of our boundary map in terms of the boundary map of the associated Morse-Smale chain complex. The proof  highlights our contribution to the field and underscores the inherent challenges of establishing results within the GGS setting. Theorem \ref{teo_matriz} provides a more precise characterization of the boundary maps of our chain complex, showing that they can be completely represented by a single matrix whose entries belong to the set $\{ -1, 0, 1\}$, along with a clear description of the matrix’s rows and columns.

In Section \ref{sec:detecting_dynamical_homotopical_cancellation}, we examine the local algebraic effects arising from homotopical dynamical cancellations. 
We show that the chain complex is capable of detecting cancellations between generalized singularities of arbitrary type. Specifically, by “detecting,” we mean that if the intersection number of singularities $x_1$ and $x_2$ is nonzero, then there exists another singularity $x_3$ such that a cancellation involving $x_1$, $x_2$ and $x_3$ is possible, as established in Theorem \ref{teo_cancelamento}.

In Section \ref{sec:SSAproach_cancellation}, we present a global theorem illustrating the relevance of the chain complex introduced herein, which involves a spectral sequence analysis of a filtered GGS chain complex. Our primary contribution in this section is establish a bijection between the algebraic cancellations in this spectral sequence and the homotopical dynamical cancellations, these results are stated precisely in Theorem \ref{teo_bijecao_sequencia}. This section also  contains examples that demonstrate the effectiveness of the proposed framework.

Finally, our concluding remarks highlight relevant questions and directions for future research prompted by this work.

\section{Collisions of Isolated Invariant Sets}
\label{sec:collisions}
In this section, we introduce the definition of a collision of singularities, initially inspired by the physical phenomenon in which two black holes merge, causing their central singularities to combine into a single, more massive singularity. To establish a formal mathematical framework for this concept, we begin by outlining the general setting in which our discussion takes place.

Let $M$ be a Hausdorff topological space and $\varphi$ be a continuous flow on $M$, that is, a continuous map $\varphi:\R\times M\rightarrow M$ such that:
\begin{enumerate}
    \item[$(i)$] $\varphi(0,x)=x$ for all $x\in M$;
    \item[$(ii)$] $\varphi(s+t,x)=\varphi(s, \varphi(t,x))$ for all $s, t\in \R$ and $x \in M$. 
\end{enumerate}
A compact subset $N\subset M$ is an \textit{isolating neighborhood} if
\begin{align*}
\operatorname{Inv}(N) \coloneqq \{x \in N \mid \varphi(t,x) \in N, \  \forall t \in \mathbb{R}\} \subset \operatorname{int}(N),
\end{align*}
where $\operatorname{int}(N)$ denotes the interior of $N$. A subset $S \subset M$ is called an \textit{isolated invariant set} if $S=\operatorname{Inv}(N)$ for some isolating neighborhood $N$.

\begin{definition}
\label{def_collisions}
Let $S = \operatorname{Inv}(N)$ be an isolated invariant set with respect to a flow $\varphi$ on $M$. We say that $S'$ is a \textit{collision} of $S$ if there exist a compact set $N'$, a topological space $M'$ and a continuous flow $\varphi'$ on $M'$ with  $S' = \operatorname{Inv}_{\varphi'}(N')\subset \operatorname{int}(N')$ satisfying the following conditions:
\begin{enumerate}
    \item[$(i)$] $N'$ is obtained as a quotient space of $N$, i.e., $N' = N/\sim$ for some equivalence relation $\sim$ on $N$ and $\partial N'$ is homeomorphic to $\partial N$;
    \item[$(ii)$] $M'$ is the adjunction space $(M \setminus \operatorname{int}(N)) \cup_f N'$, where the attaching map $f: \partial N \to \partial N'$ is a homeomorphism;
    \item[$(iii)$] $\varphi'$ extends $\varphi$ from the complement of $N$, i.e., $\varphi' $ and $\varphi$ coincides on $M\setminus N$.
\end{enumerate}
\end{definition}

Since the collision process alters the initial topological space, the following remark clarifies the topology of the resulting space $M'$.

\begin{remark}
\label{remark_topology}
We now clarify the topologies used in the construction of $M'$ from Definition \ref{def_collisions}. The sets $N$ and $M \setminus \operatorname{int}(N)$ are endowed with the subspace topology inherited from $M$. The set $N'$ is given the quotient topology induced by the natural projection $\pi: N \to N' = N/\sim$. By item $(ii)$ of the definition, there is a homeomorphism between the boundaries, which we denote by $f: \partial N \to \partial N'$. The space $M'$ is the adjunction space $M' = (M \setminus \operatorname{int}(N)) \cup_f N'$, equipped with the standard final topology determined by this gluing construction.
\end{remark}

To ensure that limits of convergent sequences in $M' $ are unique, the space must be Hausdorff. Since the quotient construction does not, in general, preserve this property, an additional condition is required. The following proposition provides a sufficient condition on the equivalence relation $\sim$ that guarantees the space $M'$ is also Hausdorff.

\begin{proposition}
\label{prop_hausdorff}
Let $\varphi$ be a continuous flow on a Hausdorff space $M$. Suppose $M'$ is the space obtained by a collision from $(M, \varphi)$, involving the quotient map $\pi: N \to N' = N/\sim$ from the compact set $N \subset M$. If the quotient map $\pi$ is open and the graph of the equivalence relation, $G_\sim = \{(p_1, p_2) \in N \times N \mid p_1 \sim p_2\}$, is a closed subset of the product space $N \times N$, then the resulting space $M'$ is also a Hausdorff space.
\end{proposition}

\begin{proof}
We need to show that under the hypotheses, the space 
\begin{align*}
    M'=(M \setminus \operatorname{int}(N)) \cup_{f} N'
\end{align*}
is Hausdorff. 

As established in Remark \ref{remark_topology}, the component spaces carry their standard topologies. Specifically, $M \setminus \operatorname{int}(N)$ and $N$ are endowed with the subspace topology inherited from $M$:
\begin{align*}
    \tau_{M \setminus \operatorname{int}(N)} &= \{ U\subset M \setminus \operatorname{int}(N) \mid U = (M \setminus \operatorname{int}(N)) \cap V \text{ for some open set } V\subset M \}; \\
    \tau_{N} &= \{ U\subset N \mid U= N\cap V\ \text{for some open set } V\subset M \}.
\end{align*}
We first confirm that the component spaces are Hausdorff. Let $x_1, x_2$ be two distinct points in $M \setminus \operatorname{int}(N)$. Since $M$ is Hausdorff, there exist disjoint open sets $V_1, V_2 \subset M$ such that $x_1 \in V_1$ and $x_2 \in V_2$. By the definition of the subspace topology, the sets $U_1 = V_1 \cap (M \setminus \operatorname{int}(N))$ and $U_2 = V_2 \cap (M \setminus \operatorname{int}(N))$ are open in $M \setminus \operatorname{int}(N)$. They contain $x_1$ and $x_2$, respectively, and are disjoint. Thus, $M \setminus \operatorname{int}(N)$ is Hausdorff. An identical argument shows that $N$ is also Hausdorff.

By hypothesis, the quotient map $\pi: N \to N'$ is open and the relation $\sim$ is a closed subset of $N \times N$. Therefore, by Corollary 3.58 of \cite{lee2000introduction}, the quotient space $N'=N/\sim$ is Hausdorff.

Finally, we show that the adjunction space $M' = (M \setminus \operatorname{int}(N)) \cup_f N'$ is Hausdorff. We have established that the component spaces $M \setminus \operatorname{int}(N)$ and $N'$ are both Hausdorff. The space $M'$ is formed by attaching these two spaces along the boundary $\partial N$. Since $N$ is a compact subset of a Hausdorff space, its boundary $\partial N$ is a closed subset of $M \setminus \operatorname{int}(N)$. Because we are attaching two Hausdorff spaces along a closed subset, the resulting adjunction space $M'$ is also Hausdorff.
\end{proof}

A broad class of operations on a dynamical system $(M, \varphi)$ can be regarded as collisions.

\begin{example}
\label{exemplo_colisoes_s2}
As an example, consider a surface $M$ homeomorphic to the $2$-sphere $\es^2$, with a continuous flow $\varphi$ generated by the negative gradient of a height function, $-\nabla f$, as depicted in Figure \ref{fig_collisions_01}.

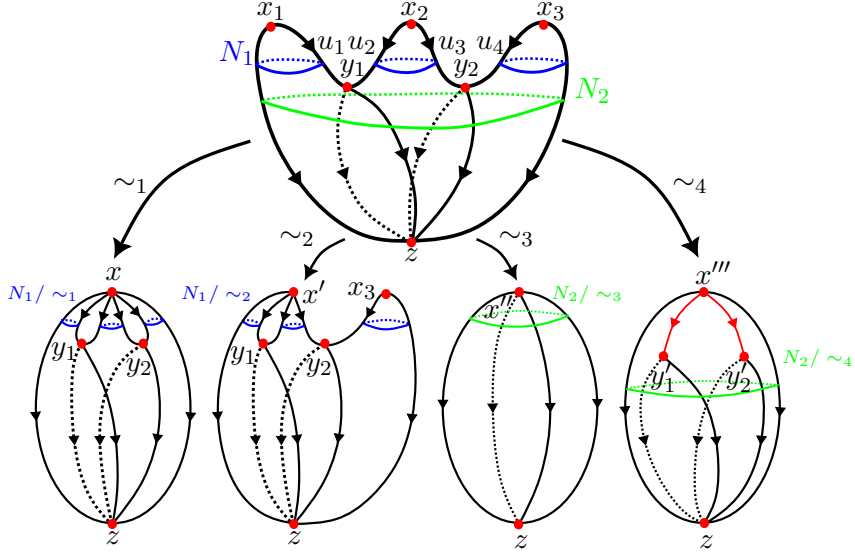
\begin{figure}[!ht]
\centering
\begin{overpic}[unit=1mm, scale=.25, width=10cm]{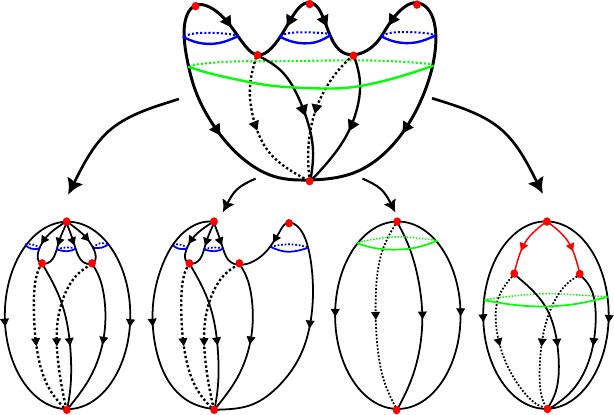}
\put(25,62){$\textcolor{blue}{N_1}$}
\put(72,57){$\textcolor{green}{N_2}$}
\put(38,63.5){$u_1$}
\put(42,63.5){$u_2$}
\put(54,63.5){$u_3$}
\put(59,63.5){$u_4$}
\put(30,68){$x_1$}
\put(49,68){$x_2$}
\put(67,68){$x_3$}
\put(41,60.5){$y_1$}
\put(56,60.5){$y_2$}
\put(49.5,35.5){$z$}
\put(11,45){$\sim_1$}
\put(33,38){$\sim_2$}
\put(62,38){$\sim_3$}
\put(85,45){$\sim_4$}
%%%%%%%%%%%%%%%%%%%%
\put(10,33){$x$}
\put(2.8,23){$y_1$}
\put(12.7,21){$y_2$}
\put(10,-2){$z$}
\put(-3,30.5){\tiny{$\textcolor{blue}{N_1/\sim_1}$}}
%%%%%%%%%%%%%%%%%%%%
\put(36,30){$x'$}
\put(42,31){$x_3$}
\put(26.5,22.5){$y_1$}
\put(36.5,21){$y_2$}
\put(34,-2){$z$}
\put(20,30.5){\tiny{$\textcolor{blue}{N_1/\sim_2}$}}
%%%%%%%%%%%%%%%%%%%%
\put(60,28){$x''$}
\put(64,-2.5){$z$}
\put(69,30.5){\tiny{$\textcolor{green}{N_2/\sim_3}$}}
%%%%%%%%%%%%%%%%%%%%
\put(88,32.5){$x'''$}
\put(82,19.2){$y_1'$}
\put(92,19.2){$y_2'$}
\put(88.5,-2.5){$z$}
\put(99.5,22){\tiny{$\textcolor{green}{N_2/\sim_4}$}}
%%%%%%%%%%%%%%%%%%%%
\end{overpic}
\caption{Examples of different collisions in $(M,\varphi)$.}
\label{fig_collisions_01}
\end{figure}

This flow has three repelling fixed points, $x_1$, $x_2$ and $x_3$. Let $N_1 \subseteq M$ be a compact isolating neighborhood composed of three disjoint components, one enclosing $x_1$, other $x_2$ and other $x_3$. Thus, the isolated invariant set is  $S = \operatorname{Inv}_{\varphi}(N_1) = \{x_1, x_2, x_3\}$. We can perform a collision by defining an equivalence relation $\sim_1$ on $N_1$ that identifies the three points, $x_1, x_2, x_3$. This operation collapses the set $S = \{x_1, x_2, x_3\}$ into a new invariant set $S' = \{x\}$ in the resulting space.

Consider the same compact set $N_1$ from the preceding paragraph, whose invariant set $\operatorname{Inv}_{\varphi}(N_1)=\{ x_1, x_2, x_3\}=S$. Let $\sim_2$ be the equivalence relation on $N_1$ that identifies $x_1$ and $x_2$. In the quotient space $N_1/ \sim_2$ we define a repelling flow $\varphi '$ whose invariant set is $\operatorname{Inv}_{\varphi '}(N_1/ \sim_2)=\{x', x_3\}$ and is a continuous extension of $\varphi$ from the complement of $N_1$.

Alternatively, consider an isolating neighborhood $N_2$ whose invariant set $S = \operatorname{Inv}_{\varphi}(N_2)$ is the connected set consisting of the repelling points $x_1$, $x_2$ and $x_3$, the saddle points $y_1$, $y_2$, and the trajectories $u_1$, $u_2$, $u_3$ and $u_4$ connecting them. On this set, we can illustrate three different types of collisions as follows.
\begin{enumerate}
    \item One can perform  a \textit{trivial collision}. In other words, we consider the trivial equivalence relation, $N=N'$ and $ \varphi ' = \varphi$.   
    \item  One can perform a \textit{total collision}, that is, a collision defined by an equivalence relation $\sim_3$ on $N_2$ that identifies all points of $S$. This collapses  $S$ to a single repelling fixed point $x''$ in the new space, so that $$S'' = \operatorname{Inv}_{\varphi''}(N_2/\sim_3) = \{x''\}.$$
    \item A further collision can be achieved by identifying entire trajectories. We define an equivalence relation $\sim_4$ on $N_2$ that identifies the trajectory $u_1$ with $u_2$ and, separately, $u_3$ with $u_4$. This procedure effectively ``glues''  each corresponding pairs of flow lines. Consequently, in the quotient space $N_2/\sim_4$, the invariant set $S^{'''}$ is composed of a repelling point $x'''$, two saddle points $y'_1$ and $y'_2$, and the two newly formed trajectories connecting them.
\end{enumerate}
\end{example}

Under the same hypotheses and notations as in Definition \ref{def_collisions}, we now define a specific type of collision based on the topological concept of a deformation retraction. 

\begin{definition}
\label{definition_canc}
An invariant set $S'$ is called a \textit{homotopical dynamical cancellation} of $S$ if it can be obtained through a collision of $S$, where the equivalence relation $\sim$ on $N$ is induced by a deformation retraction of $S$ onto one of its subsets, $A \subset S$.
\end{definition}

In Example~\ref{exemplo_colisoes_s2}, only trivial and total collisions in $N_2$ can be interpreted as a homotopical dynamical cancellation.

\begin{example}
We illustrate the concepts with two examples in $\mathbb{R}^2$, as depicted in Figure \ref{fig_collisions_monkey}.

\begin{enumerate}
    \item An isolated invariant set $S_1 = \operatorname{Inv}(N_1)$ consists of two hyperbolic saddles, $x_1$ and $x_2$, and a flow line $u$ connecting them.
    \item An isolated invariant set $S_2 = \operatorname{Inv}(N_2)$ consists of two hyperbolic saddles, $x_3$ and $x_4$, with no connecting orbit between them.
\end{enumerate}

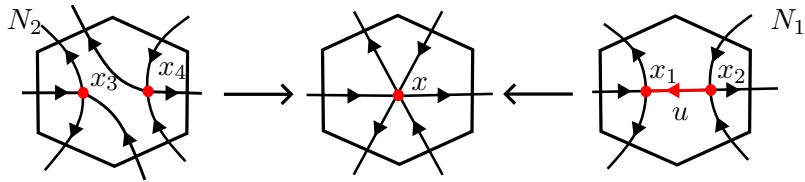
\begin{figure}[!ht]
\centering
\begin{overpic}[unit=1mm, scale=.25, width=10cm]{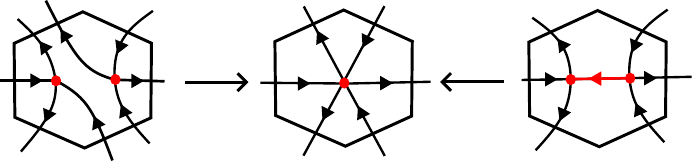}
\put(83,13.5){$x_1$}
\put(92,13.5){$x_2$}
\put(99,20){$N_1$}
\put(86,8){$u$}
%%%%%%%%%%%%%%%%%%%%%%
\put(51.5,12){$x$}
%%%%%%%%%%%%%%%%%%%%%%
\put(18,14){$x_4$}
\put(-2,20){$N_2$}
\put(9, 12.5){$x_3$}
\end{overpic}
\caption{Collisions producing a monkey saddle.}
\label{fig_collisions_monkey}
\end{figure}

In the first case, we can apply a homotopical dynamical cancellation. This is achieved by an equivalence relation that corresponds to a deformation retraction of the entire invariant set $S_1$ onto a single point. The resulting collision collapses $S_1$ into a new invariant set $S'=\{x'\}$, which may be a degenerate saddle (sometimes called a ``monkey saddle'').

In the second case, a different type of collision can be used to create the same monkey saddle. This is achieved by an equivalence relation on $N_2$ that is induced by the deformation retraction of a $2$-dimensional region between the saddles onto a $1$-dimensional arc. This operation effectively glues the saddles together, producing a new invariant set $S''$ that is topologically equivalent to $S'$.

Note that this second operation is not a homotopical dynamical cancellation, because the deformation retraction occurs on a subset of $N_2 $, not on the invariant set $S_2$ itself.
\end{example}

\begin{example}
The classical Neimark-Sacker bifurcation occurs in a one-parameter family of maps $x \mapsto F(x; \rho)$, where $x \in \mathbb{R}^2$ is the state variable, $\rho \in \mathbb{R}$ is the bifurcation parameter, and $x=0$ is a fixed point. Under certain non-degeneracy conditions, this bifurcation is characterized by the creation of a unique, invariant closed curve that emerges from the fixed point. For a detailed treatment, see \cite{Murilo_R} and Theorem 4.6 of \cite{kuznetsov1998elements}.

The transition between the two states of the Neimark-Sacker bifurcation can be interpreted a collision in two distinct ways, as illustrated in Figure \ref{fig_collisions_sacker}.

\begin{figure}[!ht]
\centering
\begin{overpic}[unit=1mm, scale=.25, width=8cm]{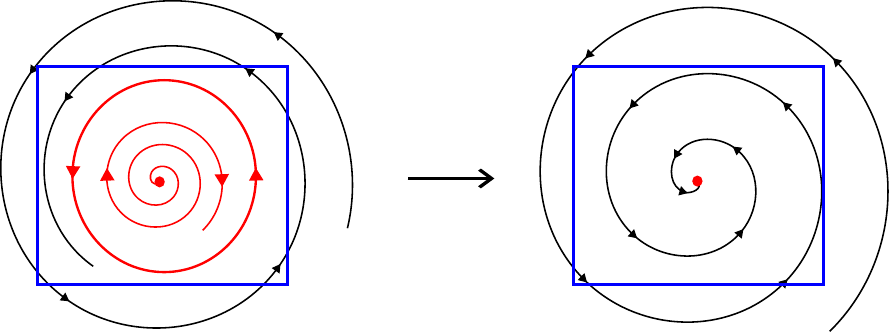}
\put(29,12){$\gamma$}
\put(17,18.65){$x$}
%%%%%%%%%%%%%%%%%%%%%%%%%%%
\put(77,18){$x$}
\put(0,30){$\textcolor{blue}{N}$}
\put(60,30){$\textcolor{blue}{N}$}
\end{overpic}
\caption{Collision in Neimark-Sacker bifurcation.}
\label{fig_collisions_sacker}
\end{figure}
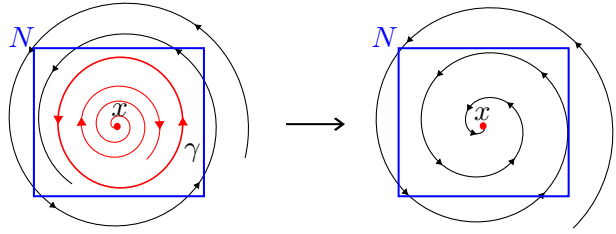

First, it can be viewed as a homotopical dynamical cancellation acting on the invariant set $S$, composed of the periodic orbit $\gamma$ and all bounded trajectories encapsulated by $\gamma$ (topologically, $S$ is a $2$-dimensional closed disk whose boundary is  $\gamma$). In this interpretation, one can realize a deformation retraction of the disk, collapsing the entire disk, and the periodic orbit $\gamma$ along with it, onto the fixed point $x=0$. 

The second interpretation also regards it as a homotopical dynamical cancellation, but here the isolating neighborhood $N$ is kept fixed while the flow $\varphi$ inside $N$ is  continuously perturbed.
\end{example}

The class of collisions arising  from homotopical dynamical cancellations has a 
special property: the Conley index remains invariant under such collisions. Originally introduced as a generalization of the Morse index, the Conley index has the notable advantage of being well-defined for any continuous topological dynamical system. This makes it particularly suitable for studying certain dynamical behaviors in topological spaces that may exhibit singularities. In what follows, we introduce only the notations and definitions necessary for our purposes; for a full exposition, see \cite{livro_coloquio} and \cite{Conley}.

Let $S\subseteq M$ be an isolated invariant set. A pair of compact sets $(L_1,L_2)$ is an \textit{index pair} of $S$ if $L_2\subseteq L_1$ and: 
\begin{enumerate}
    \item[$(i)$] $\overline{L_1 \setminus L_2}$ is an isolating neighborhood for $S$ in $M$; 
    \item[$(ii)$] $L_2$ is \textit{positively invariant} in $L_1$, that is, if $x\in L_2$ and $\varphi([0,T],x)\subseteq L_1$ then $\varphi([0,T],x)\subseteq L_2$; 
    \item[$(iii)$] $L_2$ is the \textit{exit set} of the flow in $L_1$, that is, if $x\in L_1$ and $\varphi([0,\infty], x)\not\subseteq L_1$ then there exists $T> 0$ such that $\varphi([0,T],x)\subset L_1$ and $\varphi(T,x)\in L_2$. 
\end{enumerate}
The \textit{homotopy Conley index} of an isolated invariant set $S$ is defined by: 
\begin{align*}
    h(S)\coloneqq [L_1/L_2, *],
\end{align*}
that is, the homotopy type of the pointed space where: $(L_1,L_2)$ is an index pair for $S$ and $*$ is the equivalent class of $L_2$. The \textit{homology Conley index} of $S$ is defined as $CH_*(S)\coloneqq\widetilde{H}_*(h(S))$, where $\widetilde{H}_*$ denotes the reduced singular homology over $\Z$. 

In \cite{Conley}, the existence of an index pair for any isolated invariant set is established. Moreover, it is shown that the homotopy type of the pointed space $(L_1/L_2, *)$ does not depend on the  chosen index pair $(L_1,L_2)$ of $S$. Hence, the Conley index is well-defined. 

\begin{proposition}
\label{prop_conley}
Consider $M$ a Hausdorff space and $\varphi$ be a continuous flow on $M$. Let $S=\operatorname{Inv}_{\varphi}(\overline{L_1\setminus L_2})$, where $(L_1,L_2)$ is an index pair for $S$. If $S'$ is obtained from $S$ through a homotopical dynamical cancellation and $M'$ is a Hausdorff space, then  $h(S)=h(S')$.
\end{proposition}

\begin{proof}
\begin{comment}
Let $(L_1,L_2)$ be an index pair for $S$. By definition, we have
\begin{align*}
    S = \operatorname{Inv}_{\varphi}\left( \overline{ L_1\setminus L_2}\right).
\end{align*}    
\end{comment}
By hypothesis, $S'$ is obtained through a homotopical dynamical cancellation of $S$. This means that the collision is obtained via a deformation retract. Let $\sim$ be the equivalence relation defining this retract. The new isolating neighborhood is then given by
\begin{align*}
    N' = \overline{L_1\setminus L_2} / \sim.
\end{align*}
The new dynamics take place in the space $M' = \left( M \setminus \operatorname{int}{(\overline{L_1\setminus L_2} )} \right) \cup_{f} N'$, where we have a topological dynamical system $\varphi' : \mathbb{R}\times M' \longrightarrow M'$ such that:
\begin{align*}
    S' = \operatorname{Inv}_{\varphi'}(N').
\end{align*}
We claim that the compact pair $(N' \cup L_2, L_2)$ is an index pair for $S'$ in the space $M'$.  Indeed, we verify the condition for the invariant set:
\begin{align*}
    \operatorname{Inv}_{\varphi'} \left( \overline{(N' \cup L_2) \setminus L_2} \right) &= \operatorname{Inv}_{\varphi'} (\overline{N'}) \\
    &= \operatorname{Inv}_{\varphi'} (N') \\
    &= S'.
\end{align*}
Since the flow $\varphi'$ is continuous and extends the original flow $\varphi$ on $\overline{L_1\setminus L_2}$, it follows that $L_2$ is positively invariant in $N' \cup L_2$. Furthermore, $L_2$ serves as the exit set for the flow $\varphi'$ with respect to the neighborhood $N' \cup L_2$.

A fundamental result from \cite{Conley} states that the Conley index is well-defined, meaning it is independent of the choice of index pair for a given isolated invariant set. Since $N' = \overline{L_1 \setminus L_2} / \sim$ is obtained via a deformation retract, then we have the homotopy equivalent: 
\begin{align*}
    (N' \cup L_2) \simeq L_1.
\end{align*}
This leads to the following:
\begin{align*}
    h(S') &= \left[ (N' \cup L_2)/L_2, * \right] \\
    &= \left[ L_1/L_2, * \right] \\
    &= h(S).
\end{align*}
Therefore, $h(S')=h(S)$.
\end{proof}

Returning to Example \ref{exemplo_colisoes_s2}, one can use Proposition \ref{prop_conley} to guarantee  that the collision $N_1/\sim_1$ is not a homotopical dynamical cancellation, since $h(S) \neq h(S')$. 
In fact, it suffices to take the index pair $(N_1, \partial N_1)$ for $S = \{x_1, x_2, x_3\}$ from which
\begin{align*}
    h(S)=[N_1/ \partial N_1 , *] \simeq \es^2 \vee \es^2 \vee \es^2 .
\end{align*}
On the other hand, we take $(N'_1, \partial N'_1)$ for $S'=\{x\}$ and we have
\begin{align*}
    h(S' )= [N'_1/ \partial N'_1 , *] \simeq \es^2 \vee \es^2\vee \es^2 \vee \es^1 \vee \es^1.
\end{align*}

 Proposition \ref{prop_conley} establishes a sufficient, but not necessary condition.
%in order to have $h(S)=h(S')$. 
To see this, consider the collision in Example \ref{exemplo_colisoes_s2} given in $N_2$ by the equivalence relation $\sim_4$. We can take the index pair $(N_2, \partial N_2)$ for $S=\{ x_1, x_2, x_3, u_1, u_2, u_3, u_4, y_1, y_2 \}$ and we have
\begin{align*}
   h(S)=[N_2/ \partial N_2 , *] \simeq \es^2 
\end{align*}
For $S'''$, we take the index pair $(N^{'''}_2, \partial N^{'''}_2)$, and we have
\begin{align*}
    h(S^{'''})=[N^{'''}_2/ \partial N^{'''}_2 , *] \simeq \es^2. 
\end{align*}
Therefore, $h(S)=h(S^{'''})$ and $S^{'''}$ is not a homotopical dynamical cancellation of $S$.
%%%%%%%%%%%%%%%%%%%%%%%%%%%%%%%%%%%%%%%%%%%%%%%%%%%%%%%%%%%%%%%%%%%%%%%%%%%%%%%%%%%%%%%%%%%%%%%%%%%%%%%%%%%%%%%%%%%%%%%%%%%%%%%%%%%
\section{Generalized Gutierrez-Sotomayor Class}
\label{sec:GS}

In this paper, we introduce the class  \textit{Generalized Gutierrez-Sotomayor} or \textit{$\mathcal{GGS}$}, for short, which generalize the class introduced by Gutierrez and Sotomayor in \cite{GS}. This class introduced by them is composed of singular surfaces, originally called \textit{two-dimensional manifolds with simple singularities}, are subsets $M\subset \mathbb{R}^{\ell}$ that are locally diffeomorphic to one of the following sets in $\mathbb{R}^3$:
\begin{align*}
    \mathcal{R}&=\{(x,y,z)\in \R^3\  |\  z=0\}, \text{\textit{ plane}};\\
    \mathcal{C}&=\{(x,y,z)\in \R^3\  |\  z^2-y^2-x^2=0\}, \text{\textit{ cone}};\\
    \mathcal{W}&=\{(x,y,z)\in \R^3\  |\  zx^2-y^2=0\}, \text{\textit{ Whitney's umbrella}};\\
    \mathcal{D}&=\{(x,y,z)\in \R^3\  |\  xy=0\}, \text{\textit{ double crossing}};\\
    \mathcal{T}&=\{(x,y,z)\in \R^3\  |\  xyz=0\}, \text{\textit{ triple crossing}}.
\end{align*}
For this class, there is a natural stratification. We say that a vector field $X$ of class $C^r$  on $\mathbb{R}^{\ell}$ is tangent to $M$ if it is tangent to each stratum. Among all such vector fields, we consider those that satisfy:
\begin{enumerate}
    \item[($i)$] $X$ has a finite number of singularities and periodic orbits, all of which are hyperbolic;
    \item[$(ii)$] The singular limit cycles of $X$ are simple, and $X$ has no saddle connections;
    \item[$(iii)$] The $\alpha$- and $\omega$-limit sets of every trajectory of $X$ is a singularity, a periodic orbit, or a singular limit cycle.
\end{enumerate}
For this class, the terminology \textit{GS manifold} and \textit{GS vector field} was introduced in \cite{Montufar} to honor the contributions of Gutierrez and Sotomayor. Accordingly, the flows induced by such vector fields are called \textit{GS flows}, and their singularities\footnote{A \textit{singularity} of a vector field $X$ is a point $p \in M$ such that $X(p)=0$.} are known as \textit{GS singularities}.

In this section, we formally define the class of $\mathcal{GGS}$ and discuss its fundamental properties. More generally, using Definition \ref{definition_canc}, we can define a  $\mathcal{GGS}$ pair as follows:

\begin{definition}
\label{def_GGS_pair}
A pair $(M', \varphi')$, where $M'$ is a singular $2$-manifold  and $\varphi'$ is a continuous flow on $M'$, is said to be a \textit{Generalized Gutierrez-Sotomayor pair}, or a \textit{$\mathcal{GGS}$ pair} for short, if it can be obtained from  a GS flow on a GS manifold $(M, \varphi)$ through a finite number of homotopical dynamical cancellations. In this case, we call $M'$ a \textit{$\mathcal{GGS}$ manifold} and $\varphi'$ a \textit{$\mathcal{GGS}$ flow}.
\end{definition}

The generality of Definition \ref{def_GGS_pair} makes it initially too complex for a direct analysis. Thus, in this work, we restrict our attention to a restricted class of $\mathcal{GGS}$ pairs. Within this subclass, we obtain topological and dynamical ``control'' through local charts and continuous flows induced by vector fields. Our main interest lies in homotopical dynamical cancellations along  flow lines in the smooth part, specifically, those involving three singularities: either two repelling and one saddle, or two attracting and one saddle. 

First, we introduce singular regions, which play a fundamental role in the subclass of $\mathcal{GGS}$ pairs that we analyze in this paper.

\begin{definition}\label{def:GSregions}
Consider $n$ disjoint unit disks $D_i\subset \R^2$  with center points $p_i$, where $i=1,\ldots, n$.  Define the following \textit{basic singular regions}:
    \begin{enumerate}
        \item[$(i)$] The $n$\textit{-sheet cone} $\mathbf{\mathcal{C}_n}$ region is the quotient space obtained by identifying the center points $p_i$ of  $n$ disks $D_i$ to a common point $p$;        
        \item[$(ii)$] The $n$\textit{-sheet cross-cap} $\mathbf{\mathcal{W}_n}$ region is the quotient space obtained by identifying two radii of $D_i$, where $i=2,\ldots, n-1$, with a radius of $D_{i-1}$ and a radius of $D_{i+1}$;
        \item[$(iii)$] The $n$\textit{-sheet double crossing} $\mathbf{\mathcal{D}_n}$ region, with $n\geq 2$, is the quotient space obtained by identifying a diameter $d_i$ of each $D_i$, where $i=2,\ldots, n$ to different diameters of $D_1$. That is,   $D_i \cap D_j = \{p\}$ and $D_i \cap D_1 = d_i$ for all $i\neq j$;
        \item[$(iv)$] The $n$\textit{-sheet triple crossing} $\mathcal{T}_n$ region, where $n=2k+1\geq 3$, is the quotient space obtained as follows: consider $n$ disks denoted by $D_0$, $D_i^{1}$ and $D_i^{2}$, where $i=1,\ldots, k$,  and fix different  diameters: $\{d^1_{0,i}, d^2_{0,i}\}$ in $D_0$, $\{d^1_i, \partial^1_i\}$ in $D_i^{1}$ and $\{d_i^2, \partial^2_i\}$ in $D_i^{2}$, where $i=1,\ldots, k$. We identify the diameters $\partial_i^1$ with $\partial_i^2$, the diameters $d_i^1$ with $d_{0,i}^1$, and the diameters $d_i^2$ with $d_{0,i}^2$ such that $(D_i^{1}\cup D_i^{2})\cap (D_j^{1}\cup D_j^{2})=\{p\}$, if $i\neq j$ and $i,j\in \{1,\cdots, k\}$.
    \end{enumerate}
\end{definition}

The singular regions from Definition \ref{def:GSregions}, were introduced in \cite{2021homotopical},  obtained by considering homotopical dynamical cancellations between smooth flow lines and GS singularities of the same type. Figure \ref{fig_regioes} exemplifies each of the basic singular regions of Definition \ref{def:GSregions}. Note that if $n=2$, the singular regions $\mathcal{C}_2$, $\mathcal{W}_2$, $\mathcal{D}_2$ are homeomorphic to $\mathcal{C}$, $\mathcal{W}$, $\mathcal{D}$, respectively. If $k=1$, the singular region $\mathcal{T}_{3}$ is homeomorphic to $\mathcal{T}$. That is, these more general singular regions recover those originally presented in \cite{GS}. Since this work aims to study all homotopical dynamical cancellations between smooth flow lines and singularities of any type, new singular regions are required, which we present in Definition \ref{def:GSmix}.

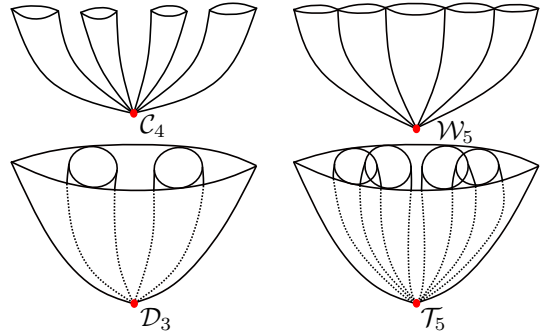
\begin{figure}[!ht]
\centering
\begin{overpic}[unit=1mm, scale=.25, width=7cm]{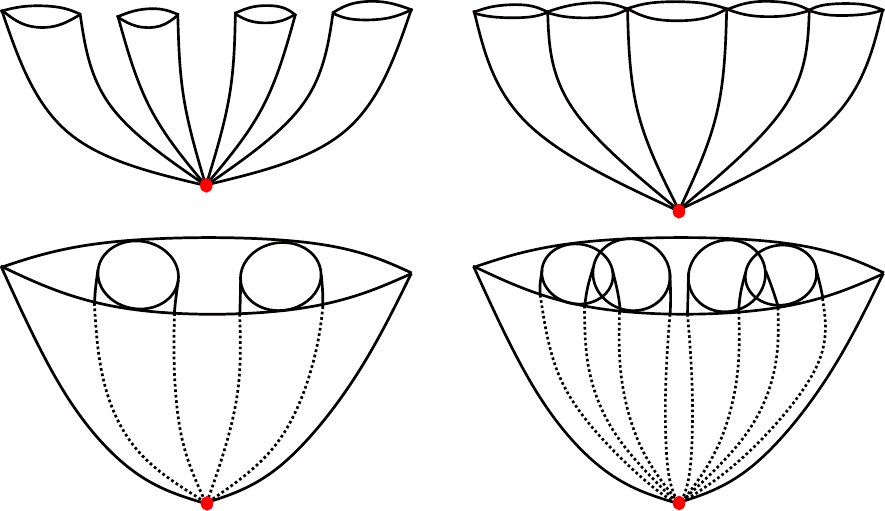}
\put(24.5,33){\small{$\mathcal{C}_4$}}
\put(24.5,-3.5){\small{$\mathcal{D}_3$}}
\put(80.5,32){\small{$\mathcal{W}_5$}}
\put(77.5,-3.5){\small{$\mathcal{T}_5$}}
\end{overpic}
%\vspace{0.2cm}
\caption{Basic Singular Regions.}
\label{fig_regioes}
\end{figure}

\begin{definition}
\label{def:GSmix}
Given $n,m, k, k' \in \mathbb{N}$ with $n,m\geq 2$ and $k, k' \geq 1$, define the following \textit{mixed singular regions}:
\begin{enumerate}
\item[$(i)$] For any basic singular regions $\mathcal{P}$ and $\mathcal{Q}$, let $\mathcal{P}\vee \mathcal{Q}$ denote the singular region formed by the wedge sum of a $\mathcal{P}$ region and a $\mathcal{Q}$ region at their distinguished singular point;  
\item[$(ii)$] The singular region $\mathcal{W}_{n}\mathcal{Q}$, where $\mathcal{Q}=\mathcal{D}_m, \mathcal{T}_{2k+1}$, is obtained  by identifying the  disc $D_1$ of a $\mathcal{D}_{m}$-region (resp., $D_0$ disc of a $\mathcal{T}_{2k+1}$-region) with the disc $D_n$ of the $n$-sheet cross-cap region;
\item[$(iii)$] The singular region ${\mathcal{D}_{n}\mathcal{T}_{2k+1}}$ is obtained by identifying the disc $D_1$ of a $\mathcal{D}_{n}$-region with the disc $D_0$ of a $\mathcal{T}_{2k+1}$-region.
\end{enumerate}
\end{definition}

To simplify notation, particularly in the proof of Proposition \ref{prop_herdeiros}, we adopt the notation $\mathcal{R}\mathcal{Q}=\mathcal{Q}$, where $\mathcal{Q}$ is any singular region and $\mathcal{R}$ is the regular region. Figure \ref{fig_mix} illustrates several examples of mixed singular regions formed through this process. The constructions described above can be applied inductively, as shown in the mixed singular regions $\mathcal{C}_2\vee\mathcal{W}_2\mathcal{D}_2$ and $\mathcal{T}_5\mathcal{W}_2\mathcal{D}_2$ (see Figure \ref{fig_mix}). In fact, such mixed singular regions naturally arise through the process of homotopical dynamical  cancellation (see Section \ref{sec:detecting_dynamical_homotopical_cancellation}). 
Throughout this paper, to simplify, we also refer to singular regions formed by a finite number of applications of the processes defined in Definition \ref{def:GSmix} as \textit{mixed singular regions}.

\begin{figure}[!ht]
\centering
\begin{overpic}[ unit=1mm, scale=.25, width=8cm]{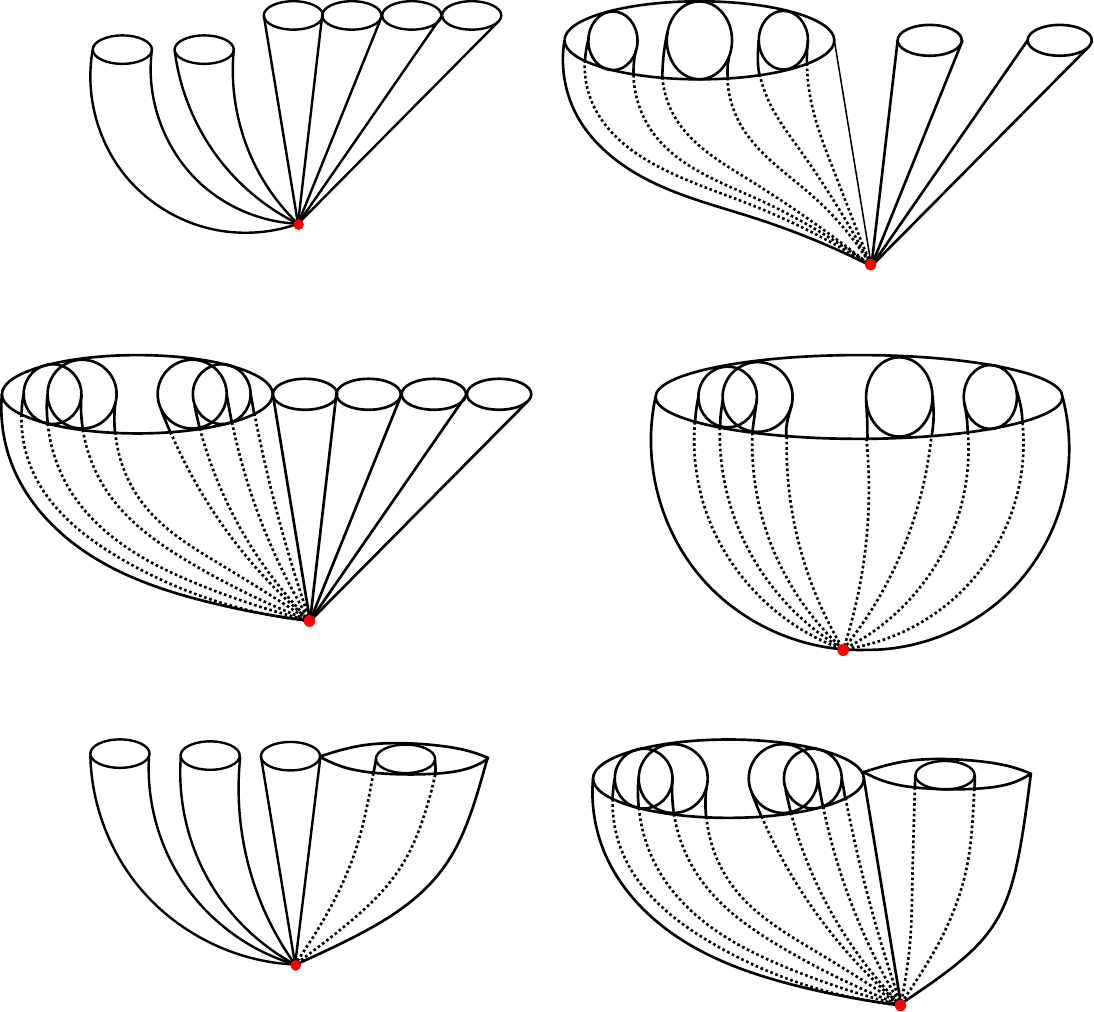}
\put(20.5,65.5
){\small{$\mathcal{C}_2\vee\mathcal{W}_4$}}
\put(24,30){\small{$\mathcal{T}_5\mathcal{W}_5$}}
\put(18,-1){\small{$\mathcal{C}_2\vee\mathcal{W}_2\mathcal{D}_2$}}
\put(73,63){\small{$\mathcal{D}_4\vee\mathcal{C}_2$}}
\put(72.5,28){\small{$\mathcal{T}_3\mathcal{D}_3$}}
\put(75,-4){\small{$\mathcal{T}_5\mathcal{W}_2\mathcal{D}_2$}}
\end{overpic}
%\vspace{0.4cm}
\caption{Mixed singular regions.}
\label{fig_mix}
\end{figure}
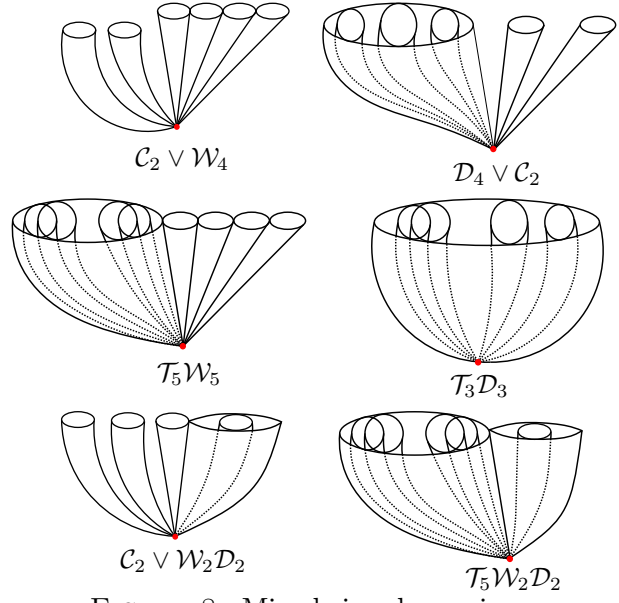

In the next definition, we adopt the notion of maps of class $C^r$ from \cite{milnor1965topology}. That is, a map $f:K\rightarrow \R^n$, where $K \subset \R^m$, is of \textit{class $C^r$} for $1\leq r\leq \infty$ if $f$ admits an extension $\hat{f}$ of class $C^r$ to an open neighborhood of ${K}$. In this context, a map $f:K_1\rightarrow K_2$ between subsets $K_1\subset \R^m$ and $K_2\subset \R^n$ is called a \textit{diffeomorphism of class $C^r$} if both $f$ and its inverse $f^{-1}$ are of class $C^r$.

\begin{definition}
\label{def:generalized}
A \textit{GGS manifold} is a subset $M \subset \mathbb{R}^\ell$ such that for all $p \in M$ there is a neighborhood $V_p \subseteq M$ of $p$ and a diffeomorphism of class $C^\infty$ $\psi:V_p \rightarrow \mathcal{P}$ such that $\psi(p)=0$, where $\mathcal{P}$ is either a regular region, a basic or mixed  singular region.
\end{definition}

Analogously to the case of GS manifolds, we define the subset $M(\mathcal{P})\subset M$ consisting of points in $M$ that admit local charts of $\mathcal{P}$-type, where $\mathcal{P}$ is either a regular region, a basic or mixed singular region. Furthermore, these sets provide the following \textit{decomposition} of a GGS manifold $M$ into smooth submanifolds:
\begin{align*}
    M=\bigcup_{\mathcal{P}}M(\mathcal{P}),
\end{align*}
where the union runs over $\mathcal{P}$ for regular region, a basic or mixed singular region. The \textit{singular part} of $M$ is defined as:
\begin{align*}
    \mathcal{SP}(M)\coloneqq \bigcup_{\mathcal{P}}M(\mathcal{P}),
\end{align*}
where $\mathcal{P}$ ranges over the basic and mixed singular regions. That is, $\mathcal{SP}(M)$ consists of all non-regular parts of $M$:
\begin{align*}
    \mathcal{SP}(M)=M\setminus M(\mathcal{R}).
\end{align*}
The \textit{set of folds} on $M$ denoted by $\mathcal{F}(M)$, is defined as the union of $1$-dimensional sets minus the $0$-dimensional sets in the decomposition on $M$. 

A vector field $X$ of class $C^r$ on $\mathbb{R}^\ell$ is said to be \textit{tangent} to a GGS manifold $M\subset \mathbb{R}^\ell$ if its restriction to each smooth submanifold $M(\mathcal{P})$ is itself tangent to $M(\mathcal{P})$. The space of all such vector fields of class $C^r$ tangent to $M$ is denoted by $\mathfrak{X}^r(M)$.

Let $\varphi$ be the flow associated of a vector field $X \in \mathfrak{X}^r(M)$.  A \textit{cycle} is a nonempty compact subset of $M$ consisting of a finite set of singularities of $X$:
$$\{p_1, p_2, \dots, p_n, p_{n+1}=p_1\},$$ and a set of 
trajectories of $\varphi$: $$\{\gamma_0, \gamma_1, \dots, \gamma_{n}=\gamma_0\},$$ such that for each $i \in \{1, \dots, n\}$, the following holds:
$$
\omega(\gamma_{i-1}) = \{p_i\} = \alpha(\gamma_i),
$$
where $\omega(\cdot)$ and $\alpha(\cdot)$ denote the $\omega$- and $\alpha$-limit sets, respectively, under the flow $\varphi$.

\begin{definition}
\label{def_GGS_vector_field}
    Let $M$ be a GGS manifold. A tangent vector field $X \in \mathfrak{X}^r(M)$ is called a \textit{GGS vector field} if it satisfies the following two conditions:
    \begin{enumerate}
        \item[$(i)$] The set of singularities and periodic orbits of $X$ is finite;
        \item[$(ii)$] The $\omega$- and $\alpha$-limit set of any trajectory is an singularity, a periodic orbit, or a cycle.
    \end{enumerate}
    The space of all GGS vector fields on $M$ is denoted by $\Sigma^r_{GGS}(M)$. The flow $\varphi$ associated with $X\in \Sigma^r_{GGS}(M)$ is called a \textit{GGS flow} on $M$ and the singularities of $X$, $\operatorname{Sing}(X)$, are referred to as \textit{GGS singularities}. In this case, the pair $(M,X)$ is called a \textit{GGS pair}. 
\end{definition}

The \textit{nature} of a GGS singularity is described by a word, that is, a finite sequence of letters in the alphabet:
\begin{align*}
    \{a, s, s_s, s_u, r\},
\end{align*}
where $a$ denotes attracting behavior; $s$, $s_s$ and $s_u$ refer to saddle type behaviors and  $r$ denotes repelling behavior of the flow. Whenever a letter is repeated $k$ times consecutively in a word, we abbreviate it using exponent notation, i.e.,  writing the letter raised to the power $k$. We also use the term \textit{union of natures} to refer to the concatenation of such words.

Let $\varphi$ the induced flow by vector field $X$, we define the \textit{stable set} and \textit{unstable set} of a point $x\in M$ as follows:
\begin{align*}
    W^s(x)\coloneqq\{ p\in M| \ \varphi_t(p)\to x \text{ as }t\to +\infty\};\\
    W^u(x)\coloneqq\{ p\in M| \ \varphi_t(p)\to x \text{ as }t\to -\infty\}.
\end{align*}
A trajectory $\gamma$ is a \textit{saddle connection} if there exist two singularities with saddle nature $p$ and $q$ of $X$ such that $\gamma \subseteq W^u(p)\cap W^s(q)$. 

\begin{remark}
We adopt the term \emph{generalized}, extending its definition from that used in \cite{2021homotopical}. Care must also be taken with the term ``singularity'', which may denote either a singularity of the vector field $X$ (a point of $\operatorname{Sing}(X)$) or a topological singularity of the GGS manifold $M$ (an point of $\mathcal{SP}(M)$). In all cases, the specific meaning will always be clear from the context.
\end{remark}

\begin{remark}
\label{remark_conexao_cross}
In this paper, we impose a single restriction on the class of GGS manifolds in Definition \ref{def:generalized}: we assume there are no connections through the singular part between a cross-cap and a double crossing, or between a cross-cap and a triple crossing. This assumption helps us avoid certain technical difficulties when defining the intersection number in Section \ref{sec:complex}. It is important to note that this restriction does not apply to connections between singularities of any type through the smooth part, this is precisely the central issue addressed in this work.
\end{remark}

\begin{remark}
\label{remark_saddles}
Definition \ref{def_GGS_vector_field} is more general than the vector fields defined in \cite{GS}, as it imposes no hyperbolicity condition. For the purposes of this article, however, we impose the following restrictions on the dynamics of GGS vector fields: they must have no periodic orbits, cycles, the $\alpha$- and $\omega$-limit set of every trajectory must be a singularity, and there are no saddle connections. The specific dynamics under consideration be clarified in the subsequent paragraphs.
\end{remark}

The restrictions discussed in Remarks~\ref{remark_conexao_cross} and~\ref{remark_saddles} are fundamental to most of our subsequent results. For clarity and precision, we gather these hypotheses in the following definition.

\begin{definition}
\label{def_hipoteses}
A GGS pair $(M,X)$ is said to \textit{satisfy the condition $\mathcal{H}$} if the following hold:
\begin{enumerate}
    \item[$(i)$] $M$ is a closed GGS manifold with no folds connecting a cross-cap and a double or triple crossing point;
    \item[$(ii)$] the vector field $X$ has no periodic orbits and all its singularities are hyperbolic; 
    \item[$(iii)$] there are no saddle connections. 
\end{enumerate}
\end{definition}
As illustrated by the inclusion diagram in Figure~\ref{fig:diagrama_classes}, this section introduces the $\mathcal{GGS}$ class a broad and comprehensive category of pairs $(M,X)$. Our main goal in this article is to establish several theorems for GGS pairs with satisfy the condition $\mathcal{H}$, which constitutes a significantly more general setting than the cases studied in \cite{2021homotopical}.

% --- INÍCIO DA FIGURA ---
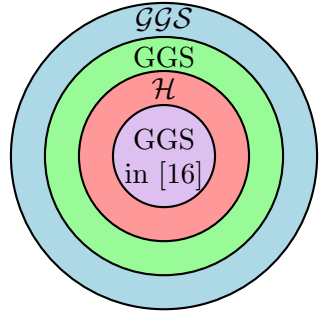
\begin{figure}[!ht]
\centering % Centraliza a imagem na página
 % --- SEU DIAGRAMA TIKZ COM CÍRCULOS ---
 \begin{tikzpicture}[scale=0.45]
 % Define cores
 \definecolor{myblue}{RGB}{173, 216, 230} % Azul claro
 \definecolor{mygreen}{RGB}{152, 251, 152} % Verde claro
 \definecolor{mycream}{RGB}{255, 253, 208}
 \definecolor{mylilac}{RGB}{221, 191, 240}
 \definecolor{mylightgray}{RGB}{245, 245, 245}
 \definecolor{mypastelred}{RGB}{255, 153, 153}
 \definecolor{myterracotta}{RGB}{220, 140, 120}
 \definecolor{mysoftred}{RGB}{240, 128, 128}
 
 % Círculo externo (Azul)
 \filldraw [fill=myblue, draw=black, thick] (0,0) circle (4.5cm);
 \node at (0, 4) {$\mathcal{GGS}$}; % Posicionado dentro da área azul
 % Segundo círculo (Verde)
 \filldraw [fill=mygreen, draw=black, thick] (0,0) circle (3.5cm);
 \node at (0, 2.9) {GGS}; % Posicionado dentro da área verde
 % Terceiro círculo (Amarela)
 \filldraw [fill=mypastelred, draw=black, thick] (0,0) circle (2.5cm);
 \node at (0, 2) {$\mathcal{H}$}; % Posicionado dentro da área amarela
 % Quarto círculo (Laranja)
 \filldraw [fill=mylilac, draw=black, thick] (0,0) circle (1.5cm);
 % Para o texto longo não sair da elipse, usamos 'text width'
 \node [text width=2.5cm, align=center] at (0, -0.1) {GGS \\ in \cite{2021homotopical}};
 % Círculo interno (Branco)
 %\filldraw [fill=white, draw=black, thick] (0,-1.5) circle (1cm);
 %\node at (0,-1.5) {GS};
\end{tikzpicture}
% --- LEGENDA E RÓTULO ---
\caption{Inclusion diagram of the classes of pairs.}
 \label{fig:diagrama_classes}
\end{figure}
% --- FIM DA FIGURA ---

Note that GS singularities are specific instances of GGS singularities. See \cite{Montufar} for a detailed discussion of the natures of the GS case. Locally, the singularities associated with a GGS flow that satisfy the condition $\mathcal{H}$ exhibit in the following natures:

\begin{itemize}
    \item If $p \in M(\mathcal{C}_n)$ for $n>2$, we say that $p$ has \textit{$n$-sheet cone type} ($\mathcal{C}_n$-type) and \textit{super attractor} (resp., \textit{repeller}) \textit{nature}, denoted by $\mathbf{a}$ (resp., $\mathbf{r}$), when $p$ is an attracting (resp., repelling) singularity. 

    \item If $p \in M(\mathcal{W}_n)$  for $n>2$, we say that $p$ has \textit{$n$-sheet cross-cap type} ($\mathcal{W}_n$-type) and \textit{super attractor} (resp., \textit{repeller}) \textit{nature}, denoted by $\mathbf{a}$ (resp., $\mathbf{r}$), when $p$ is an attracting (resp., repelling) singularity.  

    \item If $p \in M(\mathcal{D}_n)$ for $n>2$, we say that $p$ has \textit{$n$-sheet double crossing type} ($\mathcal{D}_n$-type) and \textit{super attractor} (resp., \textit{repeller}) \textit{nature}, denote by $\mathbf{a^n}$ (resp., $\mathbf{r^n}$), when $p$ is an attracting (resp., repelling) singularity. 
    
    \item If $p \in M(\mathcal{T}_{2k+1})$ for $k > 1$, we say that $p$ has \textit{$n$-sheet triple crossing type} ($\mathcal{T}_n$-type) and \textit{super attractor} (resp., \textit{repeller}) \textit{nature}, denote by $\mathbf{a^{2k+1}}$ (resp., $\mathbf{r^{2k+1}}$), when $p$ is an attracting (resp., repelling) singularity. 

    \item If $p \in M(\mathcal{P\vee Q})$, we say that $p$ is a \textit{mixed GGS singularity} and has \textit{$\mathcal{P\vee Q}$-type} with \textit{super attractor} (resp., \textit{repeller}) \textit{nature} when $p$ is an attracting (resp., repelling) singularity. We adopt the following convention for the natures of $p$:  
    \begin{itemize}
        \item If $p \in M(\mathcal{P}\vee \mathcal{Q})$, where $\mathcal{P}\in\{\mathcal{C}_n, \mathcal{W}_n\}$ and $\mathcal{Q}\in\{\mathcal{W}_m, \mathcal{D}_m, \mathcal{T}_{2k+1}\}$, and $p$ is an attracting (resp., repelling) singularity, we say that $p$ has the same nature of an attracting (resp., repelling) singularity of $\mathcal{Q}$-type.
        \item If $p \in M(\mathcal{P}\vee \mathcal{Q})$, where $\mathcal{P}\in\{\mathcal{D}_n, \mathcal{T}_{2k+1}\}$ and $\mathcal{Q}\in\{\mathcal{D}_m,\mathcal{T}_{2k'+1}\}$, and $p$ is an attracting (resp., repelling) singularity, we say that $p$ has attracting (resp., repelling) nature and denote by the union of the natures of $\mathcal{P}$ and $\mathcal{Q}$ (with the same nature). 
    \end{itemize}
    \item If $p \in M(\mathcal{PQ})$,  we say that $p$ is a \textit{mixed GGS singularity} and has \textit{$\mathcal{PQ}$-type} with attracting  (resp., repelling) nature, when $p$ is an attracting (resp., repelling) singularity. We adopt the following convention for the number of natures of $p$:
        \begin{itemize}
        \item If $p \in M(\mathcal{W}_n\mathcal{Q})$, where $\mathcal{Q}\in\{\mathcal{D}_m, \mathcal{T}_{2k+1}\}$, and $p$ is an attracting (resp., repelling) singularity, we say that $p$ has the same nature of an attracting (resp., repelling) singularity of $\mathcal{Q}$-type.
        \item If $p \in M(\mathcal{D}_n\mathcal{T}_{2k+1})$ and $p$ is an attracting (resp., repelling) singularity, we say that $p$ has attracting (resp., repelling) nature and denote it by $\mathbf{a^{n\!+\!2k}}$ (resp., $\mathbf{r^{n\!+\!2k}}$).
        \end{itemize}
\end{itemize}

\begin{definition}
\label{def_nature_number}
Let $x \in \operatorname{Sing}(X)$ with nature $a^i s^j r^l$, define the \textit{k-th nature number} of $x$, denoted by $\eta_k(x)$, as follows:   
\begin{align}
    \eta_k (x)=\eta_k(a^is^jr^l)=
    \begin{cases}
        l, & \text{ if } k=2;\\
        j, & \text{ if } k=1;\\
        i, & \text{ if } k=0.
    \end{cases}
\end{align}
For example, if $x \in M(\mathcal{D}_2)$ and has nature $sr$, then $\eta_0(x)=0$, $\eta_1(x)=1$ and $\eta_2(x)=1$.
\end{definition}

%%%%%%%%%%%%%%%%%%%%%%%%%%%%%%%%%%%%%%%%%%%%%%%%%%%%%%%%%%%%%%%%%%%%%%%%%%%%%%%%%%%%%%%%%%%%%%%%%%%%%%%%%%%%%%%%%%%%%%%%%%%%%%%%%%%
\section{Morsification of Generalized Gutierrez-Sotomayor Flows}
%\label{sec:Morsificationn GGS}
\label{sec:morsification}

In \cite{2021homotopical}, the concept of morsification was introduced for GS pairs, but restricted to a single type of singularity. In this section, we present a morsification process applicable to any GGS pair $(M,X)$ that satisfies the condition $\mathcal{H}$. This morsification process is fundamental to the definition of a GGS chain complex, as it enables the definition of the intersection number and, consequently, a GGS boundary map, which is presented in the next section. Formally, the concept of morsification is defined as follows.

\begin{definition}
\label{def:morsificacao}
    Let $(M, X)$ be a GGS pair. A quadruple $(\widetilde{M}, \widetilde{X}, \mathfrak{h}, \mathfrak{p})$ is called a \textit{morsification} of $(M, X)$ if it satisfies the following conditions:
    \begin{enumerate}
        \item[$(i)$] $\widetilde{M}$ is a smooth $2$-manifold;
        \item[$(ii)$] $\varphi_{\widetilde{X}}$ is a smooth flow on $\widetilde{M}$ with only hyperbolic singularities and no saddle connections;
        \item[$(iii)$] $\mathfrak{h}:M\longrightarrow \widetilde{M}$ is a multivalued map such that $\mathfrak{h}$ restricted to:
        \begin{align*}
            M \setminus (\mathcal{SP}(M) \cup \{x \in M| \ \omega(x)=p\text{ or }\alpha(x)=p, \\
            \text{ where } p \text{ is a saddle cone singularity}\}),
        \end{align*}
        is a homeomorphism;
        \item[$(iv)$] $\mathfrak{p}:\widetilde{M}\longrightarrow M$ is a projection satisfying $\mathfrak{p}\circ \mathfrak{h}=id_{M}$.
    \end{enumerate}
\end{definition}

Naturally, we can define morsification locally for an isolating block\footnote{ An \textit{isolating block} for a singularity $p$ of a GS flow $\varphi_X$ is an isolating neighborhood $N\subset M$ of $p$ such that the exit set $N^{-1}\coloneqq\{x \in N | \varphi_X([0,t),x) \subsetneq N, \forall t > 0\}$ is closed.}.

The objective of Theorem \ref{teo:morsificacao} is to present basic morsifications of GGS manifolds that satisfy the condition $\mathcal{H}$, that is, to assemble a \textit{catalog} of local models that can serve as a guide for morsifying any singular manifold in this context. It is important to keep in mind that, saddle connections are not allowed, either in the regular or singular parts see Definition \ref{def_hipoteses}.

\begin{theorem}
\label{teo:morsificacao}
Let $(M, X)$ be a GGS pair satisfying the condition $\mathcal{H}$. Then there exists a morsification $(\widetilde{M}, \widetilde{X}, \mathfrak{h}, \mathfrak{p})$ of $(M, X)$.
\end{theorem}

\begin{proof}
The idea of proof is to morsify blocks containing folds and isolated singularities, and then glue these morsifications together in a suitable way to obtain a global morsification.

The first observation is that the morsification for cone-type singularities presented in \cite{2021homotopical} remains valid for our purposes. These are isolated singularities, and the construction in \cite{2021homotopical} is performed within an isolating block that preserves its boundary and ensures the vector field is transverse to that boundary.

Throughout the proof, we use the notations $r\rightarrow a$, $s \rightarrow a$ and $r \rightarrow s$, where each arrow denotes a connection between singularities of the same type through folds, respecting the associated natures.

\begin{enumerate}
\item First, we consider the possible cases of cross-cap to cross-cap connections via folds. Since saddle connections through flow lines in the singular part are not allowed, the valid cases are as follows:

\begin{enumerate}
\item Let $N$ be an isolating block for a connection between a repelling $n$-sheet cross-cap $x$ and an attracting $n$-sheet cross-cap $y$. Let $\widetilde{N}$ be homeomorphic to $\es^2$ minus two points, equipped with a vector field having a repelling singularity $\widetilde{x}$ and an attracting singularity $\widetilde{y}$. See Figure \ref{fig:Morsificacao_cross_cap}.

We denote by $u_i$ the flow lines in the singular part of $N$, where $i\in \{1,\cdots, n-1 \}$. Additionally, we identify $2(n-1)$ flow lines in $\widetilde{N}$ whose $\alpha$-limit is $\widetilde{x}$ and $\omega$-limit is $\widetilde{y}$; let $\widetilde{u}_ {1,i}, \widetilde{u}_{2,i} \in \mathcal{M}_{\widetilde{x}\widetilde{y}}$, where $i\in \{1, \cdots , n-1\}$, denote these flow lines.

Let us consider $f:N\longrightarrow \R$ to be the Lyapunov function constructed in \cite{LimaTenorio_1}, such that its image is $[0,1]\subset \R$ with $f^{-1}(1)=\{x\}$ and $f^{-1}(0)=\{ y \}$. Similarly, let $\widehat{f}:\widetilde{N}\longrightarrow [0,1]\subset \R$ be a function  with $\widehat{f}^{-1}(1)=\{\widetilde{x} \}$ and $\widehat{f}^{-1}(0)=\{\widetilde{y} \}$. Thus, if $x_i \in u_i$ there exists a unique $\delta \in (0,1)$ such that $\{ x_i \}=f^{-1}(\delta)\cap u_i$. For this same $\delta$ we have unique points determined by:
\begin{align*}
    \{x_{1,i}\}=\widehat{f}^{-1}(\delta)\cap \widetilde{u}_{1,i}\ \ \ \ \text{and}\ \ \ \ \{x_{2,i}\}=\widehat{f}^{-1}(\delta)\cap \widetilde{u}_{2,i}.
\end{align*}
In this way, we associate each flow line $u_i$ in the singular part of $N$ with $\widetilde{u}_{1,i}$ and $\widetilde{u}_{2,i}$, defining $\mathfrak{h}(x_i)\coloneqq\{\widetilde{x}_{1,i}$, $\widetilde{x}_{2,i}\}$. For $x\in N\setminus \mathcal{SP}(M)$, define $\mathfrak{h}$ induced by the identity. Additionally, $\mathfrak{h}(x)=\widetilde{x}$ and $\mathfrak{h}(y)=\widetilde{y}$. Naturally, $\mathfrak{p}(\widetilde{x}_{j,i})\coloneqq x_i$ for all $i, j$ and $\mathfrak{p}|_{\widetilde{N} \setminus F}$ is induced by the identity, where $F=\cup_{i,j}\widetilde{u}_{j,i}$.

\item The case of a connecting $r\rightarrow s_s$, that is, repelling cross-cap  with a stable saddle cross-cap (resp. $s_u\rightarrow a$ unstable saddle and attracting), where the saddle block has the shape of a pair of pants, follows an argument analogous to the repelling-attractor connection, as illustrated in Figure \ref{fig:Morsificacao_cross_cap}.

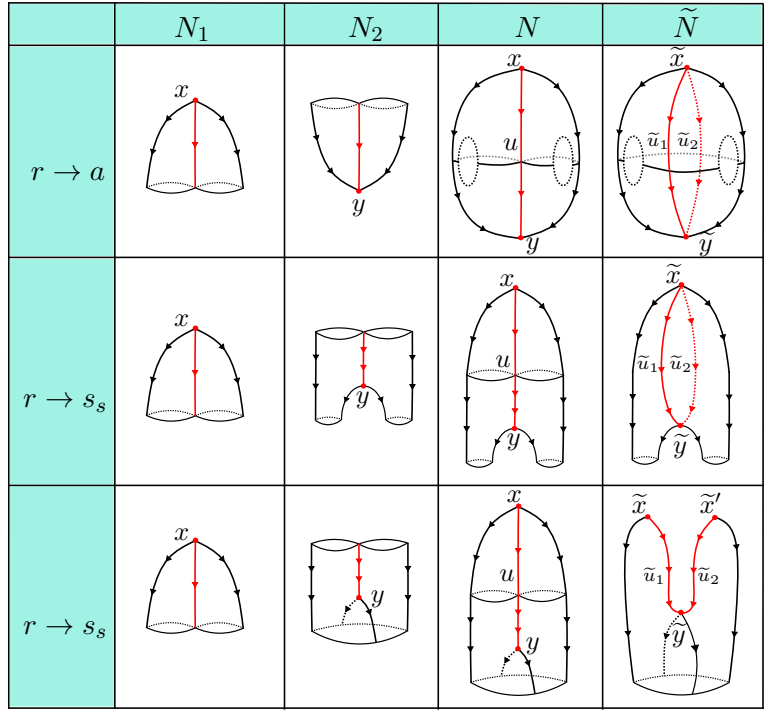
\begin{figure}[!ht]
    \centering
    \begin{overpic}[unit=0.5mm, 
    scale=.25,  width=10cm]{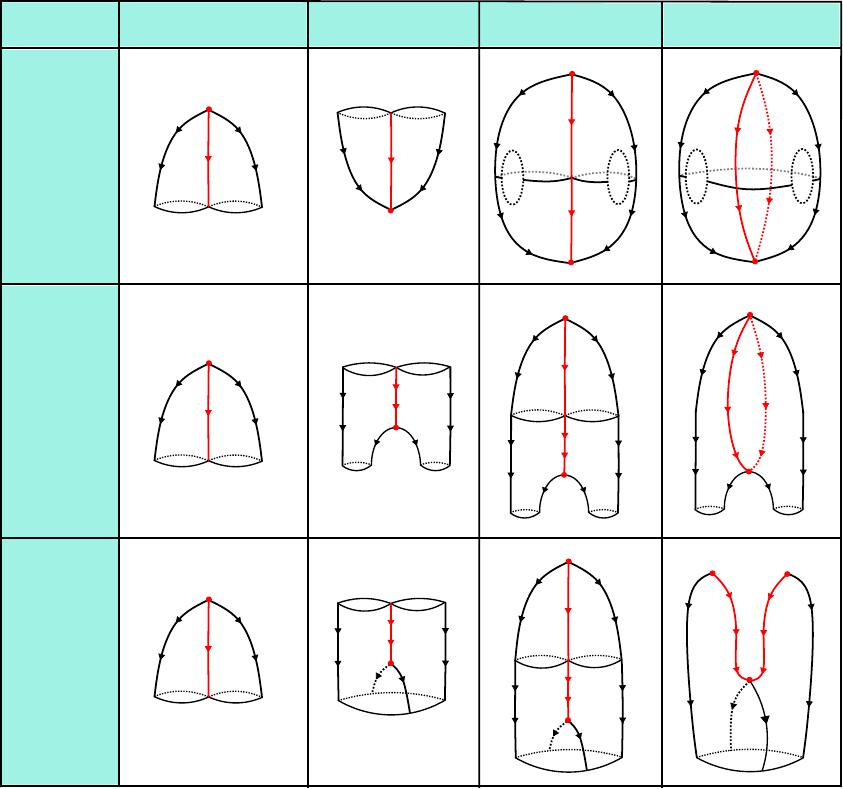}
    %legendas
    \put(3,70){$r\rightarrow a$}
    \put(2,40){$r\rightarrow s_s$}
    \put(2,10){$r\rightarrow s_s$}
    \put(22,89){$N_1$}
    \put(45,89){$N_2$}
    \put(67,89){$N$}
    \put(88,89){$\widetilde{N}$}
    %pontos_linha_1
    \put(22,81){\small{$x$}}
    \put(45.5,66){\small{$y$}}
    \put(66,85.2){\small{$x$}}
    \put(68.5,61){\small{$y$}}
    \put(87.5,85.2){\small{$\widetilde{x}$}}
    \put(91.4,60.7){\small{$\widetilde{y}$}}
    %pontos_linha_2
    \put(22,50.8){\small{$x$}}
    \put(46,40.5){\small{$y$}}
    \put(65,56){\small{$x$}}
    \put(66,35){\small{$y$}}
    \put(87,56.5){\small{$\widetilde{x}$}}
    \put(88,34){\small{$\widetilde{y}$}}
    %pontos_linha_3
    \put(22,22.5){\small{$x$}}
    \put(48,14.2){\small{$y$}}
    \put(66,27.5){\small{$x$}}
    \put(68.5,8){\small{$y$}}
    \put(82.5, 26){\small{$\widetilde{x}$}}
    \put(91.5, 26){\small{$\widetilde{x}'$}}
    \put(87.8,9){\small{$\widetilde{y}$}}
    %fluxos_linha_1
    \put(65.5,74){\small{$u$}}
    \put(84.5,74.5){\tiny{$\widetilde{u}_{1}$}}
    \put(88.5,74.5){\tiny{$\widetilde{u}_{2}$}}
    %fluxos_linha_2
    \put(83.2,45){\tiny{$\widetilde{u}_{1}$}}
    \put(87.5,45){\tiny{$\widetilde{u}_{2}$}}
    \put(64.5,45.3){\small{$u$}}
    %fluxos_linha_3
    \put(65,17){\small{$u$}}
    \put(84,17){\tiny{$\widetilde{u}_{1}$}}
    \put(91.2,17){\tiny{$\widetilde{u}_{2}$}}
    \end{overpic}
    \caption{Morsifications of connections through the fold between cross-cap.}
    \label{fig:Morsificacao_cross_cap}
\end{figure}

\item Finally, consider the case of a repelling cross-cap $x$ and a stable saddle cross-cap $y$ (resp. an unstable saddle and an attractor), with the isolating block $N_2$ for $y$ as shown in Figure \ref{fig:Morsificacao_cross_cap}, where $\partial N_2^+\cong\es^1 \vee \es^1$ and $\partial N_2 ^- \cong \es^1$. Let $\widetilde{N}$ be homeomorphic to $\es^2$ minus an open disk, equipped with a Morse-Smale flow transversal to the exit boundary, having two repelling singularities $\widetilde{x}$, $\widetilde{x}'$ and one saddle singularity $\widetilde{y}$. Denote by  $\widetilde{u}_{1}, \widetilde {u}_{2} \in \mathcal{M}_{\widetilde{x}\widetilde{y}}$ the flow lines whose $\alpha$-limit is $\widetilde{x}$ and $\omega$-limit is $\widetilde{y}$. We associate the flow line $u$ in the singular part of $N$ with $\widetilde{u}_{1}$ and $\widetilde{u}_{2}$, that is, given $x \in u$ we define $\mathfrak{h}(x)=\widetilde{x}_{1}, \widetilde{x}_{2}$, where $\widetilde{x}_{1}\in \widetilde{ u}_{1}$ and $\widetilde{x}_{2}\in \widetilde{u}_{2}$, in a unique way. For $x\in N\setminus \mathcal{SP}(M)$ define $\mathfrak{h}$ induced by the identity. Additionally, $\mathfrak{h}(x)=\{\widetilde{x}, \widetilde{x}'\}$ and $\mathfrak{h}(y)=\widetilde{y}$. Naturally, $\mathfrak{p}(\widetilde{x}_{i})\coloneqq x$ for all $i$ and $\mathfrak{p}|_{\widetilde{N} \setminus u}$ it is induced by the identity.
\end{enumerate}

\item Now, we consider the possible cases of double crossing to double crossing connections through the folds. Recall that saddle connections through flow lines in the singular part are not allowed. The relevant connections are $ r \rightarrow a$, $r\rightarrow sa$ (resp. $sr\rightarrow a$) and $r\rightarrow ss_s$ (resp. $a \rightarrow ss_u$) as illustrated in Figure \ref{fig:Morsificacao_DD}. In these cases, we can apply the morsifications presented in \cite{2021homotopical} for the isolating blocks of each singularity. Then, by gluing the exit boundary $\partial N_1^{-}$ to the entry boundary $\partial N_2^{+}$, noting that in these cases $\partial N_1^{-} \cong \partial N_2^{+}$, the desired result follows.

\item In a similar manner to the double crossing to double crossing connections, we can perform the same analysis for connections between triple crossings through the folds.

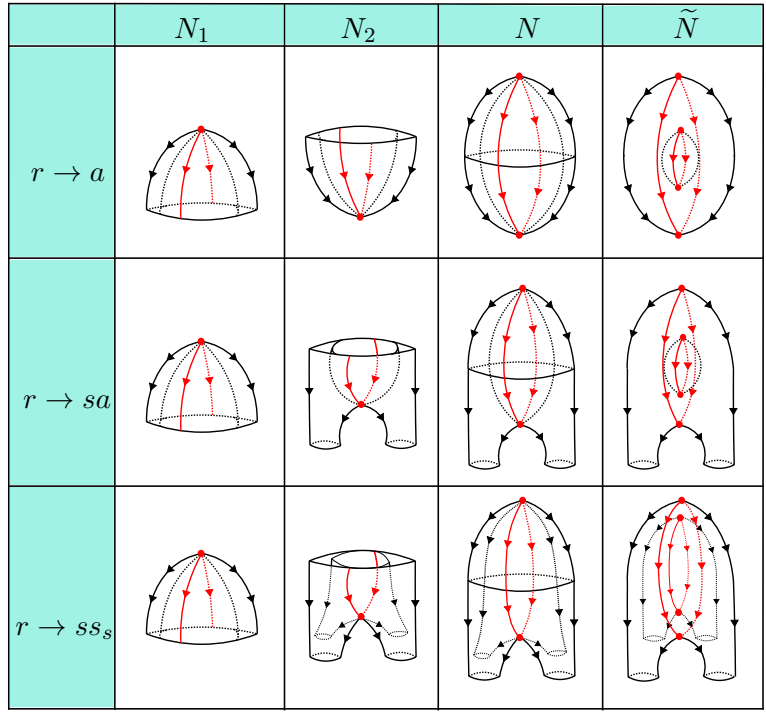
\begin{figure}[!ht]
    \centering
    \begin{overpic}[unit=1mm, scale=.25, width=10cm]{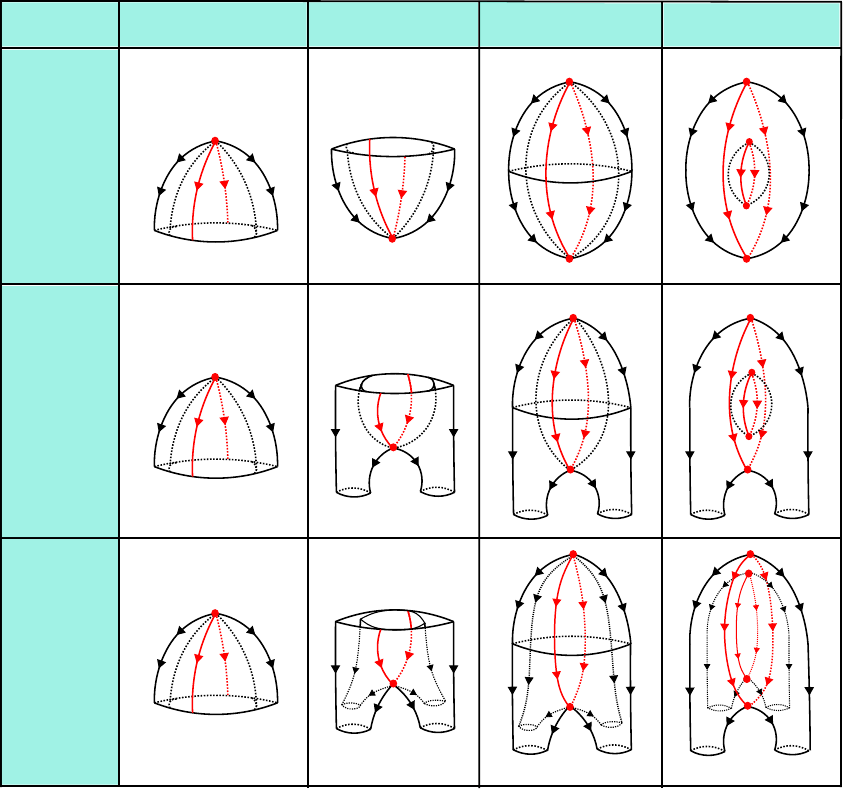}
    \put(3,70){$r\rightarrow a$}
    \put(2,40){$r\rightarrow sa$}
    \put(1,10){$r\rightarrow ss_s$}
    \put(22,89){$N_1$}
    \put(44,89){$N_2$}
    \put(67,89){$N$}
    \put(88,89){$\widetilde{N}$}
    \end{overpic}
   \caption{Morsifications of connections through the fold between double crossing.}
    \label{fig:Morsificacao_DD}
\end{figure}

\item Finally, we present the morsifications for the mixed singular regions. The morsification of these regions proceeds naturally, and we illustrate the following cases:

\begin{enumerate}
    \item Let $p$ be a singularity of type $\mathcal{P}\vee \mathcal{Q}$. We present a local morsification of $p$. To do so, let $N$ be an isolating neighborhood of $p$. By Definition \ref{def:GSmix}, we have $N=N_1\vee N_2$, where $N_1$ and $N_2$ are isolating neighborhoods of singularities of the type $\mathcal{P}$ and $\mathcal{Q}$, respectively. We define the multivalued map
    \begin{align*}
        h:N \longrightarrow N_1\sqcup N_2,
    \end{align*}
    such that $h|_{(N_1\vee N_2)\setminus \{p\}}$ is induced by the identity and $h(p)=\{p_1, p_2\}$, where $p_1$ is the singularity of type $\mathcal{P}$ in $N_1$ and $p_2$ is the singularity of type $\mathcal{Q}$ in $N_2$. Since we have already discussed the morsification of the basic singular types: $\mathcal{P}$ and $\mathcal{Q}$, there exist morsifications $\mathfrak{h}_1$ and $\mathfrak{h}_2$ of $N_1$ and $N_2 $, respectively. Thus, we define $\mathfrak{h}:N\longrightarrow \widetilde{N}\coloneqq\widetilde{N}_1\sqcup \widetilde{N}_2$ as follows:
    \begin{align*}
        \mathfrak{h}(x)\coloneqq
        \begin{cases}
            \mathfrak{h}_1 \circ h (x), \text{ if } x \in (N_1\vee N_2)\setminus N_2,\\
            \mathfrak{h}_2 \circ h (x), \text{ if } x \in (N_1\vee N_2)\setminus N_1,\\
            \mathfrak{h}_1 (p_1)\sqcup \mathfrak{h}_2 (p_2), \text{ if } x=p.
        \end{cases}
    \end{align*}
    We also have the respective projections $\mathfrak{p}_1:\widetilde{N}_1\longrightarrow N_1$ and $\mathfrak{p}_2:\widetilde{N}_2\longrightarrow N_2$, such that $\mathfrak{p }_1\circ \mathfrak{h}_1=id|_{N_1}$ and $\mathfrak{p}_2\circ \mathfrak{h}_2=id|_{N_2 }$. Define the map
    \begin{align*}
         p:N_1\sqcup N_2\longrightarrow N=N_1\vee N_2,
    \end{align*}
    where $p(p_1)=p(p_2)=p$ and $p|_{(N_1\sqcup N_2)\setminus \{p_1,p_2\}}=id$. Finally, define $\mathfrak{p}:\widetilde{N}_1\sqcup \widetilde{N}_1\longrightarrow N_1\vee N_2$ as follows:
    \begin{align*}
    \mathfrak{p}(x)\coloneqq
    \begin{cases}
        p\circ \mathfrak{p}_1(x), \text{ if } x\in \widetilde{N}_1;\\
        p\circ \mathfrak{p}_2(x), \text{ if } x\in \widetilde{N}_2.
    \end{cases}
    \end{align*}
    We claim that $\mathfrak{p}\circ \mathfrak{h}=id|_{N_1\vee N_2}$. Indeed, $\mathfrak{p}\circ \mathfrak{h}(p)=p$ and if $x \in N_i\setminus \{p\}$, where $i=1,2$, we have:
    \begin{align*}
        \mathfrak{p}\circ \mathfrak{h}(x)&= \mathfrak{p}\left( \mathfrak{h}(x) \right)\\
        &=\mathfrak{p}\left( \mathfrak{h}_i\circ h (x) \right)\\
        &=(p\circ \mathfrak{p}_i) \circ(\mathfrak{h}_i\circ h) (x) \\
        &=p\circ (\mathfrak{p}_i \circ \mathfrak{h}_i)\circ h (x) \\
        &=p\circ id_{N_i}\circ h (x) \\
        &=id_{N_1\vee N_2} (x).
    \end{align*}
    
    \item Let $p$ be a singularity of type $\mathcal{W}_n\mathcal{Q}$, with $\mathcal{Q}=\{\mathcal{D}_m, \mathcal{T}_{2k+1}\}$. It is sufficient to prove the morsification locally, so let $N$ be an isolating neighborhood of $p$. Note that $N$ can be seen as the space obtained by gluing $N_1$ and $N_2$ which are isolating neighborhoods of singularities $p_1$ and $p_2$ of types $\mathcal{Q}$ and $\mathcal {W}_{n-1}$\footnote{Note that if $n=2$, we have that $\mathcal{W}_{n-1}$ is already a smooth region}, respectively. This gluing is obtained by identifying a radius $r_1$ of $N_1$ with a radius $r_2$ of $N_2$ via the map $f$, that is, $N=N_1\cup_{f} N_2$. Consider the multivalued map
    \begin{align*}
        h:N_1\cup_{f} N_2 \longrightarrow N_1\sqcup N_2,
    \end{align*}
    such that $h|_{N\setminus r}$ is induced by the identity and $h(x)=\{x_1, x_2\}$ for all $x\in r$ such that $x_1\in r_1$, $ x_2\in r_2$. From what we proved previously, we know that there are $\mathfrak{h}_i: N_i\longrightarrow \widetilde{N}_i$ and $\mathfrak{p}_i: \widetilde{N}_i\longrightarrow N_i$ for $i=1 ,2$ such that $\mathfrak{p}_i\circ \mathfrak{h}_i=id_{N_i}$. Thus, we define $\mathfrak{h}:N_1\cup_{f}N_2 \longrightarrow \widetilde{N}_1\sqcup \widetilde{N}_2$ as follows:
    \begin{align*}
        \mathfrak{h}(x)\coloneqq
        \begin{cases}
            \mathfrak{h}_1 \circ h (x), \text{ if } x \in N_1,\\
            \mathfrak{h}_2 \circ h (x), \text{ if } x \in N_2.
        \end{cases}
    \end{align*}
    Consider the projection $p: N_1 \sqcup N_2\longrightarrow N_1\cup_{f} N_2$ and define $\mathfrak{p}: \widetilde{N_1}\sqcup \widetilde{N_2}\longrightarrow N_1\cup_{f} N_2 $ as follows:
    \begin{align*}
        \mathfrak{p}(x)\coloneqq
        \begin{cases}
            p\circ \mathfrak{p}_1(x), \text{ if } x \in \widetilde{N}_1,\\
            p\circ \mathfrak{p}_2(x), \text{ if } x \in \widetilde{N}_2.
        \end{cases}
    \end{align*}
    We affirm that $\mathfrak{p}\circ \mathfrak{h}=id|_{N_1 \cup_f N_2}$. Indeed, if $x \in N_i$, where $i=1,2$, we have:
    \begin{align*}
        \mathfrak{p}\circ \mathfrak{h}(x)&=\mathfrak{p}\left( \mathfrak{h}_i\circ h (x) \right)\\
        &=(p\circ \mathfrak{p}_i) \circ(\mathfrak{h}_i\circ h) (x) \\
        &=p\circ (\mathfrak{p}_i \circ \mathfrak{h}_i)\circ h (x) \\
        &=p\circ id_{N_i}\circ h (x) \\
        &=id|_{N_1 \cup_f N_2}.
    \end{align*}
    
    \item Let $p \in \mathcal{D}_n\mathcal{T}_{2k+1}$ and let $N$ be an isolating neighborhood of $p$. Note that by Definition~\ref{def:GSmix}, we can view $N$ as the quotient space of the following disjoint union:
    \begin{align*}
        N_0\coloneqq \bigsqcup_{i=1}^{n} \mathbb{D}^2 \bigsqcup_{j=1}^{k} N_k,
    \end{align*}
    where $N_j$ is the neighborhood of a double crossing with $2$-sheets. Define, similarly to the previous cases, the multivalued application $h:N \longrightarrow N_0 $. Furthermore, we know that there are $\mathfrak{h}_j:N_j\longrightarrow \widetilde{N}_j$ and $\mathfrak{p}_j:\widetilde{N}_j\longrightarrow N_j$. Finally, define $\mathfrak{h}:N \longrightarrow \widetilde{N}\coloneqq\sqcup_{i=1}^n \mathbb{D}^2 \sqcup_{j=1}^k \widetilde{N} _j$, as follows:

    \begin{align*}
        \mathfrak{h}(x)=
        \begin{cases}
            h(x), \text{ if } x \in \sqcup_{i=1}^n\mathbb{D}^2;\\
            \mathfrak{h}_j\circ h (x), \text{ if } x \in N_j. 
        \end{cases}
    \end{align*}
    Consider the projection onto the quotient space $p: N_0 \longrightarrow N$ and define $\mathfrak{p}: \widetilde{N} \longrightarrow N$ as follows:
    \begin{align*}
    \mathfrak{p}(x)=
    \begin{cases}
        p(x), \text{ if } x \in \sqcup_{i=1}^n\mathbb{D}^2;\\
        p \circ \mathfrak{p}_j (x), \text{ if } x \in \widetilde{N}_j. 
    \end{cases}
    \end{align*}
    Note that if $x\in \sqcup_{i=1}^n\mathbb{D}^2$ then $\mathfrak{p}\circ \mathfrak{h}(x)=p \circ h (x) =id_N(x)$. On the other hand, if $ x \in N_j$ then
    \begin{align*}
        \mathfrak{p}\circ \mathfrak{h}(x)&= (p\circ\mathfrak{p}_j)\circ (\mathfrak{h}_j\circ h)(x) \\ 
        &= p\circ (\mathfrak{p}_j \circ \mathfrak{h}_j) \circ h(x) \\ 
        &= p\circ id_{N_j} \circ h(x) \\ 
        &= id_{N}(x).
        \end{align*}
\end{enumerate}
\end{enumerate}
Since all other cases of mixed singular regions are obtained by a finite number of steps applying the operations presented in Definition \ref{def:GSmix}, the morsifications of these singularities can likewise  be obtained through a finite number of  processes described in item $4.$ of this proof. Therefore, in all cases, we have a morsification that can be conveniently glued together to obtain a global morsification and the result follows.
\end{proof}

\begin{remark}
It is important to note that the morsification of a singular region is not unique. For example, consider the morsification for the cone singularity given in \cite{2021homotopical}, and compare it with a completely different morsification for the same singularity presented in \cite{LimaTenorio_2}. Throughout this article, we consistently use the morsification given by Theorem~\ref{teo:morsificacao}.
\end{remark}

\section{Generalized Gutierrez-Sotomayor Chain Complex}
\label{sec:complex}

The aim of this section is to define a chain complex associated with a GGS pair $(M,X)$, satisfying the condition $\mathcal{H}$, which encodes information on the trajectories of the flow. To this end, we use the morsification process described in Section \ref{sec:morsification}. We also provide a characterization of the matrix of the boundary operator defining this chain complex.

\subsection{Well-Definedness of the GGS Chain Complex}

Given $(M, X)$ a GGS pair satisfying the condition $\mathcal{H}$,  denote by $\varphi_X$ the associated flow. Inspired by the Morse case, given  $x$, $y \in \operatorname{Sing}(X)$, define the \textit{connecting manifold} between $x$ and $y$ by
\begin{align*}
\mathcal{M}_{xy}\coloneqq W^u(x)\cap W^s(y),    
\end{align*}
where $W^u(x)$ is the unstable set of $x$ and $W^s(y)$ is the stable set of $y$. Taking the quotient of the connecting manifold $\mathcal{M}_{xy}$ by the natural action of $\R$ on the flow lines, we obtain the \textit{moduli space} between $x$ and $y$, denoted by $\widehat{\mathcal{M}}_{xy}$.

Two singularities $x, y \in \text{Sing}(X)$ are said to be \textit{consecutive} if they satisfy the conditions $\eta_k(x) \neq 0$ and $\eta_{k-1}(y) \neq 0$ for some $k \in \{1, 2\}$, where $\eta_j$ is the $j$-th nature number of the singularity, as defined in Definition~\ref{def_nature_number}.

Unlike in the Morse complex, a single GGS singularity may contribute to more than one generator of the GGS chain complex. Following the notation in \cite{2021homotopical}, we recall the convention for the generators associated with a singularity $x$:
\begin{align}
\label{eq_geradores}
    \{h^i_k(x)|\ i=1, \ldots, \eta_k(x)\ \text{ and } \ k=0,1,2\}.
\end{align}
For instance, if $x \in M(\mathcal{T})$ is a singularity with $ssa$ nature, then the corresponding set of generators is given by:
\begin{align*}
    \{h^i_k(x)|\ i=1, \ldots, \eta_k(x)|\ \ k=0,1,2\}=\{h^1_0(x), h^1_1(x), h^2_1(x)\}.
\end{align*}
Similarly, if $y \in M(\mathcal{W})$  has nature $r$, then the corresponding set of generators is $ \{h^1_2(y)\}$. In Figure \ref{fig_intr_02}, one has  mixed singularities, namely $x_3$ and $x'_4$. Note that $x_3$ has type $\mathcal{W}_2\mathcal{D}_2$ and  repelling nature, while $x'_4$ has type $\mathcal{C}_2\vee\mathcal{W}_2\mathcal{D}_2$ and  attracting nature. According to the nature conventions in Section \ref{sec:GS}, the corresponding set of generators  are $\{h^1_2(x_3), h^2_2(x_3)\}$ and $\{h^1_0(x'_4)$, $h^2_0(x'_4)\}$, respectively.

For simplicity, we sometimes refer to $h^i_k(x)$ and $h^j_{k-1}(y)$ as consecutive singularities.

Note that in the cases where $x$ is a double or triple singularity, the number of singularities associated to $x$ by the morsification $\mathfrak{h}$ is exactly equal to the  number of generators in the set given by (\ref{eq_geradores}). Hence, we also denote these singularities in the  morsified manifold $\widetilde{M}$ by $h^i_k(x)$.
However, in the case where $x$ is a  cross-cap or cone singularity, this bijection does not always hold. In such situations, the notation $h^1_k(x)$ also refers to the set of singularities associated to $x$ through morsification, see for example Figure \ref{fig_teo_w_4}.

With this in mind, we introduce the notations $\mathcal{M}_{h^i_k(x) h^j_{k-1}(y)}$ and $\widehat{\mathcal{M}}_{h^i_k(x) h^j_{k-1}(y)}$, which represent the connecting manifolds and the moduli space in $\widetilde{M}$, respectively. When there is a bijection between the numbers of associated singularities and the number of generators, this notation coincides with the standard usage. However, when there are multiple singularities associated with $x$ that corresponds to a single generator $h^1_k(x)$, the connecting manifold and moduli space are defined as the union of the connecting manifolds and moduli spaces in $\widetilde{M}$.

For example, suppose $x\in \operatorname{Sing}(X)$ and $\widetilde{x}, \widetilde{x}' \in \widetilde{M}$ are two singularities associated to a single generator $h^1_k(x)$. Then we have:
\begin{align*}
    \mathcal{M}_{h^1_k(x) h^j_{k-1}(y)}=\mathcal{M}_{\widetilde{x} h^j_{k-1}(y)}\cup \mathcal{M}_{\widetilde{x}' h^j_{k-1}(y)};\\
    \widehat{\mathcal{M}}_{h^1_k(x) h^j_{k-1}(y)}=\widehat{\mathcal{M}}_{\widetilde{x} h^j_{k-1}(y)}\cup \widehat{\mathcal{M}}_{\widetilde{x}' h^j_{k-1}(y)}.
\end{align*}
We fix an orientation on each unstable manifold $W^u(\widetilde{x})$ for every regular singularity $\widetilde{x}$ of the morsified flow $\varphi_{\widetilde{X}}$. Having established this notation and this convention, we now introduce the next definition, which is one of the most important in this paper.

\begin{definition}[GGS chain complex]
\label{def:GS_chain}
Let $(M, X)$ be a GGS pair satisfying the condition $\mathcal{H}$. The \textit{generalized Gutierrez-Sotomayor chain group} or simply the \textit{GGS chain group} is denoted by $C_*(M, X; \Z)$ and is defined by:
\begin{align*}
    C_k(M, X; \Z)\coloneqq\bigoplus_{x \in \operatorname{Sing}(X)}\left( \bigoplus_{i=1}^{\eta_k(x)} \Z\langle h^i_k (x) \rangle \right), \ \ \text{for all} \ k \in \Z. 
\end{align*}
The \textit{k-th GGS boundary map}   $\triangle_k: C_k(M,X; \Z) \rightarrow C_{k-1}(M, X; \Z),$  is defined on the generators of the GGS chain group by
\begin{align*}
    \triangle_k  \langle h^i_k(x) \rangle \coloneqq \sum_{y \in \operatorname{Sing}(X)}\left( \sum_{j=1}^{\eta_{k-1}(y)} n(h^i_k(x), h^j_{k-1}(y)) \cdot \langle h^j_{k-1}(y) \rangle \right).
\end{align*}
and extended by linearity. 
\end{definition}
In order for the GGS boundary operator to be well-defined, we must specify $n(h^i_k(x),h^j_{k-1}(y))$, the intersection number between the generators $h^i_k(x)$ and $h^j_{k-1}(y)$. This number depends on the chosen morsification process.

\begin{definition}
\label{def:intersection}
   Let $x,y \in \operatorname{Sing}(X)$ be consecutive singularities. The \textit{GGS characteristic sign} $n_u$ of a flow line $u\in \widehat{\mathcal{M}}_{xy}$ is defined as follows:
    \begin{enumerate}
        \item[($i)$] If $u$ lies in the regular part of $M$ and neither $x$ nor $y$ are  cone type singularities with saddle nature, define $n_u\coloneqq n_{\widetilde{u}}$, since there is only one flow line associated with $u$ in the morsification.
        
        \item[$(ii)$] If $u$ lies in the singular part of $M$, define $n_u\coloneqq(n_{\widetilde{u}^1}, n_{\widetilde{u}^2})$. That is, $n_u$ is an ordered pair whose $j$-th element is the characteristic sign of the $j$-th flow line $\widetilde{u}^j$ associated with $u$ in morsification, where $j=1,2$.

        \item[$(iii)$] Suppose $x$ is a singularity of any type with repelling nature and if $y$ is a cone singularity with saddle nature. Let $\widetilde{y}$ and $\widetilde{y}'$ be the singularities associated with $y$ in the morsification process, i.e., $\mathfrak{h}(y)=\{\widetilde{y}, \widetilde{y}'\}$ and let $h^i_{2}(x)$ be the singularity in the morsified manifold associated to $x$ which connects to both $\widetilde{y}$ and $\widetilde{y}'$. Note that $\widehat{\mathcal{M}}_{xy}\cong\widehat{\mathcal{M}}_{h^i_{2}(x)\widetilde{y}}\cong\widehat{\mathcal{M}}_{h^i_{2}(x)\widetilde{y}'}$ So, given $u \in \widehat{\mathcal{M}}_{xy}$ we have the corresponding ones flow lines in the morsification: $\widetilde{u}\in \widehat{\mathcal{M}}_{h^i_{2}(x)\widetilde{y}}$ and $\widetilde{u}'\in \widehat{\mathcal{M}}_{h^i_{2}(x)\widetilde{y}'}$. We define
        \begin{align*}
            n_u\coloneqq
            \begin{cases}
                n_{\widetilde{u}}, \text{ if }n_{\widetilde{u}}=n_{\widetilde{u}'};\\
                0, \text{ if } n_{\widetilde{u}}\neq n_{\widetilde{u}'}.
            \end{cases}
        \end{align*}
        Analogously, we define the GGS characteristic sign for a cone type singularity with saddle nature that connects to a singularity of any type with attracting nature.
    \end{enumerate}
Now, we  define \textit{the intersection number} between $h^{i}_{k}(x)$ and $h^{j}_{k-1}(y)$ as
\begin{align*}
    n(h^i_k(x), h^j_{k-1}(y))\coloneqq
    \begin{cases}
    \displaystyle\sum_{u \in \widehat{\mathcal{M}}_{xy}}n_u,  \text{ if } x \text{ or } y \text{ is a saddle cone }; \\
    \displaystyle\sum_{\widetilde{u} \in \widehat{\mathcal{M}}_{h^i_k\!(x) h^j_{k\!-\!1}\!(y)}} \!\! \!\!\! n_{\widetilde{u}},  \text{ otherwise.}  
\end{cases}
\end{align*}
where, the last sum is taken over all $\widetilde{u} \in \widehat{\mathcal{M}}_{h^i_k(x) h^j_{k-1}(y)}$ and $\widehat{\mathcal{M}}_{h^i_k(x) h^j_{k-1}(y)}$ lies in the morsified manifold.  
\end{definition}

\begin{lemma}
\label{lemma_cone}
    Let $(M, X)$ be a GGS pair satisfying the condition $\mathcal{H}$. If $y$ is a cone singularity of saddle nature, then:
    \begin{align*}
        n(h^i_2(x),h^1_1(y))\cdot n(h^1_1(y),h^l_0(z))=0, 
    \end{align*}
    for all $h^i_2(x)$, $h^l_0(z)$, with $x,z\in \operatorname{Sing}(X)$.
\end{lemma}

\begin{proof}
Let $h^i_2(x)$ be a repelling generator that connects with $y$, which in turn, connects with an attracting generator $h^l_0(z)$, as represented in Figure \ref{fig_teorema_02}. By the definition of the GGS characteristic sign we have:
\begin{align*}
    n_u\coloneqq
    \begin{cases}
    n_{\widetilde{u}}, \text{ if }n_{\widetilde{u}}=n_{\widetilde{u}'};\\
   \ 0, \ \text{ if } n_{\widetilde{u}}\neq n_{\widetilde{u}'};
    \end{cases}\qquad \text{and} \qquad
    n_v\coloneqq
    \begin{cases}
    n_{\widetilde{v}}, \text{ if }n_{\widetilde{v}}=n_{\widetilde{v}'};\\
 \   0,\  \text{ if } n_{\widetilde{v}}\neq n_{\widetilde{v}'}.
    \end{cases}
\end{align*}

\begin{figure}[!ht]
\centering
\begin{overpic}[unit=1mm, scale=.25, width=7cm]{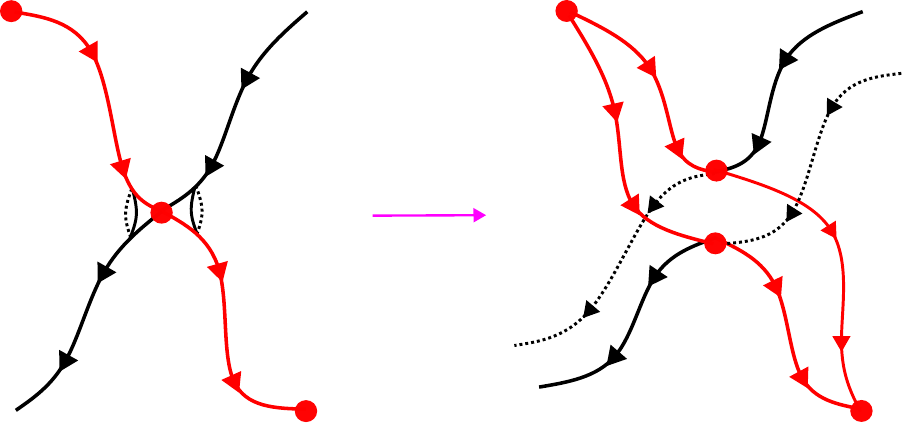}
\put(13,35){$u$}
\put(27,15){$v$}
\put(0,48){$x$}
\put(17,18){$y$}
\put(33,-4){$z$}
\put(46.5,25){$\mathfrak{h}$}
\put(61,48){$\widetilde{x}$}
\put(73,40){$\widetilde{u}$}
\put(62,30){$\widetilde{u}'$}
\put(78,32){$\widetilde{y}$}
\put(78,14){$\widetilde{y}'$}
\put(95,12){$\widetilde{v}$}
\put(83,5){$\widetilde{v}'$}
\put(94,-4){$\widetilde{z}$}
\end{overpic}
\caption{Morsification of a saddle cone singularity and its connections with attracting and repelling singularities.}
\label{fig_teorema_02}
\end{figure}
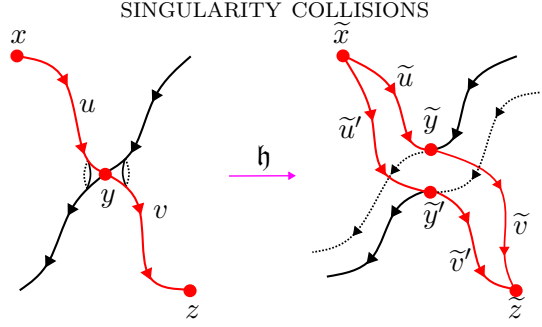

Since the sign is determined from a Morse-Smale flow on the smooth manifold $\widetilde{M}$, as shown in Figure \ref{fig_teorema_02}, we necessarily have either $n_u=0$ or $n_v=0$. Otherwise, we would obtain:
\begin{align*}
    \partial^m_1 \circ \partial^m_2 (\widetilde{x})\neq 0,
\end{align*}
where $\partial^m_*$ denotes the boundary map of the Morse complex of $\widetilde{M}$. Indeed, if both $n_u\neq 0$ and $n_v\neq 0$, then it must hold that $n_{\widetilde{u}}=n_{\widetilde{u}'}$ and $n_{\widetilde{v}}= n_{\widetilde{v}'}$, which implies that
\begin{align*}
    n_{\widetilde{u}'}\cdot n_{\widetilde{v}'}+ n_{\widetilde{u}}\cdot n_{\widetilde{v}}\neq 0.
\end{align*}
This leads to a contradiction, since we are in a Morse case. Therefore
\begin{align*}
    n(h^i_2(x), h^j_1(y))\cdot n(h^j_1(y), h^l_0(z)) =0.
\end{align*}
\end{proof}

\begin{theorem}
\label{teo:complexo}
Let $(M, X)$ be a GGS pair satisfying the condition $\mathcal{H}$. Then the GGS chain complex $(C_*(M, X; \Z), \triangle_*)$ is well-defined, that is,  for all $k \in \Z$,  $$\triangle_{k-1}\circ \triangle_{k}=0.$$
\end{theorem}

\begin{proof}
Since $\triangle_k=0$ for all $k \in \Z$ with $k\neq 1, 2$, to prove that $(C_*(M, X; \Z), \triangle_*)$  is a chain complex, it is sufficient to verify that  $\triangle_{1}\circ \triangle_{2}=0.$  Hence, it suffices to show that $\triangle_{1}\circ \triangle_{2}\langle h^i_2(x) \rangle=0$, for any generator $\langle h^i_2(x) \rangle$. According to  Definition \ref{def:GS_chain}, we have:
\begin{align*}
   & \triangle_{1}\circ \triangle_{2}\langle h^i_2(x) \rangle\\ & =\triangle_{1}\left( \sum_{y \in \operatorname{Sing}(X)}\sum_{j=1}^{\eta_1(y)}n(h^i_2(x),h^j_1(y))\cdot \langle h^j_1(y)\rangle \right)\\
    & =\sum_{y \in \operatorname{Sing}(X)}\sum_{j=1}^{\eta_1(y)}n(h^i_2(x),h^j_1(y))\cdot \triangle_{1} \langle h^j_1(y)\rangle\\
    & =\sum_{y \in \operatorname{Sing}(X)}\sum_{j=1}^{\eta_1(y)}n(h^i_2(x),h^j_1(y))\cdot \left( \sum_{z \in \operatorname{Sing}(X)}\sum_{l=1}^{\eta_0(z)}n(h^j_1(y), h^l_0(z))\cdot \langle h^l_0(z)\rangle\right)\\
    & =\sum_{z \in \operatorname{Sing}(X)} \sum_{y \in \operatorname{Sing}(X)} \left( \sum_{j=1}^{\eta_1(y)} \sum_{l=1}^{\eta_0(z)} n(h^i_2(x),h^j_1(y)) \cdot n(h^j_1(y), h^l_0(z))\right) \langle h^l_0(z)\rangle .
\end{align*}
By Lemma \ref{lemma_cone} and the last equality, we  conclude that $\triangle_{1}\circ \triangle_{2} \langle h^i_2(x) \rangle$ is independent of saddle cone singularities. Thus, without loss of generality, we proceed under the assumption that there are no cone-type singularities with saddle nature. In this case, by the definition of the intersection number, we have:
\begin{align*}
    n(h^i_2(x),h^j_1(y)) = \!\!\!\sum_{\widetilde{u} \in \widehat{\mathcal{M}}_{h^i_2(x) h^j_{1}(y)}} \!\!\!\!\!\!n_{\widetilde{u}}  \ \ \ \ \text{   and   } \ \ \ \  n(h^j_1(y), h^l_0(z)) =\!\!\!\sum_{\widetilde{v} \in \widehat{\mathcal{M}}_{h^j_1(y) h^l_{0}(z)}} \!\!\!\!\!\!n_{\widetilde{v}}.
\end{align*}
This observation is general and will be used repeatedly throughout the proof. Another important point is that the saddle singularities arising from the morsification of  cone singularities with attracting or repelling nature do not contribute to the sum in the boundary map of the Morse chain complex. This is because they always appear as part of a double connection with the regular attractor or repeller associated with the cone generator of the GGS chain complex.

The central idea of the proof is to characterize $\triangle_{2}\langle h^i_2(x) \rangle $ and $\triangle_{1}\langle h^j_1(y) \rangle$ using the Morse chain complex $(C^m(\widetilde{M},\widetilde{X}; \Z), \partial^m_*)$ of the morsified manifold. We analyze each case individually as follows.

\begin{enumerate}
\item\label{item_1}
Let $h^i_2(x)$ be a generator, where $x\in M(\mathcal{C}_n)$, $M(\mathcal{D}_n)$, $M(\mathcal{T}_{2k+1})$, or $M(\mathcal{R})$. In these cases, there is a bijection between the generators of the GGS chain complex associated to $x$ and the critical points associated to $x$ on the morsification, which are generators of the Morse complex of $\widetilde{M}$. Thus
\begin{align*}
\triangle_2 \langle h^i_2(x) \rangle &=\sum_{y\in \operatorname{Sing}(X)} \sum_{j=1}^{\eta_1(y)} n(h^i_2(x), h^j_1(y)) \cdot \langle h^j_1(y) \rangle\\
&=\sum_{y\in \operatorname{Sing}(X)} \sum_{j=1}^{\eta_1(y)}\sum_{\widetilde{u}\in \widehat{\mathcal{M}}_{h^i_2(x)h^j_1(y)}} n_{\widetilde{u}} \cdot \langle h^j_1(y) \rangle \\
&=\sum_{\widetilde{y}\in Crit_1(\varphi_{\widetilde{X}})} \sum_{\widetilde{u}\in \widehat{\mathcal{M}}_{\widetilde{x}^i,\widetilde{y}} } n_{\widetilde{u}} \cdot \langle \widetilde{y} \rangle\\
&=\sum_{\widetilde{y}\in Crit_1(\varphi_{\widetilde{X}})}n(\widetilde{x}^i,\widetilde{y}) \cdot  \langle \widetilde{y} \rangle\\
&=\partial^m_2\langle \widetilde{x}^i \rangle .
\end{align*}

\item \label{item_2}
Let $h^1_2(x)$ be a generator, where $x\in M(\mathcal{W}_n)$. Note that the morsification depends on the fold connections. Since $x\in M(\mathcal{W}_n)$ there are $(n-1)$-folds, and thus up to $(n-1)$ saddle singularities $y_1$, $\cdots$, $y_{n-1}$ of cross-cap type that may connect with $x$. The morsifications of each of these folds are described in the proof of Theorem \ref{teo:morsificacao}, with explicit examples for $n=4$ shown in Figure \ref{fig_teo_w_4}. Consequently, if the morsification associates to $x$ the set of singularities $\{\widetilde{x}^1, \cdots ,\widetilde{x}^p\}$, where $1\leq p\leq 2(n-1)$, then by convention we have:
\begin{align*}
\widehat{\mathcal{M}}_{h^1_2(x)h^j_1(y)}=\widehat{\mathcal{M}}_{\widetilde{x}^1h^j_1(y)}\cup \cdots \cup \widehat{\mathcal{M}}_{\widetilde{x}^ph^j_1(y)}.
\end{align*}
\begin{figure}[!ht]
\centering
\begin{overpic}[unit=1mm, scale=.25, width=8cm]{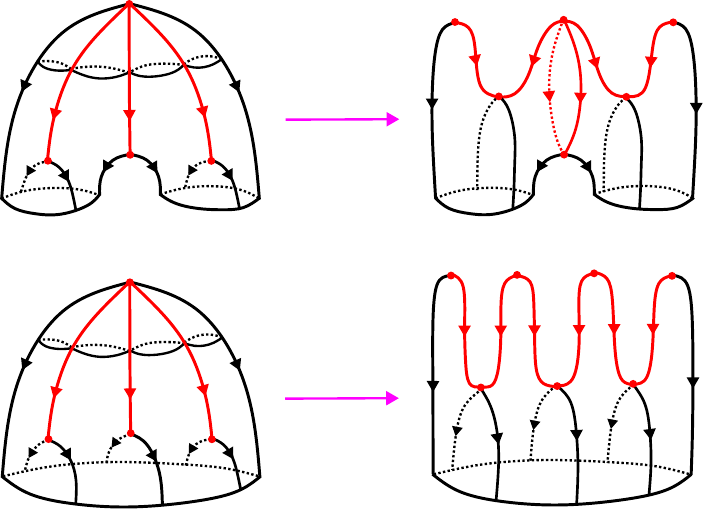}
\put(17,74){$x$}
\put(9.2,50){$y_1$}
\put(16.5,47){$y_2$}
\put(31.5,50){$y_3$}
\put(48,57){$\mathfrak{h}$}
\put(64,71){$\widetilde{x}^1$}
\put(78.5,71){$\widetilde{x}^2$}
\put(94.5,71){$\widetilde{x}^3$}
\put(70,61){$\widetilde{y}_1$}
\put(78.5,45.8){$\widetilde{y}_2$}
\put(87.4,61){$\widetilde{y}_3$}
%%%%%%%%%%%%%%%%%%%%%%%%%%%%
\put(17,33.5){$x$}
\put(8.7,10.5){$y_1$}
\put(20.3,10.8){$y_2$}
\put(30.9,10.8){$y_3$}
\put(48,18){$\mathfrak{h}$}
\put(63,35){$\widetilde{x}^1$}
\put(73,35){$\widetilde{x}^2$}
\put(84,35){$\widetilde{x}^3$}
\put(94,35){$\widetilde{x}^4$}
\put(66.9,19.5){$\widetilde{y}_1$}
\put(77.4,19.4){$\widetilde{y}_2$}
\put(88.4,19.5){$\widetilde{y}_3$}
\end{overpic}
\caption{Examples of morsification of repeller cross-cap singulariries $x\in M(\mathcal{W}_4)$. } 
\label{fig_teo_w_4}
\end{figure}
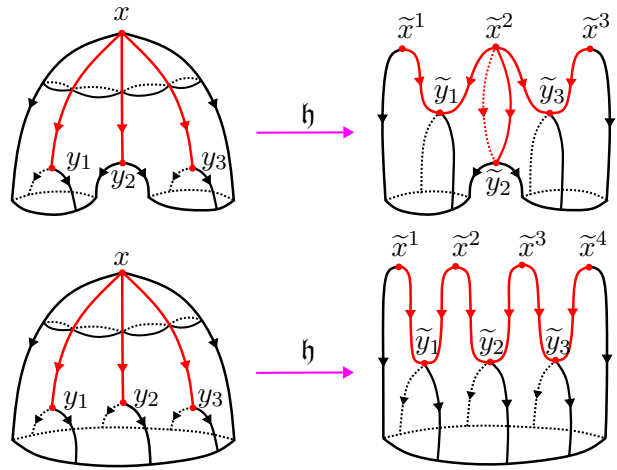

Thus,
\begin{align*}
& \triangle_2 \langle h^1_2(x) \rangle \\ &=\sum_{y\in \operatorname{Sing}(X)} \sum_{j=1}^{\eta_1(y)} n(h^1_2(x), h^j_1(y)) \cdot \langle h^j_1(y) \rangle\\
&=\sum_{y\in \operatorname{Sing}(X)} \sum_{j=1}^{\eta_1(y)} \sum_{\widetilde{u}\in \widehat{\mathcal{M}}_{h^1_2(x)h^j_1(y)}} n_{\widetilde{u}} \cdot \langle h^j_1(y) \rangle\\
&=\sum_{y\in \operatorname{Sing}(X)} \sum_{j=1}^{\eta_1(y)} \sum_{\widetilde{u}\in \widehat{\mathcal{M}}_{\widetilde{x}^1h^j_1(y)}} \!\!\!\!\!\!\!\!\! n_{\widetilde{u}} \cdot \langle h^j_1(y) \rangle \ + \cdots  + \!\!\! \!\!\! \sum_{y\in \operatorname{Sing}(X)} \sum_{j=1}^{\eta_1(y)} \sum_{\widetilde{u}\in \widehat{\mathcal{M}}_{\widetilde{x}^ph^j_1(y)}} \!\!\!\!\!\!\!\!\! n_{\widetilde{u}} \cdot \langle h^j_1(y) \rangle\\
&=\sum_{\widetilde{y}\in Crit_1(\varphi_{\widetilde{X}})}\sum_{\widetilde{u}\in \widehat{\mathcal{M}}_{\widetilde{x}^1\widetilde{y}}} n_{\widetilde{u}} \cdot \langle \widetilde{y}\rangle \  + \cdots +\sum_{\widetilde{y}\in Crit_1(\varphi_{\widetilde{X}})} \sum_{\widetilde{u}\in \widehat{\mathcal{M}}_{\widetilde{x}^p\widetilde{y}}} n_{\widetilde{u}} \cdot \langle \widetilde{y} \rangle\\
&=\sum_{\widetilde{y}\in Crit_1(\varphi_{\widetilde{X}})}n(\widetilde{x}^1, \widetilde{y})\cdot \langle \widetilde{y} \rangle  + \cdots + \sum_{\widetilde{y}\in Crit_1(\varphi_{\widetilde{X}})}n(\widetilde{x}^p, \widetilde{y})\cdot \langle \widetilde{y} \rangle\\
&=\partial^m_2\langle \widetilde{x}^1 \rangle + \cdots + \partial^m_2\langle \widetilde{x}^p \rangle .
\end{align*}

\item \label{item_3}
 Let $x\in M(\mathcal{P} \vee\mathcal{Q})$. By the construction on the proof of Theorem \ref{teo:morsificacao}, we have $h(x)=\{x^1, x^2\}$, where $x^1\in M(\mathcal{P})$ and $x^2\in M(\mathcal{Q})$. We follow the convention for the number of natures from Section \ref{sec:GS}. If $x$ has a repelling nature, then both  $x^1$ and $x^2$ also have repelling nature. In the case where $\mathcal{P}$ (resp., $\mathcal{Q}$) is of type $\mathcal{C}_n$ or $\mathcal{W}_n$,  we have:
\begin{align*}
    n(h^i_2(x), h^j_1(y))&=n(h^1_2(x^1), h^j_1(y)) + n(h^i_2(x^2), h^j_1(y));\\
    \text{resp., }  n(h^i_2(x), h^j_1(y))&=n(h^i_2(x^1), h^j_1(y)) + n(h^1_2(x^2), h^j_1(y)).
\end{align*}
In the case where $\mathcal{P}\vee \mathcal{Q}=\mathcal{C}_n \vee \mathcal{W}_m$ or $\mathcal{P}\vee \mathcal{Q}=\mathcal{W}_n\vee \mathcal{W}_m$, following the convention for the number of natures we have $\eta_2(x)=1$ and we have: 
\begin{align*}
    n(h^1_2(x), h^j_1(y))&=n(h^1_2(x^1), h^j_1(y)) + n(h^1_2(x^2), h^j_1(y)).
\end{align*}
These equalities hold in the general case and can be proven independently of type as follows:
\begin{align*}
    n(h^i_2(x), h^j_1(y))&=\sum_{u\in {\widehat{\mathcal{M}}}_{h^i_2(x)h^j_1(y)}} n_{\widetilde{u}}\\
    &=\sum_{u\in {\widehat{\mathcal{M}}}_{h^i_2(x^1)h^j_1(y)}} n_{\widetilde{u}} + \sum_{u\in {\widehat{\mathcal{M}}}_{h^i_2(x^2)h^j_1(y)}} n_{\widetilde{u}}\\
    &=  n(h^i_2(x^1), h^j_1(y))+ n(h^i_2(x^2), h^j_1(y)) . 
\end{align*}
Hence, 
\begin{align*}
    \triangle_2\langle h^i_2(x) \rangle &= \sum_{y\in \operatorname{Sing}(X)}\sum_{j=1}^{\eta_1(y)} n(h^i_2(x), h^j_1(y)) \cdot \langle h^j_1(y) \rangle\\
    &= \sum_{y\in \operatorname{Sing}(X)}\sum_{j=1}^{\eta_1(y)} \left( n(h^i_2(x^1), h^j_1(y))+ n(h^i_2(x^2), h^j_1(y))\right) \cdot \langle h^j_1(y) \rangle\\
    &=\sum_{y\in \operatorname{Sing}(X)}\sum_{j=1}^{\eta_1(y)} n(h^i_2(x^1), h^j_1(y)) \cdot \langle h^j_1(y) \rangle \  +\\
    &+ \sum_{y\in \operatorname{Sing}(X)}\sum_{j=1}^{\eta_1(y)} n(h^i_2(x^2), h^j_1(y)) \cdot \langle h^j_1(y) \rangle\\
    &=\triangle_2\langle h^i_2(x^1) \rangle + \triangle_2\langle h^i_2(x^2) \rangle. 
\end{align*}
Since $x^1\in M(\mathcal{P})$ and $x^2\in M(\mathcal{Q})$, where $M(\mathcal{P})=\mathcal{C}_n$, $\mathcal{W}_n$, $\mathcal{D}_n$, $\mathcal{T}_{2k+1}$ and $M(\mathcal{Q})=\mathcal{W}_m$, $\mathcal{D}_m$, $\mathcal{T}_{2k'+1}$,  by the previous  items (\ref{item_1}) and (\ref{item_2}), both are characterized by the Morse operator. Therefore,
\begin{align*}
     \triangle_2\langle h^i_2(x) \rangle &=\triangle_2\langle h^i_2(x^1) \rangle + \triangle_2\langle h^i_2(x^2) \rangle\\
     &=\sum_{l=1}^{p} \partial^m_2\langle \widetilde{x}^{1,l} \rangle + \sum_{l'=1}^{p'}\partial^m_2\langle \widetilde{x}^{2,l'}\rangle ,
\end{align*}
where $1\leq p\leq 2(n-1)$ and $1\leq p'\leq 2(m-1)$ and depends on the morsifications of $\mathcal{W}_n$ and $\mathcal{W}_m$. If $\mathcal{P}\neq \mathcal{W}_n$ and  $\mathcal{Q}\neq \mathcal{W}_m$ then $p=p'=1$.

\item \label{item_6}
If $x \in M(\mathcal{W}_n\mathcal{Q})$, where $\mathcal{Q}\in \{\mathcal{D}_m, \mathcal{T}_{2k+1} \}$, and has a repelling nature then by the convention of Section \ref{sec:GS} we have $\eta_2 (x)=m, 2k+1$. By the construction in Theorem \ref{teo:morsificacao} we have $h(x)=\{x^1, x^2\}$, where $x^1\in M(\mathcal{W}_n)$ and $x^2\in M(\mathcal{Q})$ with $\mathcal{Q}\in \{\mathcal{D}_m, \mathcal{T}_{2k+1}\}$. From this, we have:
\begin{align*}
    n(h^i_2(x), h^j_1(y))&=\sum_{\widetilde{u}\in \widehat{\mathcal{M}}_{h^i_2(x) h^j_1(y)}}n_{\widetilde{u}}\\
    &=\sum_{\widetilde{u}\in \widehat{\mathcal{M}}_{h^1_2(x^1) h^j_1(y)}}n_{\widetilde{u}}+\sum_{\widetilde{u}\in \widehat{\mathcal{M}}_{h^i_2(x^2) h^j_1(y)}}n_{\widetilde{u}}\\
    &=n(h^1_2(x^1), h^j_1(y))+n(h^i_2(x^2), h^j_1(y)).
\end{align*}
Then,
\begin{align*}
    \triangle_2\langle h^i_2(x) \rangle &=  \triangle_2\langle h^1_2(x^1) \rangle +  \triangle_2\langle h^i_2(x^2) \rangle\\
    &= \left( \partial^m_2 \langle \widetilde{x}^{1,1} \rangle + \cdots \partial^m_2 \langle \widetilde{x}^{1,k} \rangle\right) + \partial^m_2 \langle \widetilde{x}^{2,i} \rangle,
\end{align*}
where $k$ depends on the morsifications of the constructions of $x^1\in M(\mathcal{W}_m)$.

\item \label{item_7}
If $x \in M(\mathcal{D}_n\mathcal{T}_{2k+1})$ with  repelling nature, then $\eta_2(x)=n+2k$. By the construction in Theorem \ref{teo:morsificacao}, we have a bijection between the GGS generators associated to $x$ and the critical points in the morsified manifold associated to $x$. Therefore,
\begin{align*}
\triangle_2 \langle h^i_2(x)\rangle = \partial^m_2 \langle \widetilde{x}^i \rangle .    
\end{align*}

\item \label{item_8}
Let $h^j_1(y)$ be a generator associated to a regular, cross-cap, double or triple singularity. Again, we have a bijection between the generators of the morsification and those of the GGS chain complex, even when $y$ is a cross-cap with saddle nature, as illustrated in Figure \ref{fig_teorema}. Therefore,
\begin{align*}
\triangle_1 \langle h^j_1(y) \rangle &=\sum_{z\in \operatorname{Sing}(X)} \sum_{l=1}^{\eta_0(z)} n(h^j_1(y), h^l_0(z)) \langle h^l_0(z) \rangle\\
&=\sum_{\widetilde{z}\in Crit_0(\varphi_{\widetilde{X}})}n(\widetilde{y}^j,\widetilde{z})\langle \widetilde{z} \rangle \ = \ \partial^m_1 \langle \widetilde{y}^j \rangle .
\end{align*}
\begin{figure}[!ht]
\centering
\vspace{0.4cm}
\begin{overpic}[ unit=1mm, scale=.25, width=8cm]{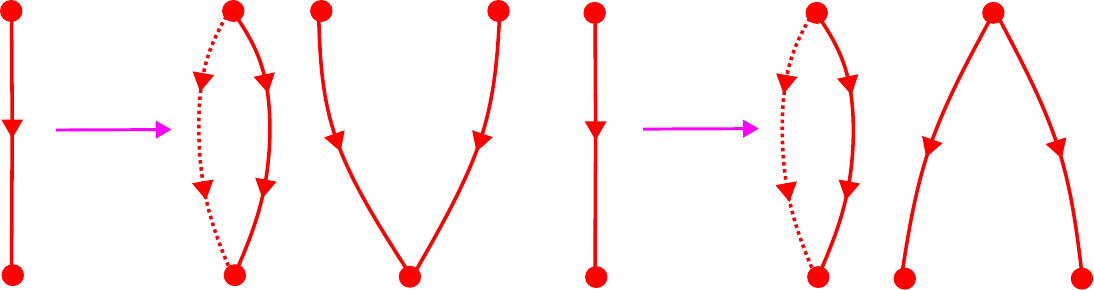}
\put(-2.5,13){$u$}
\put(0,28){$x$}
\put(0,-3){$y$}
\put(8.5,16){$\mathfrak{h}$}
\put(26,8){$\widetilde{u}_1$}
\put(13,8){$\widetilde{u}_2$}
\put(20.5,28){$\widetilde{x}$}
\put(20.5,-3.5){$\widetilde{y}$}
\put(28,28){$\widetilde{x}$}
\put(44.5,28){$\widetilde{x}'$}
\put(36,-3.5){$\widetilde{y}$}
\put(39.5,15.5){$\widetilde{u}_1$}
\put(31,15.5){$\widetilde{u}_2$}
%%%%%%%%% segundo pedaço
\put(55,8){$v$}
\put(53.5,28){$y$}
\put(53.5,-3){$z$}
\put(62,16){$\mathfrak{h}$}
\put(79,8){$\widetilde{v}_1$}
\put(67,8){$\widetilde{v}_2$}
\put(74,28){$\widetilde{y}$}
\put(74,-3.5){$\widetilde{z}$}
\put(82,-3.5){$\widetilde{z}$}
\put(98,-3.5){$\widetilde{z}'$}
\put(90,28){$\widetilde{y}$}
\put(94,8){$\widetilde{v}_1$}
\put(85,8){$\widetilde{v}_2$}
\end{overpic}
\vspace{0.2cm}
\caption{Morsification possibilities for connections between cross-caps through folds.}
\label{fig_teorema}
\end{figure}
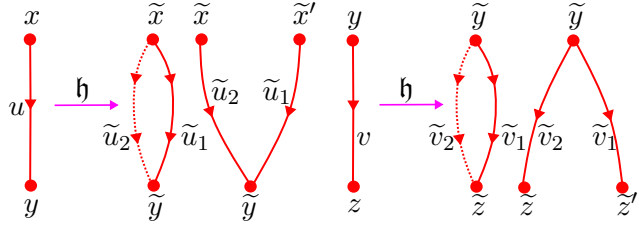
\end{enumerate}

Now, that we have characterized each boundary map via the boundary map of the morsification, we can do the same for their compositions. It follows directly from items (\ref{item_1}) -(\ref{item_8}) and the linearity of the boundary maps that:
\begin{align*}
    \triangle_1 \circ \triangle_2 \langle h^i_2 (x) \rangle= \sum_{l=1}^{p} \partial^m_1 \circ \partial^m_2 (\widetilde{x}^l),
\end{align*}
where $p \in \N$ depends on the morsification process of $x$ as described in Theorem \ref{teo:morsificacao}. The result follows from the fact that $\partial^m_1 \circ \partial^m_2 =0$. Furthermore, cases involving mixed singularities with more than two types follow the same reasoning and are omitted here for simplicity. Therefore, in all cases we have $\triangle_1 \circ \triangle_2 \langle h^i_2 (x) \rangle=0$.
\end{proof}

\subsection{Matrix of  a GGS boundary operator}
The main objective of this subsection is to introduce the matrix $\triangle$ associated with the GGS chain complex and to characterize some of its properties.

\begin{definition}
    Let $(C_*(M,X;\mathbb{Z}), \triangle_*)$ be a GGS chain complex. Define 
$$
\mathcal{C} \coloneqq C_0(M,X;\mathbb{Z}) \oplus C_1(M,X;\mathbb{Z}) \oplus C_2(M,X;\mathbb{Z}),
$$
and consider the $\mathbb{Z}$-linear map
$$\triangle : \mathcal{C} \longrightarrow \mathcal{C},$$
called the \emph{operator associated with the GGS chain complex}, such that with respect to a fixed basis $\{J_0,J_1,J_2\}$, where $J_i$ denotes the set of generators of $C_i(M,X;\mathbb{Z})$, the operator $\triangle$ is defined by
$$
\triangle|_{J_0} = 0, \qquad
\triangle|_{J_1} = \triangle_1, \qquad
\triangle|_{J_2} = \triangle_2,
$$
and extended by linearity.
\end{definition}

Thus, the matrix of the operator $\triangle$ has the following block form
\begin{align*}
\triangle \coloneqq
    \begin{pmatrix}
        0 & \triangle_1 & 0\\
        0 & 0 & \triangle_2\\
        0 & 0 & 0
    \end{pmatrix},   
\end{align*}
$\triangle$ is a square matrix with dimension equal to that of $\mathcal{C}$. It is a strictly upper triangular matrix composed of zero submatrices, except for the submatrices $\triangle_1$ and $\triangle_2$ which are, a priori, non-zero and correspond to the matrices of the GGS boundary operators in the basis $J_1$ and $J_2$, respectively.

In addition to the hypotheses assumed throughout this paper, see Remarks~\ref{remark_conexao_cross} and~\ref{remark_saddles}  and Definition \ref{def_hipoteses}, for the next theorem we require that $M$ admits a morsification $\widetilde{M}$ that is orientable. We fix an orientation for $\widetilde{M}$ and accordingly fix the same orientation for all unstable manifolds $W^u(\widetilde{x})$, where $\widetilde{x}\in Sing(\widetilde{X})$.

\begin{theorem}
\label{teo_matriz}
If $\triangle$ is the linear operator matrix associated to a GGS chain complex $(C_*(M,X;\Z), \triangle_*)$ and $M$ admits a morsification $\widetilde{M}$ that is orientable, then:
\begin{enumerate}
    \item[$(a)$] the entries of $\triangle_{1}$ and $\triangle_{2}$ belong to the set $\{-1,0,1\}$;
    \item[$(b)$] each column of $\triangle_1$ contains either exactly two non-zero entries, namely, $-1$ and $1$, or is zero;
    \item[$(c)$] each row of $\triangle_2$ contains either exactly two non-zero entries, namely $-1$ and $1$, or is zero.
\end{enumerate}
\end{theorem}

\begin{proof}
Let $x\in \operatorname{Sing}(X)$ be any singularity and consider the associated  generators $h^i_k(x)$ in dimension $k \in \{1,2\}$. Since
\begin{align*}
    \triangle_k\langle h^i_k(x)\rangle=\sum_{y\in \operatorname{Sing}(X)}\left(\sum_{j=1}^{\eta_{k-1}(y)} n(h^i_k(x), h^j_{k-1}(y)) \cdot \langle h^j_{k-1}(y) \rangle\right),
\end{align*}
the entries of the matrix $\triangle_k$ are given by $n(h^i_k(x), h^j_{k-1}(y))$, as we vary over all generators in dimensions $k$ and $k-1$. In what follows, we prove each statement of the theorem.

\begin{enumerate}
\item[$(a)$] To prove the first item, it suffices to show that
$
    n(h^i_k(x), h^j_{k-1}(y))\in \{-1,0,1\}.
$ 
By definition, we have
\begin{align*}
    n(h^i_k(x), h^j_{k-1}(y))=
    \begin{cases}
    \displaystyle\sum_{u \in \widehat{\mathcal{M}}_{xy}}n_u,  \text{ if } x \text{ or } y \text{ is a cone type saddle}; \\
    \displaystyle\sum_{\widetilde{u} \in \widehat{\mathcal{M}}_{h^i_k(x) h^j_{k-1}(y)}} \!\!\!\!\!\!n_{\widetilde{u}},  \text{ otherwise}.  
    \end{cases}
\end{align*}
Since $\dim(M)=\dim(\widetilde{M})=2$, then
$    \#\widehat{\mathcal{M}}_{xy}\leq 2 $  and $  \#\widehat{\mathcal{M}}_{h^i_k(x) h^j_{k-1}(y)}\leq 2.
$ 
    \begin{enumerate}
        \item[$(1)$] Assume that one of the singularities $x$ or $y$ is a saddle cone. Without loss of generality, let $y$ be a saddle cone. We have two cases to analyze $\#\widehat{\mathcal{M}}_{xy}=1$ and $\#\widehat{\mathcal{M}}_{xy}=2$, that is, $\widehat{\mathcal{M}}_{xy}=\{u\}$ and $\widehat{\mathcal{M}}_{xy}=\{u,v\}$. In the first case, by definition, we have
            \begin{align*}
                n(h^i_k(x), h^j_{k-1}(y))=n_u=
                \begin{cases}
                    n_{\widetilde{u}}, \text{ if } n_{\widetilde{u}}=n_{\widetilde{u}'}\\
                    0, \text{ if } n_{\widetilde{u}}\neq n_{\widetilde{u}'}.
                \end{cases}
            \end{align*}
            Here $\widetilde{u}$ and $\widetilde{u}'$ are flow lines in the morsification manifold. Since, $n_{\widetilde{u}}, n_{\widetilde{u}'}\in \{-1,0,1\}$, the result follows. Now if $\widehat{\mathcal{M}}_{xy}=\{u,v\}$, then
            \begin{align*}
                n(h^i_k(x), h^j_{k-1}(y))=n_u+n_v.
            \end{align*}
            Let $\widetilde{u}, \widetilde{u}'$ and $\widetilde{v}, \widetilde{v}'$ be the flow lines in the morsification manifold associated with $u$ and $v$, respectively. If either $n_{\widetilde{u}}\neq n_{\widetilde{u}}'$ or $n_{\widetilde{v}}\neq n_{\widetilde{v}}'$ the result follows since $n_u=0$ or $n_v=0$. If $n_{\widetilde{u}}= n_{\widetilde{u}}'$ and $n_{\widetilde{v}}= n_{\widetilde{v}}'$ then by definition
            \begin{align*}
                n(h^i_k(x), h^j_{k-1}(y))&=n_u+n_v =n_{\widetilde{u}}+n_{\widetilde{v}}.
            \end{align*}
            Since we are in the Morse configuration, this sum equals zero. That is, $n_{\widetilde{u}}$ necessarily has the opposite sign to $n_{\widetilde{v}}$, and therefore the result holds in all cases.
            
            \item[$(2)$] Now, assume that neither $x$ nor $y$ are saddle cones. By definition we have
                \begin{align*}
                    n(h^i_k(x), h^j_{k-1}(y))&=
                    %\begin{cases}
                    %\displaystyle\sum_{u \in \widehat{\mathcal{M}}_{xy}}n_u,  \text{ if } x \text{ or } y \text{ they are cone type saddles};\\
                    %\displaystyle\sum_{\widetilde{u} \in \widehat{\mathcal{M}}_{h^i_k(x) h^j_{k-1}(y)}}n_{\widetilde{u}}, \text{ otherwise}.
                    %\end{cases}
                    %\\
                    \displaystyle\sum_{\widetilde{u} \in \widehat{\mathcal{M}}_{h^i_k(x) h^j_{k-1}(y)}}n_{\widetilde{u}}.
                \end{align*}
                
            If $\widehat{\mathcal{M}}_{h^i_k(x) h^j_{k-1}(y)}=\{\widetilde{u}\}$, then $n(h^i_k(x), h^j_{k-1}(y))=n_{\widetilde{u}}$ and the result follows. If $\widehat{\mathcal{M}}_{h^i_k(x) h^j_{k-1}(y)}=\{\widetilde{u}, \widetilde{v}\}$ and none of the singularities is of the cross-cap type, then there is a bijection of generators, and the sum corresponds to the usual Morse configuration; hence, the result follows again. If $\widehat{\mathcal{M}}_{h^i_k(x) h^j_{k-1}(y)}=\{\widetilde{u}, \widetilde{v}\}$ and one of the singularities is a cross-cap which does not admit a bijection of generators with the morsification, then, by the orientation relations established, we have that $n_{\widetilde{u}}\neq n_{\widetilde{v}}$ and thus the sum is zero.
        \end{enumerate}

Therefore, in any case $n(h^i_k(x), h^j_{k-1}(y))\in \{-1,0,1\}$.

\item[$(b)$] Given $\triangle_{1}$, the  entries of each column of this matrix are given by
\begin{align*}
    \triangle_1\langle h^i_1(x)\rangle=\sum_{y\in \operatorname{Sing}(X)}\left(\sum_{j=1}^{\eta_{0}(y)} n(h^i_1(x), h^j_{0}(y)) \cdot \langle h^j_{0}(y) \rangle\right),
\end{align*}
where $h^i_1(x)$ is any saddle generator.
\begin{enumerate}
    \item[$(1)$] If $x$ is of regular, double or triple crossing type, there is a bijection between the GGS generators and the morsification. Therefore, in the morsification manifold we have two flow lines in 
       $  W^u(h^i_1(x))\setminus \{h^i_1(x)\},$ 
    with opposite characteristic signs. Since $h^{i}_1(x)$ connects either to two generators in dimension $0$ or to one generator in dimension $0$, then:
    \begin{align*}
        \triangle_1\langle h^i_1(x) \rangle&= +1 \langle h^j_0(x_k) \rangle -1 \langle h^l_0(x_m) \rangle ,\\
        \text{ or } \ \triangle_1\langle h^i_1(x) \rangle&= (+1-1) \langle h^j_0(x_k) \rangle=0\langle h^j_0(x_k) \rangle. 
    \end{align*}
    Hence, the column associated with $h^i_1(x)$ has only two possibilities: either the column is zero or it has two non-zero entries, one positive and one negative.
    \item[$(2)$] Let $x$ be a saddle cone. In the morsification of $x$ we associate two singularities that are regular saddles, say $\widetilde{x}$ and $\widetilde{x}'$, as illustrated in Figure \ref{fig:teorema_matriz}.
    
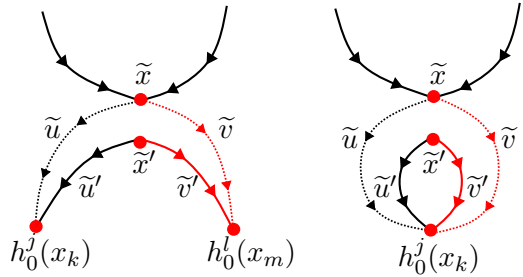
\begin{figure}[!ht]
\centering
\begin{overpic}[unit=1mm, scale=.25, width=6.5cm]{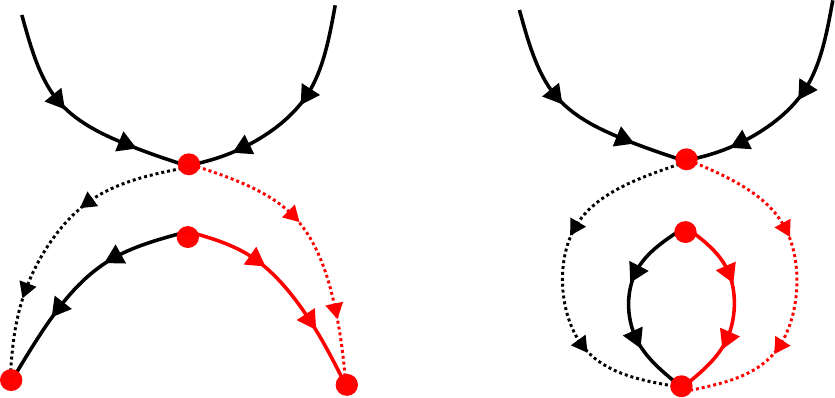}
\put(21.5,31){$\widetilde{x}$}
\put(21,13){$\widetilde{x}'$}
\put(81.5,31){$\widetilde{x}$}
\put(80,13){$\widetilde{x}'$}
\put(3,20){$\widetilde{u}$}
\put(10,8){$\widetilde{u}'$}
\put(63.5,18){$\widetilde{u}$}
\put(70,8){$\widetilde{u}'$}
\put(38.5,20){$\widetilde{v}$}
\put(30,8){$\widetilde{v}'$}
\put(96,18){$\widetilde{v}$}
\put(89,8){$\widetilde{v}'$}
\put(-4,-5){$h^j_0(x_k)$}
\put(36,-5){$h^l_0(x_m)$}
\put(75,-6){$h^j_0(x_k)$}
\end{overpic}
\vspace{0.2cm}
\caption{Morsifications of the saddle cone $x$ and its connections.}
\label{fig:teorema_matriz}
\end{figure}

If $n_u=0$ then $n_v=0$. Indeed, by definition, if $n_u=0$,  we have
    \begin{align}
    \label{eq_1_te0_6.1}
        n_{\widetilde{u}}\neq n_{\widetilde{u}'},
    \end{align}
and in the morsified manifold we also have:
    \begin{align}
        \label{eq_2_te0_6.1}n_{\widetilde{u}'}\neq n_{\widetilde{v}'}\\
        \label{eq_3_te0_6.1}n_{\widetilde{u}}\neq n_{\widetilde{v}}.
    \end{align}
From (\ref{eq_1_te0_6.1}) and (\ref{eq_2_te0_6.1}) we deduce $n_{\widetilde{u}}=n_{\widetilde{v}'}$ and from (\ref{eq_1_te0_6.1}) and (\ref{eq_3_te0_6.1}) we obtain $n_{\widetilde{u}'}=n_{\widetilde{v}}$. Thus,
    \begin{align*}
        n_{\widetilde{v}'}=n_{\widetilde{u}}\neq n_{\widetilde{u}'}=n_{\widetilde{v}}
    \end{align*}
Hence, $n_{\widetilde{v}'}\neq n_{\widetilde{v}}$ and $n_{v}=0$. Therefore, the column corresponding to $h^1_1(x)$ consists entirely of zero entries.

If $n_{u}=n_{\widetilde{u}}$ then $n_{v}=n_{\widetilde{v}}$. Indeed, by definition, we have $n_{\widetilde{u}}=n_{\widetilde{u}'}$. On the other hand, $n_{\widetilde{u}}\neq n_{\widetilde{v}}$ and $n_{\widetilde{u}'}\neq n_{\widetilde{v}'}$, so it follows that $n_{\widetilde{v}}= n_{\widetilde{v}'}$. Therefore, the column corresponding to $h^1_1(x)$  has exactly two non-zero entries with opposite signs.
    
If $h^1_1(x)$ connects with a single generator $h^j_0(x_k)$, then the column is null, since $\widetilde{x}$ and $\widetilde{x}'$ are regular saddles that have a double connection with $h^j_0(x_k)$ in the morsification manifold.

\item[$(3)$] If $x$ is a saddle cross-cap. Then $h^1_1(x)$ connects with a GGS generator $h^1_0(y)$. From the morsification of the two blocks, see Figure \ref{fig:Morsificacao_cross_cap}, and our chosen orientation, we have: $n(h^1_1(x),h^1_0(y))=+1-1=0$ and the result follows.
\end{enumerate} 
    
\item[$(c)$] Given $\triangle_{2}$. The column entries of this matrix are given by
\begin{align*}
    \triangle_2\langle h^i_2(x)\rangle=\sum_{y\in \operatorname{Sing}(X)}\left(\sum_{j=1}^{\eta_{1}(y)} n(h^i_2(x), h^j_{1}(y)) \cdot \langle h^j_{1}(y) \rangle\right),
\end{align*}
where $h^i_2(x)$ is any saddle generator. The entries of a row of $\triangle_2$ are given by the repelling generators that are mapped by $\triangle_2$ into a fixed saddle generator $h^j_1(y)$. Therefore, the proof proceeds analogously to item $(b)$.
\end{enumerate}
\end{proof}

\subsection{Examples of GGS chain complex}

We now present some examples of GGS pairs $(M,X)$  satisfying the condition $\mathcal{H}$, together with the computation of the associated GGS chain complexes defined in Definition \ref{def:GS_chain}, and the corresponding matrix $\triangle$.

\begin{example}
\label{exemplo_01}
Let $(M, X)$ be a GGS pair satisfying the condition $\mathcal{H}$ as shown on the left in the Figure \ref{fig:ex_02}.
\begin{figure}[!ht]
\centering
\begin{overpic}[ unit=1mm, scale=.25, width=11cm]{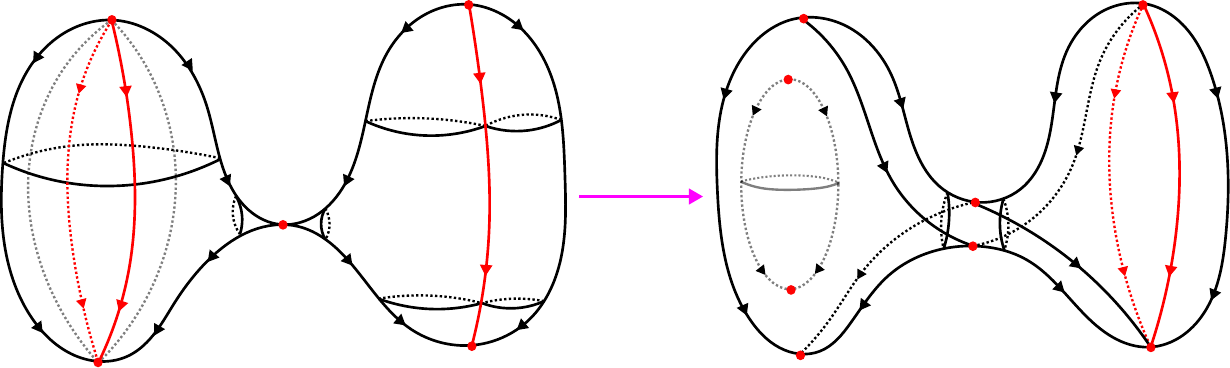}
        %pontos_no_singular
        \put(8,29){$x_1$}
        \put(37,30.5){$x_2$}
        \put(21.8,12.7){$x_3$}
        \put(6.5,-1.8){$x_4$}
        \put(37,-0.9){$x_5$}
        %\put(20,28){\footnotesize{$(M, X)$}}
        \put(51,15){$\mathfrak{h}$}
        %pontos_no_suave
        \put(62,30){$h^1_2(x_1)$}
        \put(61.2,24.6){\tiny{$h^2_2(x_1)$}}
        \put(61,-2){$h^1_0(x_4)$}
        \put(61.3,4.4){\tiny{$h^2_0(x_4)$}}
        \put(78.1,6.3){$\widetilde{x}_3'$}
        \put(78.1,15.6){$\widetilde{x}_3$}
        \put(90,31){$h^1_2(x_2)$}
        \put(90,-1.5){$h^1_0(x_5)$}
        %\put(76,28){\footnotesize{$(\widetilde{M}, \widetilde{X})$}}
\end{overpic}
\vspace{0.2cm}
\caption{GGS manifold with cone, cross-cap and double crossing singularities, and its morsification.}
\label{fig:ex_02}
\end{figure}
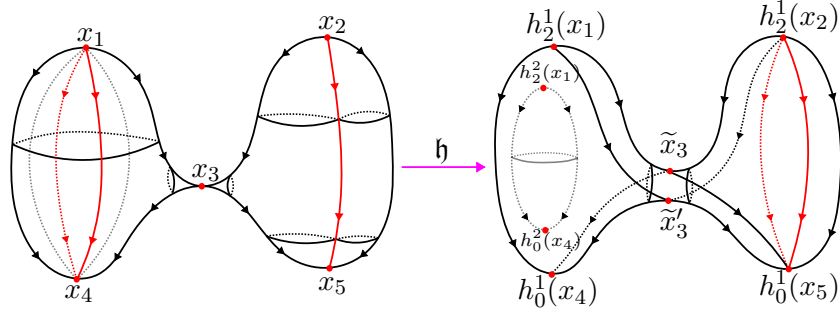

In this example, the GGS chain groups are given by:
\begin{align*}
     C_k(M, X; \Z)=
\begin{cases}
    \Z\langle h^1_2(x_1) \rangle \oplus \Z \langle h^2_2(x_1) \rangle \oplus \Z \langle h^1_2(x_2) \rangle, &  \text{if} \ k=2;\\
    \Z\langle h^1_1(x_3) \rangle, &  \text{if} \ k=1;\\
    \Z\langle h^1_0(x_4) \rangle \oplus \Z\langle h^2_0(x_4) \rangle \oplus \Z\langle h^1_0(x_5) \rangle, & \text{if} \ k=0.
\end{cases}
\end{align*}

The boundary maps are fully determined once their values on the generators of the GGS chain groups are specified. One has that 
\begin{align*}
    \triangle_2 \langle h^1_2(x_1) \rangle  &= n(h^1_2(x_1), h^1_{1}(x_3)) \langle h^1_{1}(x_3) \rangle; \\
    \triangle_2 \langle h^2_2(x_1) \rangle  &= n(h^2_2(x_1), h^1_{1}(x_3)) \langle h^1_{1}(x_3) \rangle; \\
    \triangle_2 \langle h^1_2(x_2) \rangle  &= n(h^1_2(x_2), h^1_{1}(x_3)) \langle h^1_{1}(x_3) \rangle; \\
    \triangle_1  \langle h^1_1(x_3) \rangle  &= n(h^1_1(x_3), h^1_{0}(x_4)) \langle h^1_{0}(x_4) \rangle +  n(h^1_1(x_3), h^2_{0}(x_4)) \langle h^2_{0}(x_4) \rangle \\
    & \ \ \ \  + n(h^1_1(x_3), h^1_{0}(x_5)) \langle h^1_{0}(x_5) \rangle.
\end{align*}
Furthermore, $\triangle_k=0$ for all $k \in \Z$ such that $k \neq 1$ and $k \neq 2$. From the  morsification process of the pair $(M, X)$ described in Theorem \ref{teo:morsificacao}, and illustrated in Figure \ref{fig:ex_02}, we have:
\begin{align*}
    n(h^1_2(x_1), h^1_1(x_3))&=+1; &  n(h^2_2(x_1), h^1_1(x_3))&=0; &  n(h^1_2(x_2), h^1_1(x_3))&=-1;\\
 n(h^1_1(x_3), h^1_0(x_4))&=0;     & n(h^1_1(x_3), h^2_0(x_4))&=0;
    &  n(h^1_1(x_3), h^1_0(x_5))&=0.  
\end{align*}
% \begin{align*}
%     n(h^1_2(x_1), h^1_1(x_3))&=+1; & n(h^1_1(x_3), h^1_0(x_4))&=0;\\
%     n(h^2_2(x_1), h^1_1(x_3))&=0;  & n(h^1_1(x_3), h^2_0(x_4))&=0;\\
%     n(h^1_2(x_2), h^1_1(x_3))&=-1; &  n(h^1_1(x_3), h^1_0(x_5))&=0.  
% \end{align*}
Thus
\begin{align*}
    \triangle_2 \langle h^1_2(x_1) \rangle &=+\langle h^1_{1}(x_3) \rangle;\qquad   \triangle_2 \langle h^2_2(x_1) \rangle = 0;\qquad 
    \triangle_2 \langle h^1_2(x_2) \rangle = - \langle h^1_{1}(x_3) \rangle;
\end{align*}
and, $\triangle_1 \equiv 0$. The associated matrix $\triangle$  is as shown in Table \ref{table_1}.

\begin{table}[!ht]
\centering
%\resizebox{\textwidth}{!}{
\begin{tabular}{|c|c|c|c|c|c|c|c|c|c|c|c|c| } 
\hline
& \cellcolor{ForestGreen!40} $h^1_0(x_4)$& \cellcolor{ForestGreen!40} $h^2_0(x_4)$ & \cellcolor{ForestGreen!40} $h^1_0(x_{5})$  & \cellcolor{ForestGreen!40} $h^1_1(x_{3})$  &  \cellcolor{ForestGreen!40} $h^1_2(x_{1})$ & \cellcolor{ForestGreen!40} $h^2_2(x_1)$ & \cellcolor{ForestGreen!40}  $h^1_2(x_{2})$ \\
\hline
\cellcolor{ForestGreen!40} $h^1_0(x_4)$ & $0$ & $0$ & $0$ & \cellcolor{WildStrawberry!40}$0$ & $0$ & $0$ & $0$   \\
\hline
\cellcolor{ForestGreen!40} $h^2_0(x_4)$ & $0$ & $0$ & $0$ & \cellcolor{WildStrawberry!40}$0$ & $0$  & $0$ & $0$  \\
\hline
\cellcolor{ForestGreen!40} $h^1_0(x_{5})$ & $0$  & $0$ & $0$ & \cellcolor{WildStrawberry!40}$0$ & $0$  & $0$ & $0$   \\
\hline
\cellcolor{ForestGreen!40} $h^1_1(x_{3})$ & $0$  & $0$ & $0$ & $0$ & \cellcolor{MidnightBlue!40}$+1$ & \cellcolor{MidnightBlue!40}$0$ & \cellcolor{MidnightBlue!40}$-1$  \\
\hline
\cellcolor{ForestGreen!40} $h^1_2(x_{1})$ & $0$  & $0$ & $0$ & $0$ & $0$ & $0$ & $0$ \\
\hline
\cellcolor{ForestGreen!40} $h^2_2(x_{1})$ & $0$  & $0$ & $0$ & $0$  & $0$ & $0$ & $0$ \\
\hline
\cellcolor{ForestGreen!40} $h^1_2(x_{2})$ & $0$ & $0$ & $0$ & $0$ & $0$ & $0$ & $0$ \\
\hline
\end{tabular}
%}
\caption{Matrix $\triangle$ associated with the $(C_*(M,X;\Z), \triangle_*)$.}
\label{table_1}
\end{table}
\end{example}

\begin{example}
\label{exemplo_02}
Let $(M, X)$ be a GGS pair satisfying the condition $\mathcal{H}$, as shown on the left in Figure \ref{fig:exemplo_final}. The GGS groups are given as follows:
\begin{align*}
    C_0(M, X; \Z) =&  \Z\langle h^1_0(x_{12}) \!\rangle\! \oplus  \Z\langle h^1_0(x_8) \rangle \!\oplus\! \Z\langle h^1_0(x_{11}) \rangle \!\oplus\!  \Z\langle h^1_0(x_7) \rangle \!\oplus \!  \Z\langle h^1_0(x_6) \rangle \!\oplus\! \Z\langle h^1_0(x_{10}) \rangle; \\
 C_1(M, X; \Z)= & \Z\langle h^1_1(x_9) \rangle \oplus \Z\langle h^1_1(x_4) \rangle \oplus \Z\langle h^1_1(x_5) \rangle \oplus \Z\langle h^2_1(x_5) \rangle \oplus \Z\langle h^1_1(x_3) \rangle ,\\
 C_2(M, X; \Z)=& \Z\langle h^1_2(x_1) \rangle \oplus \Z\langle h^1_2(x_2) \rangle \oplus \Z\langle h^2_2(x_2) \rangle.
\end{align*}

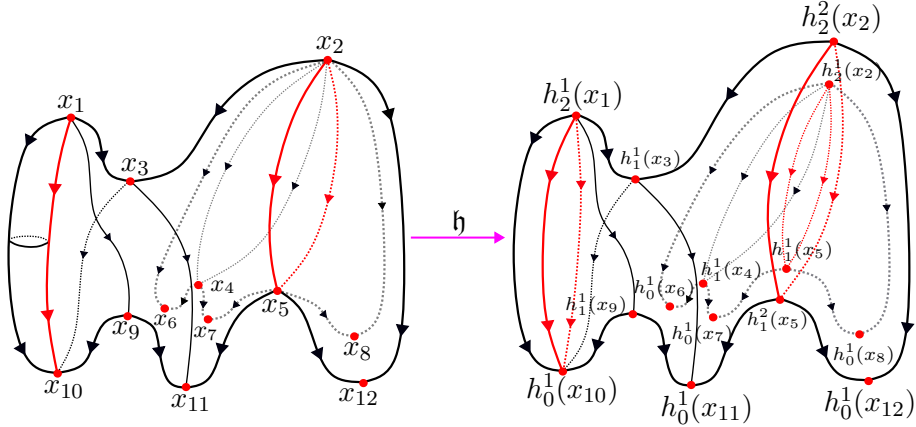
\begin{figure}[!ht]
\centering
\begin{overpic}[unit=1mm, scale=.25, width=12cm]{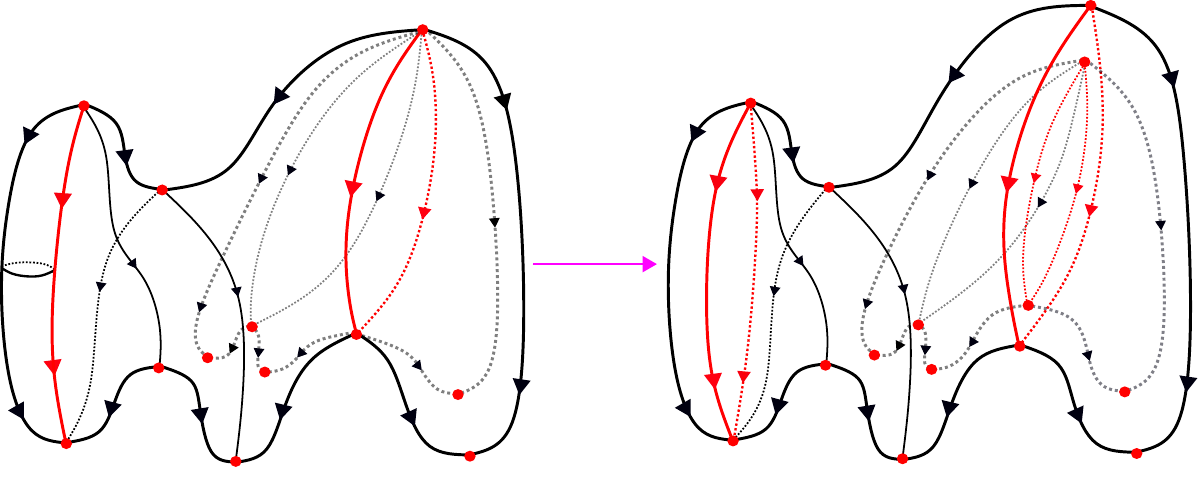}
%\put(0,36){$(M,X)$}
\put(5.5,33.5){$x_1$}
\put(34,40){$x_2$}
\put(12.5,26.5){$x_3$}
\put(22.4,13.3){\footnotesize{$x_4$}}
\put(28,11){$x_5$}
\put(16.1,9.5){\footnotesize{$x_6$}}
\put(20.5,8.2){\footnotesize{$x_7$}}
\put(37,6.2){$x_8$}
\put(11.5,8.5){$x_9$}
\put(4,2){$x_{10}$}
\put(18,0.5){$x_{11}$}
\put(36.5,1.2){$x_{12}$}
\put(49,20){$\mathfrak{h}$}
%%%%
%\put(63,36){$(\widetilde{M},\widetilde{X})$}
\put(59,34){$h^1_2(x_1)$}
\put(87.5,42.5){$h^2_2(x_2)$}
\put(89.7,37){\tiny{$h^1_2(x_2)$}}
\put(67.5,27.3){\tiny{$h^1_1(x_3)$}}
\put(76.4,15){\tiny{$h^1_1(x_4)$}}
\put(84.2,17.2){\tiny{$h^1_1(x_5)$}}
\put(81.5,9.5){\tiny{$h^2_1(x_5)$}}
\put(69,13){\tiny{$h^1_0(x_6)$}}
\put(73.2,8){\tiny{$h^1_0(x_7)$}}
\put(91,5.7){\tiny{$h^1_0(x_8)$}}
\put(61.5,11.2){\tiny{$h^1_1(x_9)$}}
\put(57.2,1.5){$h^1_0(x_{10})$}
\put(72, -0.5){$h^1_0(x_{11})$}
\put(90,0){$h^1_0(x_{12})$}
\end{overpic}
\caption{GGS manifold $M$ and its morsification $\widetilde{M}$.}
\label{fig:exemplo_final}
\end{figure}
Using the morsification depicted in Figure \ref{fig:exemplo_final} and applying the definition of the caracteristic sign, we obtain the expression of the boundary map on the generators:
\begin{align*}
    \triangle_2 \langle h^1_2(x_{1}) \rangle&=\langle h^1_1(x_3) \rangle+\langle h^1_1(x_9) \rangle,\\
    \triangle_2\langle h^2_2(x_{2}) \rangle&=-\langle h^1_1(x_3) \rangle-\langle h^1_1(x_9) \rangle,\\
    \triangle_1\langle h^1_1(x_{9}) \rangle&=\langle h^1_0(x_{10}) \rangle-\langle h^1_0(x_{11}) \rangle,\\
    \triangle_1\langle h^1_1(x_{4}) \rangle&=-\langle h^1_0(x_{6}) \rangle + \langle h^1_0(x_{7}) \rangle,\\
    \triangle_1\langle h^1_1(x_{5}) \rangle&=-\langle h^1_0(x_{7}) \rangle +\langle h^1_0(x_{8}) \rangle,\\
    \triangle_1\langle h^2_1(x_{5}) \rangle&=-\langle h^1_0(x_{11}) \rangle+\langle h^1_0(x_{12}) \rangle,\\
    \triangle_1\langle h^1_1(x_{3}) \rangle&=-\langle h^1_0(x_{10}) \rangle+\langle h^1_0(x_{11}) \rangle,
\end{align*}
and all other entries are zero. The matrix $\triangle$  the represetes the boundary operator of the GGS chain complex of $(M,X)$ is shown in Table \ref{table:1}.
\end{example}

\section{Detecting Local Homotopical Dynamical Cancellations}
\label{sec:detecting_dynamical_homotopical_cancellation}

In Section \ref{sec:collisions}, we introduced Definition \ref{definition_canc} of homotopical dynamical cancellation. Here, we present the first key feature of the GGS chain complex: it detects a relevant family of such cancellations. Beyond the properties described in Section \ref{sec:collisions}, this family has two further important characteristics: it preserves the basic singular structure of the GGS manifold and reduces the number of generators of the GGS chain complex.
 
Let $(M, X)$ be a GGS pair satisfying the condition $\mathcal{H}$. We perform a homotopical dynamical cancellation between three singularities, $x_1, x_2, x_3 \in \operatorname{Sing}(X)$, and two flow lines, $u_1$ and $u_2$, in the regular part connecting them. Furthermore two of the singularities have repelling (or attracting) nature, while the third has saddle nature. For simplicity, we say that the three singularities are cancelled, or that their natures are cancelled. It is not always possible to find three singularities that can be cancelled. 

The following theorem establishes that the chain complex defined in Section \ref{sec:complex} encodes some dynamical structure that makes it possible to detect homotopical cancellations.

\begin{theorem}
\label{teo_cancelamento}
\textit{(Homotopical Dynamical Cancellation Theorem)} Let $(M, X)$ be a GGS pair satisfying the condition $\mathcal{H}$, and let $x_1, x_2 \in \operatorname{Sing}(X)$. Suppose that their associated GGS group generators satisfy either 
\begin{align*}
n(h^j_2(x_1), h^m_1(x_2)) \neq 0\ \ \text{ or } \ \ n(h^m_1(x_2), h^j_0(x_1)) \neq 0.
\end{align*}
Then there exists $x_3 \in \operatorname{Sing}(X)$, along with connecting flows $u_1$ from $x_1$ to $x_2$ and $u_2$ from $x_3$ to $x_2$, such that the set $S=\{x_1, x_2, x_3, u_1, u_2\}$ admits a homotopical dynamical cancellation to a single GGS singularity $S'=\{x'\}$.
\end{theorem}

\begin{proof}
If $n(h^j_2(x_1), h^m_1(x_2))\neq 0$ then, by item $(c)$ of Theorem \ref{teo_matriz}, there exists another singularity $x_3\neq x_1$ with $$n(h^l_2(x_3), h^m_1(x_2))\neq 0.$$ Note that, if $x_2$ is not a cone of saddle nature, there is a  bijection between the GGS generators and the singularities in the morsification. In this case, $h^m_1(x_2)=\{\widetilde{x}_2\}$. On the other hand, if $x_2$ is a saddle cone singularity, then $h^m_1(x_2)=\{\widetilde{x}_2, \widetilde{x}'_2\}$ and all flow lines connecting to either  $\widetilde{x}_2$ or $\widetilde{x}'_2$ lie in the same connected component. Therefore, in either case, the generators $h^j_2(x_1), h^m_1(x_2), h^l_2(x_3)$ belong to the same connected component of the morsified manifold.

Since $n(h^j_2(x_1),h^m_1(x_2))\neq 0$ and $n(h^l_2(x_3),h^m_1(x_2))\neq 0$ it follows that there exist connecting orbits
\begin{align*}
    \widetilde{u}_1 \in \widehat{\mathcal{M}}_{h^j_2(x_1)h^m_1(x_2)} \ \ & \text{ and } \ \ \widetilde{u}_2 \in \widehat{\mathcal{M}}_{h^l_2(x_3)h^m_1(x_2)},
\end{align*}
and let $u_1\in \mathcal{M}_{x_1x_2}$ and $u_2\in \mathcal{M}_{x_3x_2}$ be the flow lines associated to $\widetilde{u}_1$ and $\widetilde{u}_2$ via morsification, that is, $u_1=\mathfrak{p}(\widetilde{u}_1)$ and $u_2=\mathfrak{p}(\widetilde{u}_2)$.

Denote by $S$ the subset of $M$ formed by the singularities $x_1$, $x_2$, $x_3$ together with the flow lines $u_1$ and $u_2$. Observe that $S$ is an isolated invariant set. Let $N$ be an isolating block for $S$ such that the vector field $X$ is transverse to the boundary  $\partial N$.

Let $\gamma(x)$ denote the orbit through a point $x \in M$. Clearly, $\gamma(x_i)=x_i$ for $i=1,2,3$. Define the path $\delta$  by concatenating the following  flow lines:
\begin{align*}
    \delta \coloneqq \gamma(x_3) * \gamma(u_2) * \gamma(x_2) * \gamma^{-1}(u_1) * \gamma(x_1).
\end{align*}
Thus, $\delta$ is a path starting at $x_3$ and ending at $x_1$, passing through the flow lines $u_2$ and $u_1$. Since $x_1 \neq x_3$, the path $\delta$ is not a closed loop and is therefore contractible to a point $x'$. This contraction induces a deformation of $N$ into a new neighborhood $N'$ such that $\partial N' = \partial N$ and $N'$ is homotopy equivalent to $N$.

Define $M'=(M\setminus N)\sqcup_{id} N'$. Moreover, we introduce a vector field on $N'$ which is transversal to the boundary $\partial N'$ and has a unique singularity at $x'$. The nature of this singularity is given by the natures of $x_1$, $x_2$, and $x_3$, excluding the contributions from the generators $h^m_1(x_2)$ and $h^l_2(x_3)$. Therefore, $S$ can be dynamically homotopically cancelled to a single GGS singularity $x'$. An analogous argument holds for the case involving the generators $h^j_0(x_1), h^m_1(x_2),$ and $h^l_0(x_3)$.
\end{proof}

The following corollary is a direct consequence of the construction of the vector field $X'$ in the proof of Theorem \ref{teo_cancelamento}.

\begin{corollary}
\label{corollary_nature_number}
Let $(M, X)$ be a GGS pair satisfying the condition $\mathcal{H}$. If $x_1, x_2, x_3 \in \operatorname{Sing}(X)$ and $h^j_2(x_1)$, $h^l_1(x_2)$, $h^m_2(x_3)$ or $h^j_0(x_1)$, $h^l_1(x_2)$, $h^m_0(x_3)$ are homotopically dynamically cancelled by Theorem \ref{teo_cancelamento}, then for resulting singularity $x'$ it holds the following equality
\begin{align*}
    \sum_{k=0}^2\eta_k(x')&=\left( \sum_{k=0}^2\eta_k(x_1)+\sum_{k=0}^2\eta_k(x_2)+\sum_{k=0}^2\eta_k(x_3)\right)-2.
\end{align*}
Furthermore, if $h^j_2(x_1)$, $h^l_1(x_2)$, $h^m_2(x_3)$ are cancelled, then
\begin{align*}
\begin{cases}
\eta_2(x')=&\eta_2(x_1)+\eta_2(x_2)+\eta_2(x_3)-1,\\
\eta_1(x')=&\eta_1(x_1)+\eta_1(x_2)+\eta_1(x_3)-1,\\
\eta_0(x')=&\eta_0(x_1)+\eta_0(x_2)+\eta_0(x_3).\\
\end{cases}
\end{align*}
And if $h^j_0(x_1)$, $h^l_1(x_2)$, $h^m_0(x_3)$ are cancelled, then
\begin{align*}
\begin{cases}
\eta_2(x')=&\eta_2(x_1)+\eta_2(x_2)+\eta_2(x_3),\\
\eta_1(x')=&\eta_1(x_1)+\eta_1(x_2)+\eta_1(x_3)-1,\\
\eta_0(x')=&\eta_0(x_1)+\eta_0(x_2)+\eta_0(x_3)-1.\\
\end{cases}
\end{align*}
\end{corollary}

\begin{remark}
\label{obs_naturezas_canceladas}
The singularity obtained via a cancellation is not necessarily of repelling nature. In the case where $x_2$ is a singularity with $\eta_1(x_2)> 1$, only one of its associated saddle generators is cancelled, while another remains in the resulting singularity $x'$. An analogous situation occurs in the attracting case. For instance, in Example \ref{ex_cancelamento}, during the fourth cancellation, the generators  $h^1_0(x'_6)$, $h^2_1(x_5)$, $h^1_0(x_8)$ are cancelled, resulting in the GGS singularity $x'_5$, which is of  double type and has saddle attracting nature. This outcome arises precisely because $x_5$ is a double singularity of saddle-saddle nature.
\end{remark}

\begin{remark}
\label{obs_selas_misturadas}
In Example \ref{ex_cancelamento}, during the fourth cancellation, we could alternatively apply Theorem \ref{teo_cancelamento} to cancel the generators $h^1_0(x_{10})$, $h^1_1(x_5)$, and $h^1_0(x_{12})$ instead of $h^1_0(x'_6)$, $h^2_1(x_5)$, and $h^1_0(x_8)$. Following the argument in the proof of Theorem \ref{teo_cancelamento}, this would result in a new singularity $x'_5$ of type $\mathcal{W}_2\mathcal{D}_2$  with saddle-attracting nature. For simplicity, we do not include such vector fields in the class considered in this paper. In general, these mixed saddle  may arise in the presence of double-crossing singularities of nature $ss_u$ (resp. $ss_s$) and triple-crossing singularities of nature $ssa$ (resp. $ssr$). Nevertheless, all definitions and results developed in this work, including the GGS chain complex, can be suitably extended to incorporate these cases. Furthermore, mixed saddle vector fields can be immediately canceled, yielding a GGS manifold.
\end{remark}

\begin{definition}
Define the \textit{singular number} $\#s(p)$, of a singularity $p$ as the number of folds plus the number of cone sheets minus one, computed in an isolating neighborhood of the singularity $p$. When convenient, we use the notation $\#s(\mathcal{P})$, for a singularity $p \in M(\mathcal{P})$.
\end{definition}

\begin{example}
We present below a table with the singular numbers of the GGS singularities.
\begin{center}
\begin{tabular}{|c|c|c|c|}
\hline
\cellcolor{ForestGreen!40} \textit{\ Type \ } & \cellcolor{ForestGreen!40} \textit{ \ $\#s(\mathcal{P})$ \ } & \cellcolor{ForestGreen!40} \textit{ \ Type \ } & \cellcolor{ForestGreen!40} \textit{ \ $\#s(\mathcal{P})$ \ } \\
\hline
\cellcolor{ForestGreen!40} $\mathcal{R}$ & $0$ &  \cellcolor{ForestGreen!40} $\mathcal{C}_n$ & $n-1$\\
\hline
\cellcolor{ForestGreen!40} $\mathcal{D}_{n}$ & $2(n-1)$ &  \cellcolor{ForestGreen!40} $\mathcal{D}_{n}\mathcal{T}_{2k+1}$ & $2(n-1)+6k$\\
\hline
\cellcolor{ForestGreen!40} $\mathcal{W}_{n}$ & $n-1$ & \cellcolor{ForestGreen!40} $\mathcal{P}\vee\mathcal{Q}$ & $\#s(\mathcal{P})+\#s(\mathcal{Q})+1$\\
\hline
\cellcolor{ForestGreen!40} $\mathcal{T}_{2k+1}$ & $6k$ & \cellcolor{ForestGreen!40} $\mathcal{W}_n\mathcal{Q}$ & $(n-1)+\#s(\mathcal{Q})$\\
\hline
\end{tabular}
\end{center}
\end{example}

The next theorem shows that the cancellations described in Theorem \ref{teo_cancelamento} simplify the dynamics without losing fundamental singular properties such as folds and cone sheets.

\begin{proposition}
\label{prop_herdeiros}
\textit{(Succession Proposition)} Let $(M, X)$ be a GGS pair satisfying the condition $\mathcal{H}$. If $x_1, x_2, x_3 \in \operatorname{Sing}(X)$ and $h^j_2(x_1)$, $h^l_1(x_2)$, $h^m_2(x_3)$ (resp. $h^j_0(x_1)$, $h^l_1(x_2)$, $h^m_0(x_3)$) are homotopically dynamically cancelled by Theorem \ref{teo_cancelamento}, resulting in the singularity $x'$, then:
\begin{align*}
\#s(x')=\#s(x_1)+\#s(x_2) +\#s(x_3).
\end{align*}
\end{proposition}

\begin{proof}
By hypothesis $h^j_2(x_1)$, $h^l_1(x_2)$, $h^m_2(x_3)$ are canceled by Theorem \ref{teo_cancelamento}. Keeping the notation from  the proof of Theorem \ref{teo_cancelamento}, we have:
\begin{align*}
    \delta = \gamma(x_3) * \gamma(u_2) * \gamma(x_2) * \gamma^{-1}(u_1) * \gamma(x_1),
\end{align*}
which is the path connecting $x_1$, $x_2$ and $x_3$, passing through the flow lines $u_1\in \mathcal{M}_{x_1x_2}$ and $u_2\in \mathcal{M}_{x_3x_2}$. Then the singularity $x'$ is obtained by homotopically deforming $\delta$ to a single point. Suppose that $x_1$ and $x_3$ are arbitrary basic singularities. The proof of the equality proceeds by considering the different types of $x_2$:

\begin{enumerate}
\item If $x_2$ is a regular saddle, $x_1\in M(\mathcal{P})$ and $x_3\in M(\mathcal{Q})$, where $\mathcal{P}, \mathcal{Q}\in\{ \mathcal{D}_n, \mathcal{W}_n, \mathcal{T}_{2k+1}, \mathcal{R}$\}, then, by homotopically deforming $\delta$ to $x'$, we have that $x'\in M(\mathcal{PQ})$ when $\mathcal{P}\neq \mathcal{Q}$ (the case $\mathcal{P}=\mathcal{Q}$ is treated in \cite{2021homotopical}). Thus,  
    \begin{align*}
        \#s(x')&= \#s(\mathcal{P}) + \#s(\mathcal{Q})\\
        &= \#s(x_1) + 0 + \#s(x_3)\\
        &= \#s(x_1) + \#s(x_2) + \#s(x_3).
    \end{align*}
    
\item If $x_2$ is a regular saddle, $x_1\in M(\mathcal{C}_n)$ and $x_2\in M(\mathcal{Q})$, where $\mathcal{Q}\in\{\mathcal{D}_m$, $\mathcal{W}_m$, $\mathcal{T}_{2k+1}\}$, then, $u_1\in\mathcal{M}_{x_1x_2}$ is a flow line connecting one of the cone sheets of $x_1\in M(\mathcal{C}_n)$ to the saddle $x_2$. Upon deforming $\delta$ to $x'$ this cone sheet  is identified with a sheet of $\mathcal{Q}$, resulting in $x' \in M(\mathcal{C}_{n-1}\vee\mathcal{Q})$. Then,
\begin{align*}
    \#s(x')&= \#s(\mathcal{C}_{n-1}) + 1+ \#s(\mathcal{Q})\ \text{\ \ \ ($+1$ comes from the wedge sum)}\\
    &= (n-2)+ 1 + \#s(x_3)\\
    &= (n-1)+ 0 + \#s(x_3)\\
    &= \#s(x_1) + \#s(x_2) +\#s(x_3).
\end{align*}

\item Let $x_2$ be a saddle of the type $\mathcal{C}_2$, $x_1\in M(\mathcal{P})$ and $x_3\in M(\mathcal{Q})$, where $\mathcal{P}, \mathcal{Q}\in\{\mathcal{C}_n, \mathcal{D}_n, \mathcal{W}_n, \mathcal{T}_{2k+1}, \mathcal{R}\}$. Then deforming $\delta$ to a point yields a singularity $x'\in M(\mathcal{P\vee Q})$, since $x_2$ is a cone type saddle. Thus:
\begin{align*}
    \#s(x')&= \#s(\mathcal{P}) + 1 + \#s(\mathcal{Q})\ \text{\ \ \ ($+1$ comes from the wedge sum)}\\
    &= \#s(x_1) + \#s(x_2) + \#s(x_3).
\end{align*}

\item Let $x_2$ be a saddle of the type $\mathcal{D}_2$, $x_1\in M(\mathcal{P})$ and $x_3\in M(\mathcal{Q})$, where $\mathcal{P}, \mathcal{Q}\in\{\mathcal{C}_n, \mathcal{D}_n, \mathcal{W}_n, \mathcal{T}_{2k+1}, \mathcal{R}\}$. Then, the resulting singularity satisfies $x'\in M(\mathcal{P}\mathcal{D}_2\mathcal{Q})$ and:
\begin{align*}
    \#s(x')&= \#s(\mathcal{P}) + 2 + \#s(\mathcal{Q}) \ \text{\ \ ($+2$ comes from the double fold)}\\
    &= \#s(x_1) + \#s(x_2) + \#s(x_3).
\end{align*}

\item  Let $x_2$ be a saddle of the type $\mathcal{W}_2$, $x_1\in M(\mathcal{P})$ and $x_3\in M(\mathcal{Q})$, where $\mathcal{P},\mathcal{Q}\in\{\mathcal{C}_n, \mathcal{D}_n, \mathcal{W}_n, \mathcal{T}_{2k+1}, \mathcal{R}\}$. Then the resulting singularity satisfies $x'\in M(\mathcal{P}\mathcal{W}_2\mathcal{Q})$ and we have:
\begin{align*}
    \#s(x')&= \#s(\mathcal{P}) + 1 + \#s(\mathcal{Q}) \text{ \ \ ($+1$ comes from the cross-cap fold)}\\
    &= \#s(x_1) + \#s(x_2) + \#s(x_3).
\end{align*}

\item  Let $x_2$ be a saddle of the type $\mathcal{T}_{3}$, $x_1\in M(\mathcal{P})$ and $x_3\in M(\mathcal{Q})$, where $\mathcal{P}, \mathcal{Q}\in\{\mathcal{C}_n, \mathcal{D}_n, \mathcal{W}_n, \mathcal{T}_{2k+1}, \mathcal{R}\}$. Then the resulting singularity satisfies  $x'\in M(\mathcal{P}\mathcal{T}_{2k+1}\mathcal{Q})$ and we have:
    \begin{align*}
    \#s(x')&= \#s(\mathcal{P}) + 6 + \#s(\mathcal{Q}) \text{\ \ ($+6$ comes from the triple folds)}\\
    &= \#s(x_1) + \#s(x_2) + \#s(x_3).
    \end{align*}
\end{enumerate}

\begin{comment}
Now, if $x_1$ and $x_3$ are mixed singularities the proof is analogous, it suffices to consider all the saddle possibilities as we did above. Therefore, given any $x_1, x_2, x_3\in \operatorname{Sing}(X)$ such that $h^j_2(x_1)$, $h^l_1(x_2)$, $h^m_2(x_3)$ (resp. $h^j_0(x_1), h^l_1(x_2), h^m_0(x_3)$) are dynamically homotopically cancelled, then:    
\end{comment}

Now, if $x_1$ and $x_3$ are mixed singularities the proof proceeds analogously; it suffices to consider all saddle configurations as done above. Therefore, given any $x_1, x_2, x_3\in \operatorname{Sing}(X)$ such that $h^j_2(x_1)$, $h^l_1(x_2)$, $h^m_2(x_3)$ (resp. $h^j_0(x_1), h^l_1(x_2), h^m_0(x_3)$) are dynamically homotopically cancelled  by Theorem \ref{teo_cancelamento}, we obtain:
\begin{align*}
    \#s(x')=\#s(x_1)+\#s(x_2)+\#s(x_3).
\end{align*}
\end{proof}

\section{Spectral Sequence Approach to Global Homotopical Dynamical Cancellations}
\label{sec:SSAproach_cancellation}

A relevant application of the GGS chain complex, introduced in Section \ref{sec:complex}, is the derivation of homotopical cancellation theorem from its associated spectral sequence. The purpose of this section is to present this application and the resulting theorem.  

\subsection{Basics concepts}

Jean Leray introduced spectral sequences as a tool for computing homology and cohomology in his seminal works \cite{leray1946lanneau, leray1946structure, leray1950anneau}. The wide applicability and strength of the theory were subsequently showcased in the foundational contributions of Jean-Pierre Serre \cite{serre1951homologie, serre1953groupes}. A didactic presentation of Serre's methods can be found in \cite[Chapter~9]{Sp}. More recently, spectral sequences have been a key tool for extracting dynamical information, as demonstrated in several works \cite{BLMdRS, bertolim2017algebraic, LMdRS, 2021homotopical}. 

The formal definition is as follows.

\begin{definition}
Let $R$ be a principal ideal domain. A $k$-\textit{spectral sequence} $E$ over $R$ is a sequence $\{E^{r},d^{ r}\}_{r \geq k }$, such that: 
\begin{enumerate}
\item[$(i)$] $E^{r}$ is a bigraded module over $R$, that is, an indexed collection of $R$-modules $E^{r}_{p,q}$, for all $p,q\in\mathbb{Z}$;

\item[$(ii)$] $d^{r}$ is a differential of bidegree $(-r,r-1)$ on $E^{r}$, i.e., an indexed collection of homomorphisms $d^{r}:E^{r}_{p,q}\rightarrow E^{r}_{p-r,q+r-1}$, for all $p,q\in\mathbb{Z}$, and $(d^{r})^{2}=0$;

\item[$(iii)$] for all $r\geq k$, there exists an isomorphism $H(E^r)\cong E^{r+1}$, where
\begin{align*}
H_{p,q}(E^r)&=\frac{\operatorname{Ker}\left( d^r:E^{r}_{p,q}\to E^r_{p-r,q+r-1} \right)}{\operatorname{Im}\left( d^r:E^{r}_{p+r,q-r+1}\to E^r_{p,q}\right)}.
\end{align*}
\end{enumerate}   
\end{definition}

We now define some important notions arising from the spectral sequence, for which we use the following notation:
\begin{align*}
    Z^k_{p,q}&\coloneqq \operatorname{Ker} \left(d^k_{p,q}: E^k_{p,q} \rightarrow E^k_{p-k,q+k-1}\right),\\
    B^k_{p,q}&\coloneqq \operatorname{Im} \left(d^k_{p+k,q-k+1}:E^k_{p+k,q-k+1} \rightarrow E^k_{p,q}\right).
\end{align*}
Then $B^k\subset Z^k$, $E^{k+1}=Z^k/B^k$ and we have for $r\geq k$:
\begin{align*}
 B^k\subset B^{k+1}\subset \ldots \subset B^r \subset \ldots \subset Z^r\subset \ldots\subset Z^{k+1}\subset Z^k. 
\end{align*}
Consider the bigraded modules: 
$$
Z^{\infty} \coloneqq \bigcap_r Z^r \quad \text{and} \quad 
B^{\infty} \coloneqq \bigcup_r B^r.
$$
This notation leads to the following definitions:
\begin{itemize}
    \item The bigraded module $E^{\infty} \coloneqq Z^{\infty}/B^{\infty}$ is called the \textit{limit of the spectral sequence}; 
    \item A spectral sequence $E=\{E^{r},d^{r}\}_{r\geq k}$ is \textit{convergent} if given $p,q$  there is  $r(p,q)\geq k$ such that for all $r\geq r(p,q)$, $d^r_{p,q}: E^{r}_{p,q} \rightarrow E^{r}_{p-r,q+r-1}$ is trivial;
    %\item A spectral sequence $E=\{E^{r},d^{r}\}_{r\geq k}$ is \textit{convergent in the strong sense} if given $p,q \in \mathbb{Z}$ there is  $r(p,q)\geq k$ such that $E^{r}_{p,q}\cong E^{\infty}_{p,q}$, for all $r\geq r(p,q)$.
\end{itemize}

We now introduce the concept of a filtration, which will be crucial for the results that follow.

\begin{definition}
Let $(C_{\ast},\partial_\ast)$ be a chain complex. An \textit{increasing filtration} $F$ on $(C_\ast, \partial_\ast )$ is a sequence of submodules $F_pC_{\ast}$ of $C_{\ast}$ such that $F_pC_{\ast}\subset F_{p+1}C_{\ast}$, for all integer $p$, and the filtration is compatible with the gradation of $(C_{\ast},\partial_{\ast})$, that is, $F_pC_{\ast}$ is a chain subcomplex of $C_{\ast}$ graded by $\{F_{p}C_{q}\}_{q}$.    
\end{definition}

From the preceding definition, we derive the following concepts:
\begin{itemize}
    \item A filtration $F$ on $C_{\ast}$ is called \textit{convergent} if $\cap_{p}F_{p}C_{\ast}=0$ and $\cup_{p}F_{p}C_{\ast} = C_{\ast}$; 
    \item It is called \textit{finite} if there are $p,p' \in \mathbb{Z}$ such that $F_{p}C_{\ast}=0$ and $F_{p'}C_{\ast}=C_{\ast}$;
    \item Also, it is said to be \textit{bounded below} if for any $q$ there is $p(q)$ such that $F_{p(q)}C_{q}=0$.
\end{itemize}
   
Given a filtration $F$ on $(C_{\ast}, \partial_{\ast})$, the \textit{associated bigraded module} $G(C)$ is defined as
\begin{align*}
G(C)_{p,q} \coloneqq \dfrac{F_{p}C_{p+q}}{F_{p-1}C_{p+q}}.
\end{align*}
A filtration $F$ on $C_{\ast}$ induces a filtration $F$ on $H_{\ast}(C_{\ast})$ defined by:
\begin{align*}
    F_{p}H_{\ast}(C) \coloneqq \operatorname{Im} \left( H_{\ast}(F_{p}C)\rightarrow H_{\ast}(C) \right).
\end{align*}
If the filtration $F$ on $C_{\ast}$ is convergent and bounded below then the same holds for the induced filtration on $H_{\ast}(C_{\ast})$.

More specifically, our results rely on a \textit{finest filtration} on a finite chain complex $(C_{\ast}, \partial_{\ast})$ where each $C_{k}$ is finitely generated. We order the collection of all basis elements of $C_\ast$ into a single sequence, the \textit{finest filtration} $F$ on $C_{\ast}$ is then defined on each graded piece.  One can reorder the set of the generators  of $C_{\ast}$ as $\{h_{0}^{1},\cdots, h_{0}^{\ell_{0}}, h_{1}^{\ell_{0}+ 1},\cdots, h_{1}^{\ell_{1}}, \cdots, h_{k}^{\ell_{k-1}+1},\cdots, h_{k}^{\ell_{k}}, \cdots \},$ where $\ell_{k}=c_{0}+\cdots + c_{k}$ {\footnote{Where $c_i$ denotes the rank of $C_i$, for each $i$.}}.  Let $F$ be a \textit{finest filtration} on $C_{\ast}$ defined by 
\begin{align*}
 F_{p}C_{k} = \displaystyle\bigoplus_{h^{\ell}_{k}, \ \ell \leq p+1} \mathbb{Z} \langle h^{\ell}_{k} \rangle .    
\end{align*}
The GGS chain complex associated with $(M, X)$ a GGS pair satisfying the condition $\mathcal{H}$ admits a finest filtration that is convergent and bounded below. Consequently, we can use the following theorem:
 
\begin{theorem}[Spanier,~\cite{Sp}]
Let $F$ be a convergent and bounded below filtration on a chain complex $(C_{\ast}, \partial_{\ast})$. There is a convergent  spectral sequence with 
$$E^{0}_{p,q} = \dfrac{F_{p}C_{p+q}}{F_{p-1}C_{p+q}} =G(C)_{p,q}  \qquad \text{and} \qquad E^{1}_{p,q} \cong H_{p+q}\left(\dfrac{F_{p}C_{p+q}}{F_{p-1}C_{p+q}}\right)$$
and $E^{\infty}$ is isomorphic to the bigraded module $GH_{\ast}(C)$ associated to the induced filtration on $H_ {\ast}(C)$. 
\end{theorem}

The spectral sequence associated to $(C_{\ast}, \partial_{\ast})$ with this finest filtration has a special property: the only $q$ for which $E_{p,q}^r$ is non-zero is $q=k-p$, where $k$ is the index of the chain in $F_p C \setminus F_{p-1} C$. Hence, in this case, we omit reference to $q$.  

\subsection{Spectral Sequences Analysis for GGS chain complex}

This section establishes a key property of the GGS chain complex. After unfolding the associated spectral sequence, we can perform homotopical dynamical cancellations in the sense of Section \ref{sec:detecting_dynamical_homotopical_cancellation}. Similar results were previously presented in \cite{2021homotopical}. Since the chain complex in Definition \ref{def:GS_chain} generalizes the GS chain complex studied in \cite{2021homotopical}, it is natural to expect analogous results under our broader assumptions. Indeed,  such results are obtained.

A primary result of our spectral sequence analysis is the construction of a core flow,  defined as follows.

\begin{definition}
Let $(M, X)$ be a GGS pair satisfying the condition $\mathcal{H}$. If there is no triple of singularities in $\operatorname{Sing}(X)$ that can be homotopically cancelled, we say that the  $X$ is a \textit{core GGS vector field}, or that the induced flow $\varphi_X$ is a \textit{core flow}.
\end{definition}

A classical example of a core flow in the smooth case is given by the height function on a sphere. In Figures \ref{fig_intr_01} and \ref{fig_intr_02} we show core flows on a GGS manifold after cancellation.

Theorem~\ref{teo_bijecao_sequencia} below is central to our construction. It not only ensures the existence of a core flow but also establishes that the intermediate steps of our procedure yield a family of GGS flows. Moreover, Theorem~\ref{teo_bijecao_sequencia} highlights the strong relationship between our chain complex and the associated spectral sequence which ``translates'' all dynamical cancellations into algebraic terms. In other words, there is a one-to-one correspondence between homotopical dynamical cancellations and algebraic cancellations.  By analysing the unfolding of the spectral sequence, one  sees that whenever a differential $d^r_{p}: E^r_p \rightarrow E_p^{r-1}$ is an isomorphism, then at the next stage of the spectral sequence, one has an \textit{algebraic cancellation} $E^{r+1}_p=E^{r+1}_{q}=0$.

\begin{theorem}
\label{teo_bijecao_sequencia}
Let $(C_*(M,X; \Z), \triangle_*)$ be the GGS chain complex associated with a GGS pair $(M, X)$ satisfying condition $\mathcal{H}$ and assume that $M$ admits an orientable morsification. Let $(E^r, d^r)$ be the spectral sequence associated with a finest filtration $F$ on $(C_*(M,X; \Z), \triangle_*)$. Then: 

\begin{enumerate}
    \item[$(a)$] the algebraic cancellations of the modules $E^r$  are in one-to-one correspondence with the homotopical dynamical cancellations involving  the natures of the singularities of $X$;
    
    \item[$(b)$] the order in which these homotopical cancellations occur corresponds to the increasing values of $r$ in the filtration $F$;

    \item[$(c)$] there exists a family of GGS flows $\{\varphi^1=\varphi_X, \varphi^2, \cdots, \varphi^{\omega}\}$, where  $\varphi^{r+1}$ is obtained from $\varphi^r$ by a collision through a homotopical dynamical cancellation and $\varphi^{\omega}$ is a core GGS flow.
\end{enumerate}
\end{theorem}

\begin{proof}
Throughout this proof, we use the notation $n(h^j_k(x_1), h^m_{k-1}(x_2); \varphi)$ to denote the intersection number between the generators $h^j_k(x_1)$ and $h^m_{k-1}(x_2)$ of the GGS complex with respect to the flow $\varphi$. We denote the original GGS flow by $\varphi^1=\varphi_X$.

The proof relies on two algorithms: Spectral Sequence Sweeping Algorithm (SSSA) and Row Cancellation Algorithm (RCA), both of which take as input a totally unimodular upper triangular matrix $\Delta$ satisfying $\Delta^2=0$. Each algorithm produces a family of matrices obtained by sweeping the diagonal of $\Delta$ and performing changes of basis (in different orders). 
The SSSA, introduced in \cite{CORNEA_DEREZENDE_DASILVEIRA_2010},  recovers the modules and differentials of a spectral sequence arising from a finest filtration on a finite chain complex over $\mathbb{Z}$. Further details can be found in \cite{CORNEA_DEREZENDE_DASILVEIRA_2010}.
The RCA, introduced in \cite{BLMdRS, bertolim2017algebraic}, interprets the SSSA through the lens of homotopical dynamical cancellations.  More specifically, it performs changes of basis in a different order, reflecting the birth and death of connections that occur during cancellations.

Now, let $(E^r, d^r)$ be the spectral sequence associated with a finest filtration $F$ on $(C_*(M,X; \Z), \triangle_*)$. 

Consider the case $r=1$. Given a non-zero differential $d^1_{p,q}$ of the spectral sequence, it corresponds to multiplication by a primary pivot detected by the SSSA when applied to $\Delta$, namely $\Delta^1(p,p+1) = \pm 1$ (i.e., a non-zero entry of the matrix of the boundary operator $\triangle$). Therefore, there exist singularities $x_1, x_2 \in \operatorname{Sing}(X)$ with consecutive generators $h_k^i(x_1)$, $h_{k-1}^{j}(x_2)$ such that
\begin{align*}
    n(h_k^i(x_1), h_{k-1}^{j}(x_2); \varphi^1)=\pm 1.
\end{align*}
By Theorem \ref{teo_cancelamento}, these generators  admit a homotopical dynamical cancellation, that is, there exists a  singularity $x_3\in \operatorname{Sing}(X)$ such that the isolated invariant set
\begin{align*}
    S=\{ x_1, x_2, x_3, u_1, u_2 \},
\end{align*}
where $u_1$ and $u_2$ are the flow lines  connecting the three singularities, admits a homotopical dynamical cancellation to a single GGS singularity, $\{x'\}=S'$. Moreover, there exists a new GGS pair $(M', X')$ with a GGS flow $\varphi_{X'}=\varphi'$ that coincides with $\varphi^1$ outside a neighborhood of $S'$.

Next, we analyze the relations between the intersection numbers of generators with respect to the flows $\varphi^1$ and $\varphi'$.

By hypothesis, the morsified manifold $M$ is orientable hence, without loss of generality, we can assume that the unstable manifolds of the repelling singularities of the morsified flow $\varphi_{\widetilde{X}}$ all share the same orientation. Thus, if $h^j_1(x)$ is a saddle singularity in the morsification, then the flow lines in the set
\begin{align*}
    W^s(h^j_1(x))\setminus\{ h^j_1(x) \},
\end{align*}
have opposite characteristic signs. The same holds for $W^u(h^j_1(x))\setminus\{ h^j_1(x) \}$.

We have two cases to consider: $k=2$ and $k=1$.
\begin{enumerate}
    \item[$(1)$] If $k=2$, we have that if $n(h^i_2(x_1), h^{j}_{1}(x_2))=\pm 1$, then
$\mathcal{M}_{h^i_2(x) h^{j}_1(y)}=\{ u_1 \}$ and $\mathcal{M}_{h^l_2(x_3)h^{j}_1(x_2)}=\{ u_2 \}$. If we cancel the generators $h^i_2(x_1)$ with $h^{j}_1(x_2)$, then every saddle generator $h^p_1(w)$ that connects either to $h^i_2(x_1)$ or to $h^l_2(x_3)$ in $\varphi^1$ will connect to a generator $h^t_2(x')$ in $\varphi'$. The existence of the generator $h^t_2(x')$ is guaranteed by Corollary \ref{corollary_nature_number}. Since the flow lines other than $u_1$ and $u_2$ remain unchanged, we obtain the following relation between the intersection numbers:
\begin{align*}
    n(h^t_2(x'), h^p_1(w); \varphi')=n(h^i_2(x_1), h^p_1(w); \varphi^1) + n(h^l_2(x_3), h^p_1(w); \varphi^1).
\end{align*}

\item[$(2)$] If $k=1$, we have $n(h^i_1(x_1), h^{j}_{0}(x_2))=\pm 1$, then $\mathcal{M}_{h^i_1(x_1) h^{j}_0(x_2)}=\{ u_1 \}$ and $\mathcal{M}_{h^{i}_1(x_1)h^l_0(x_3)}=\{ u_2 \}$. By Corollary \ref{corollary_nature_number}, cancelling the  generators, $h^i_1(x_1)$ with $h^{j}_0(x_2)$, implies that every saddle generator $h^p_1(w)$ that connects to $h^{j}_0(x_2)$ or to $h^{l}_0(x_3)$ in $\varphi^1$ will instead connect to $h^t_0(x')$ in $\varphi'$. Since the flow lines other than $u_1$ and $u_2$ remain unchanged, we obtain the following relation between the intersection numbers:
$$
n(h^p_1(w), h^t_0(x'); \varphi')=n(h^p_1(w), h^{j}_0(x_2); \varphi^1) + n(h^p_1(w), h^l_0(x_3); \varphi^1).$$
\end{enumerate}

By performing all cancellations determined by the non-zero differentials of the spectral sequence on the first diagonal, we obtain a GGS flow $\varphi_2$. Interestingly, the matrix associated to the GGS boundary operator for $\varphi_2$ is obtained from the matrix of the flow $\varphi_1$ by applying the RCA algorithm and deleting the rows and columns corresponding to the primary pivots on the first diagonal.

Now, for $r=2$, we consider differentials $d^2_{p,q}$ of the spectral sequence. The non-zero differentials correspond to primary pivots on the second diagonal of the SSSA; these same  entries appear as primary pivots on the second diagonal of the RCA, and represent intersection numbers between generators with respect to the flow $\varphi_2$. By Theorem \ref{teo_cancelamento}, we can cancel them exactly as was done in the initial process.

Since the set of generators of the GGS chain complex is finite, we can repeat this process until we obtain a zero matrix on the RCA. At this final stage, there are no longer  non-zero intersection numbers, and no further cancellations are possible, that is, we have reached a core flow. Thus we obtain  a family of GGS flows $ \{ \varphi^1=\varphi_X, \varphi^2, \ldots, \varphi^\omega \},$ with the desired properties.
\end{proof}

Now, we present an application of Theorem \ref{teo_bijecao_sequencia} to detect cancellations of singularities until we obtain a core flow. We also describe the matrix associated with the GGS chain complex at each stage of the cancellation.

\begin{example}
\label{ex_cancelamento}
Consider the GGS pair $(M,X)$ presented in Example \ref{exemplo_02}. The matrix associated with the GGS chain complex is given in Table \ref{table:1}. 

\begin{table}[!ht]
\centering
{
\resizebox{\textwidth}{!}{
\begin{tabular}{|c|c|c|c|c|c|c|c|c|c|c|c|c|c|c| } 
\hline
& \cellcolor{ForestGreen!40} $h^1_0(x_{12})$ & \cellcolor{ForestGreen!40} $h^1_0(x_8)$ & \cellcolor{ForestGreen!40} $h^1_0(x_{11})$ & \cellcolor{ForestGreen!40} $h^1_0(x_{7})$ & \cellcolor{ForestGreen!40} $h^1_0(x_{6})$ & \cellcolor{ForestGreen!40} $h^1_0(x_{10})$ & \cellcolor{ForestGreen!40} $h^1_1(x_9)$ & \cellcolor{ForestGreen!40} $h^1_1(x_{4})$ & \cellcolor{ForestGreen!40} $h^1_1(x_{5})$ & \cellcolor{ForestGreen!40} $h^2_1(x_5)$ & \cellcolor{ForestGreen!40} $h^1_1(x_{3})$ & \cellcolor{ForestGreen!40} $h^1_2(x_{1})$ & \cellcolor{ForestGreen!40} $h^1_2(x_{2})$ & \cellcolor{ForestGreen!40} $h^2_2(x_{2})$\\
\hline
\cellcolor{ForestGreen!40} $h^1_0(x_{12})$ & $0$ & $0$ & $0$ & $0$ & $0$ & $0$ & \cellcolor{WildStrawberry!40}$0$ &\cellcolor{WildStrawberry!40} $0$ & \cellcolor{WildStrawberry!40}$0$ & \cellcolor{WildStrawberry!40}$+1$ & \cellcolor{WildStrawberry!40}$0$ & $0$ & $0$ & $0$\\
\hline
\cellcolor{ForestGreen!40} $h^1_0(x_8)$ & $0$ & $0$ & $0$ & $0$ & $0$ & $0$ & \cellcolor{WildStrawberry!40}$0$ &\cellcolor{WildStrawberry!40}$0$ & \cellcolor{WildStrawberry!40}$+1$ & \cellcolor{WildStrawberry!40}$0$ & \cellcolor{WildStrawberry!40}$0$ & $0$ & $0$ & $0$\\
\hline
\cellcolor{ForestGreen!40} $h^1_0(x_{11})$ & $0$ & $0$ & $0$ & $0$ & $0$ & $0$ & \cellcolor{WildStrawberry!40}$-1$ & \cellcolor{WildStrawberry!40}$0$ & \cellcolor{WildStrawberry!40}$0$ & \cellcolor{WildStrawberry!40}$-1$ & \cellcolor{WildStrawberry!40}$+1$ & $0$ & $0$ & $0$\\
\hline
\cellcolor{ForestGreen!40} $h^1_0(x_{7})$ & $0$ & $0$ & $0$ & $0$ & $0$ & $0$ & \cellcolor{WildStrawberry!40}$0$ & \cellcolor{WildStrawberry!40}$+1$ & \cellcolor{WildStrawberry!40}$-1$ & \cellcolor{WildStrawberry!40}$0$ & \cellcolor{WildStrawberry!40}$0$ & $0$ & $0$ & $0$\\
\hline
\cellcolor{ForestGreen!40} $h^1_0(x_{6})$ & $0$ & $0$ & $0$ & $0$ & $0$ & $0$ & \cellcolor{WildStrawberry!40} $0$ & \cellcolor{WildStrawberry!40}$-1$ & \cellcolor{WildStrawberry!40}$0$ & \cellcolor{WildStrawberry!40}$0$ & \cellcolor{WildStrawberry!40}$0$ & $0$ & $0$ & $0$\\
\hline
\cellcolor{ForestGreen!40} $h^1_0(x_{10})$ & $0$ & $0$ & $0$ & $0$ & $0$ & $0$ & \cellcolor{WildStrawberry!40}$+1$ & \cellcolor{WildStrawberry!40}$0$ & \cellcolor{WildStrawberry!40}$0$ & \cellcolor{WildStrawberry!40}$0$ & \cellcolor{WildStrawberry!40}$-1$ & $0$ & $0$ & $0$\\
\hline
\cellcolor{ForestGreen!40} $h^1_1(x_{9})$ & $0$ & $0$ & $0$ & $0$ & $0$ & $0$ & $0$ & $0$ & $0$ & $0$ & $0$ & \cellcolor{MidnightBlue!40}$+1$ & \cellcolor{MidnightBlue!40}$0$ & \cellcolor{MidnightBlue!40}$-1$\\
\hline
\cellcolor{ForestGreen!40} $h^1_1(x_{4})$ & $0$ & $0$ & $0$ & $0$ & $0$ & $0$ & $0$ & $0$ & $0$ & $0$ & $0$ & \cellcolor{MidnightBlue!40}$0$ & \cellcolor{MidnightBlue!40}$0$ & \cellcolor{MidnightBlue!40}$0$\\
\hline
\cellcolor{ForestGreen!40} $h^1_1(x_{5})$ & $0$ & $0$ & $0$ & $0$ & $0$ & $0$ & $0$ & $0$ & $0$ & $0$ & $0$ & \cellcolor{MidnightBlue!40}$0$ & \cellcolor{MidnightBlue!40}$0$ & \cellcolor{MidnightBlue!40}$0$\\
\hline
\cellcolor{ForestGreen!40} $h^2_1(x_{5})$ & $0$ & $0$ & $0$ & $0$ & $0$ & $0$ & $0$ & $0$ & $0$ & $0$ & $0$ & \cellcolor{MidnightBlue!40}$0$ & \cellcolor{MidnightBlue!40}$0$ & \cellcolor{MidnightBlue!40}$0$\\
\hline
\cellcolor{ForestGreen!40} $h^1_1(x_{3})$ & $0$ & $0$ & $0$ & $0$ & $0$ & $0$ & $0$ & $0$ & $0$ & $0$ & $0$ & \cellcolor{MidnightBlue!40}\tikz[baseline=(current bounding box.center)]\node[draw, circle, fill=white, inner sep=0.5pt]{+1}; & \cellcolor{MidnightBlue!40}$0$ & \cellcolor{MidnightBlue!40}$-1$\\
\hline
\cellcolor{ForestGreen!40} $h^1_2(x_{1})$ & $0$ & $0$ & $0$ & $0$ & $0$ & $0$ & $0$ & $0$ & $0$ & $0$ & $0$ & $0$ & $0$ & $0$\\
\hline
\cellcolor{ForestGreen!40} $h^1_2(x_{2})$ & $0$ & $0$ & $0$ & $0$ & $0$ & $0$ & $0$ & $0$ & $0$ & $0$ & $0$ & $0$ & $0$ & $0$\\
\hline
\cellcolor{ForestGreen!40} $h^2_2(x_{2})$ & $0$ & $0$ & $0$ & $0$ & $0$ & $0$ & $0$ & $0$ & $0$ & $0$ & $0$ & $0$ & $0$ & $0$\\
\hline
\end{tabular}
}
}
\caption{Matrix  associated with the  $(C_*(M,X;\Z), \triangle_*)$.}
\label{table:1}
\end{table}

Consider the finest filtration on
$(C_*(M,X;\Z), \triangle_*)$ induced by the following order of its generators
\[
\begin{aligned}
\{\, &h^1_0(x_{12}),\, h^1_0(x_{8}),\, h^1_0(x_{11}),\, h^1_0(x_{7}),\, 
       h^1_0(x_{6}),\, h^1_0(x_{10}),\, h^1_1(x_{9}), \\
    & h^1_1(x_{4}), \, h^1_1(x_{5}),\, h^2_1(x_{5}),\, h^1_1(x_{3}),\, 
      h^1_2(x_{1}),\, h^1_2(x_{2}),\, h^2_2(x_{2}) \,\}.
\end{aligned}
\]
The associated spectral sequence is shown in Figure \ref{fig_ss}. 

\begin{figure}[h!]
\centering
\begin{tikzpicture}[
    scale=1, transform shape,
    diff/.style={draw=blue, ->, >=Stealth, thick}
]
\matrix (m) [matrix of math nodes, 
             row sep=1em,
             column sep=-0.5em,
             nodes={anchor=center,font=\normalsize},
             % --- CORREÇÃO APLICADA AQUI ---
             row 1/.style={nodes={font=\tiny}}, % Define a fonte para a linha 1
             column 1/.style={nodes={font=\normalsize}}, % Define a fonte para a linha 1
             nodes in empty cells                    % Necessário para o estilo da linha
            ]
{
% --- Cabeçalho (agora sem \footnotesize manual) ---
& h_0^1(x_{12}) & h_0^1(x_{8}) & h_0^1(x_{11}) & h_0^1(x_{7}) & h_0^1(x_{6}) & h_0^1(x_{10}) & h_1^1(x_{9}) & h_1^1(x_{4}) & h_1^1(x_{5}) & h_1^2(x_{5}) & h_1^1(x_{3}) & h_2^1(x_{1}) & h_2^1(x_{2}) & h_2^2(x_{2}) \\
% --- Linha E^0 ---
E^0: & \mathbb{Z} & \mathbb{Z} & \mathbb{Z} & \mathbb{Z} & \mathbb{Z} & \mathbb{Z} & \mathbb{Z} & \mathbb{Z} & \mathbb{Z} & \mathbb{Z} & \mathbb{Z} & \mathbb{Z} & \mathbb{Z} & \mathbb{Z} \\
% --- Linha E^1 ---
E^1: & \mathbb{Z} & \mathbb{Z} & \mathbb{Z} & \mathbb{Z} & \mathbb{Z} & \mathbb{Z} & \mathbb{Z} & \mathbb{Z} & \mathbb{Z} & \mathbb{Z} & \mathbb{Z} & \mathbb{Z} & \mathbb{Z} & \mathbb{Z} \\
% --- Linha E^2 ---
E^2: & \mathbb{Z} & \mathbb{Z} & \mathbb{Z} & \mathbb{Z} & \mathbb{Z} & 0 & 0 & \mathbb{Z} & \mathbb{Z} & \mathbb{Z} & 0 & 0 & \mathbb{Z} & \mathbb{Z} \\
% --- Linha E^3 ---
E^3: & \mathbb{Z} & \mathbb{Z} & \mathbb{Z} & \mathbb{Z} & \mathbb{Z} & 0 & 0 & \mathbb{Z} & \mathbb{Z} & \mathbb{Z} & 0 & 0 & \mathbb{Z} & \mathbb{Z} \\
% --- Linha E^4 ---
E^4: & \mathbb{Z} & \mathbb{Z} & \mathbb{Z} & \mathbb{Z} & 0 & 0 & 0 & 0 & \mathbb{Z} & \mathbb{Z} & 0 & 0 & \mathbb{Z} & \mathbb{Z} \\
% --- Linha E^5 ---
E^5: & \mathbb{Z} & \mathbb{Z} & \mathbb{Z} & \mathbb{Z} & 0 & 0 & 0 & 0 & \mathbb{Z} & \mathbb{Z} & 0 & 0 & \mathbb{Z} & \mathbb{Z} \\
% --- Linha E^6 ---
E^6: & \mathbb{Z} & \mathbb{Z} & \mathbb{Z} & 0 & 0 & 0 & 0 & 0 & 0 & \mathbb{Z} & 0 & 0 & \mathbb{Z} & \mathbb{Z} \\
% --- Linha E^7 ---
E^7: & \mathbb{Z} & \mathbb{Z} & \mathbb{Z} & 0 & 0 & 0 & 0 & 0 & 0 & \mathbb{Z} & 0 & 0 & \mathbb{Z} & \mathbb{Z} \\
% --- Linha E^8 ---
E^8: & \mathbb{Z} & \mathbb{Z} & 0 & 0 & 0 & 0 & 0 & 0 & 0 & 0 & 0 & 0 & \mathbb{Z} & \mathbb{Z} \\
};
% --- Setas
\path [diff, bend left=40] (m-3-8) edge node[below, yshift=-1mm] {$d_6^1$} (m-3-7);
\path [diff, bend left=40] (m-3-13) edge node[below, yshift=-1mm] {$d_{11}^1$} (m-3-12);
\path [diff, bend left=30] (m-5-9) edge node[below] {$d_7^3$} (m-5-6);
\path [diff, bend left=20] (m-7-10) edge node[below] {$d_8^5$} (m-7-5);
\path [diff, bend left=15] (m-9-11) edge node[below] {$d_9^7$} (m-9-4);
\end{tikzpicture}
\caption{Spectral sequence diagram for Example \ref{ex_cancelamento}.}
\label{fig_ss}
\end{figure}
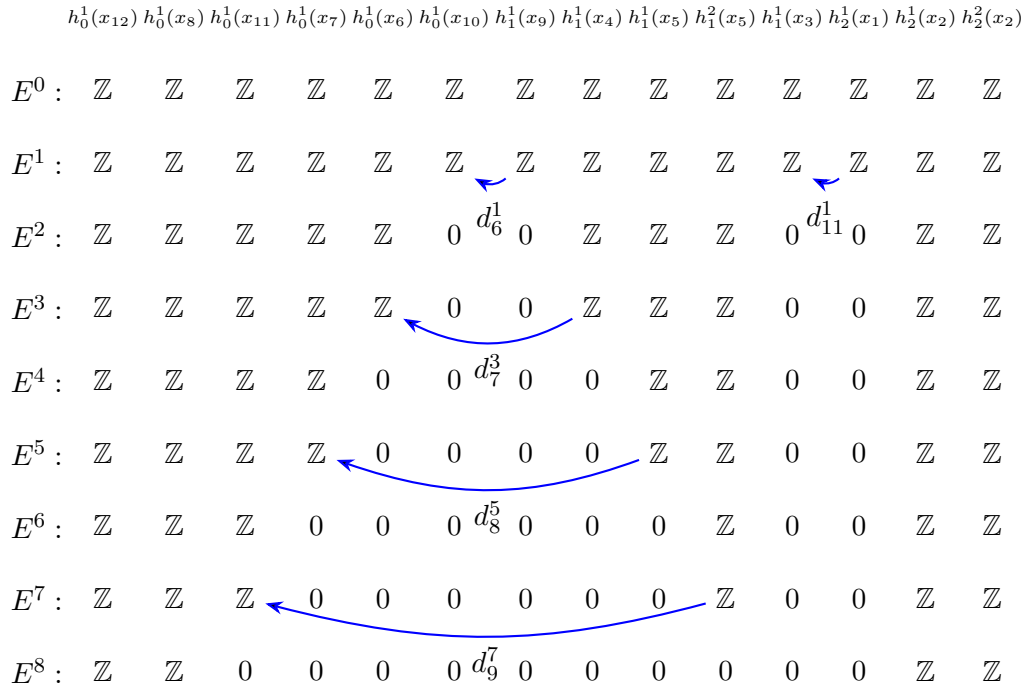

Note that the non-zero differentials of the spectral sequence determine the following algebraic cancellations:
\begin{itemize}
    \item $r=1$: $d^1_{11}:E^1_{11}\rightarrow E^1_{10}$ produces the algebraic cancellations $E^2_{11} = E^2_{10} =0$; $d^1_{6}:E^1_{6}\rightarrow E^1_{5}$ produces the algebraic cancellations $E^2_{6} = E^2_{5} =0$;
    \item  $r=3$: $d^3_{7}:E^3_{7}\rightarrow E^3_{4}$ produces the algebraic cancellations $E^4_{7} = E^4_{4} =0$; 
   \item  $r=5$: $d^5_{8}:E^5_{8}\rightarrow E^5_{3}$ produces the algebraic cancellations $E^6_{8} = E^6_{3} =0$;
    \item  $r=7$: $d^7_{9}:E^7_{9}\rightarrow E^7_{2}$ produces the algebraic cancellations $E^{8}_{9} = E^{8}_{2} =0$.
\end{itemize}
As we proved in Theorem \ref{teo_bijecao_sequencia}, these algebraic cancellations determine the homotopical dynamical cancellations of the singularities of the initial GGS flow, and by following the unfolding of the spectral sequence, we obtain a GGS core flow. In what follows, we explicitly describe how this correspondence occurs.

We begin by analysing $d^1_{11}$ which determines the homotopical dynamical cancellations of  $h^1_2(x_1)$ and  $h^1_1(x_3)$. In fact, note that $n(h^1_2(x_1), h^1_1(x_3))=+1$, hence, by the Cancellation Theorem \ref{teo_cancelamento}, we can cancel these generators together with $h^2_2(x_2)$, producing a new singularity of type $\mathcal{W}_2\mathcal{D}_2$ with repelling nature, as illustrated in $(M_1,X_1)$ in Figure \ref{fig:exemplo_final_5} (right). The matrix of the associated GGS chain complex is given in Table \ref{table:2}.

\begin{figure}[!ht]
\centering
\begin{overpic}[unit=1mm, scale=.25, width=10.5cm]{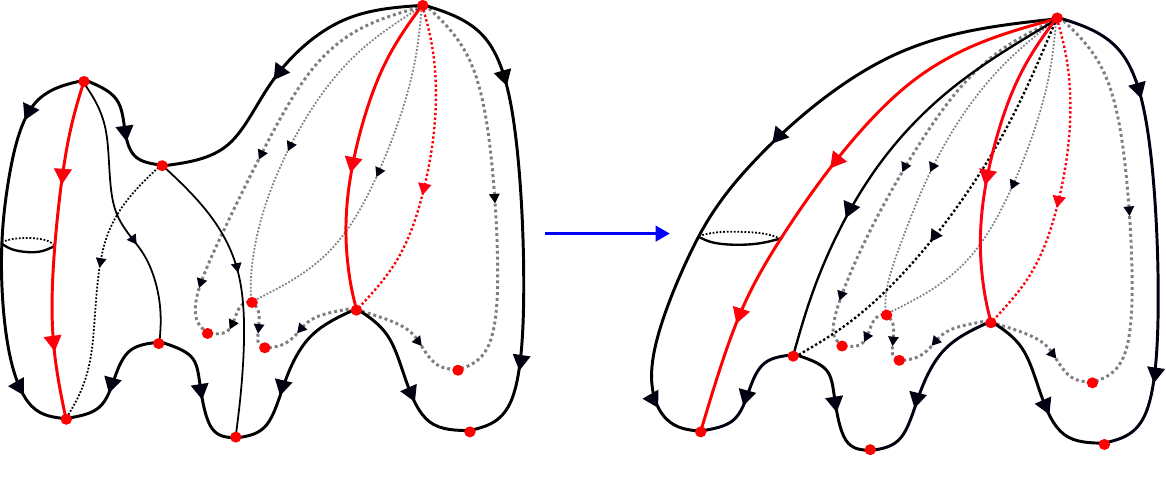}
\put(10,36){$(M,X)$}
\put(5.5,35.5){$x_1$}
\put(35,42){$x_2$}
\put(12.5,28.5){$x_3$}
\put(20.8,16.5){\footnotesize{$x_4$}}
\put(29.5,12){$x_5$}
\put(16.6,10.8){\footnotesize{$x_6$}}
\put(21,9.5){\footnotesize{$x_7$}}
\put(38,7.5){$x_8$}
\put(12,9.5){$x_9$}
\put(4,3){$x_{10}$}
\put(18,1.5){$x_{11}$}
\put(37.2,2){$x_{12}$}
%%%%%%
\put(60,36){$(M_1,X_1)$}
\put(90,41){$x'_1$}
\put(77,15.5){\footnotesize{$x_4$}}
\put(83.5,11){$x_5$}
\put(71,10){\footnotesize{$x_6$}}
\put(76.5,8.3){\footnotesize{$x_7$}}
\put(93,6){$x_8$}
\put(67,8.5){$x_9$}
\put(57,2.2){$x_{10}$}
\put(72,0.5){$x_{11}$}
\put(94,1){$x_{12}$}
%%%%
\end{overpic}
\caption{Initial GGS manifold $M$ and first cancellation.}
\label{fig:exemplo_final_5}
\end{figure}
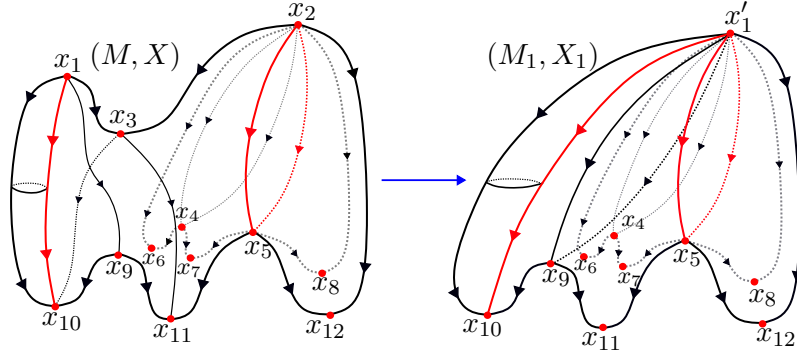

\begin{table}[!ht]
\centering
{
\begin{tabular}{|c|c|c|c|c|c|c|c|c|c|c|c|c| } 
\hline
& \cellcolor{ForestGreen!40} $h^1_0(x_{12})$ & \cellcolor{ForestGreen!40} $h^1_0(x_8)$ & \cellcolor{ForestGreen!40} $h^1_0(x_{11})$ & \cellcolor{ForestGreen!40} $h^1_0(x_{7})$ & \cellcolor{ForestGreen!40} $h^1_0(x_{6})$ & \cellcolor{ForestGreen!40} $h^1_0(x_{10})$  & \cellcolor{ForestGreen!40} $h^1_1(x_{9})$ & \cellcolor{ForestGreen!40} $h^1_1(x_{4})$ & \cellcolor{ForestGreen!40} $h^1_1(x_5)$ & \cellcolor{ForestGreen!40} $h^2_1(x_{5})$ & \cellcolor{ForestGreen!40} $h^1_2(x'_{1})$ & \cellcolor{ForestGreen!40} $h^2_2(x'_1)$\\
\hline
\cellcolor{ForestGreen!40}  $h^1_0(x_{12})$  & $0$ & $0$ & $0$ & $0$ & $0$ & $0$ & \cellcolor{WildStrawberry!40}$0$ & \cellcolor{WildStrawberry!40}$0$ & \cellcolor{WildStrawberry!40}$0$ & \cellcolor{WildStrawberry!40}$+1$ & $0$ &  $0$ \\
\hline
\cellcolor{ForestGreen!40}  $h^1_0(x_8)$ & $0$ & $0$ & $0$ & $0$ & $0$ & $0$  & \cellcolor{WildStrawberry!40}$0$ & \cellcolor{WildStrawberry!40}$0$ & \cellcolor{WildStrawberry!40}$+1$ & \cellcolor{WildStrawberry!40}$0$ & $0$ & $0$ \\
\hline
\cellcolor{ForestGreen!40}  $h^1_0(x_{11})$ & $0$ & $0$ & $0$ & $0$ & $0$ & $0$  & \cellcolor{WildStrawberry!40}$-1$ & \cellcolor{WildStrawberry!40}$0$ & \cellcolor{WildStrawberry!40}$0$ & \cellcolor{WildStrawberry!40}$-1$ & $0$ & $0$ \\
\hline
\cellcolor{ForestGreen!40}  $h^1_0(x_{7})$ & $0$ & $0$ & $0$ & $0$ & $0$ & $0$  & \cellcolor{WildStrawberry!40}$0$ & \cellcolor{WildStrawberry!40}$+1$ & \cellcolor{WildStrawberry!40}$-1$ & \cellcolor{WildStrawberry!40}$0$ & $0$ & $0$ \\
\hline
\cellcolor{ForestGreen!40}  $h^1_0(x_{6})$ & $0$ & $0$ & $0$ & $0$ & $0$ & $0$  & \cellcolor{WildStrawberry!40}$0$ & \cellcolor{WildStrawberry!40}$-1$ & \cellcolor{WildStrawberry!40}$0$ & \cellcolor{WildStrawberry!40}$0$ & $0$ & $0$ \\
\hline
\cellcolor{ForestGreen!40}  $h^1_0(x_{10})$ & $0$ & $0$ & $0$ & $0$ & $0$ & $0$ & \cellcolor{WildStrawberry!40}\tikz[baseline=(current bounding box.center)]\node[draw, circle, fill=white, inner sep=0.1pt]{+1};& \cellcolor{WildStrawberry!40}$0$ & \cellcolor{WildStrawberry!40}$0$ & \cellcolor{WildStrawberry!40}$0$ & $0$ & $0$ \\
\hline
\cellcolor{ForestGreen!40}  $h^1_1(x_{9})$ & $0$ & $0$ & $0$ & $0$ & $0$  & $0$ & $0$ & $0$ & $0$ & $0$ & \cellcolor{MidnightBlue!40}$0$ & \cellcolor{MidnightBlue!40}$0$ \\
\hline
\cellcolor{ForestGreen!40}  $h^1_1(x_{4})$ & $0$ & $0$ & $0$ & $0$ & $0$ & $0$ & $0$ & $0$ & $0$ & $0$ &\cellcolor{MidnightBlue!40} $0$ & \cellcolor{MidnightBlue!40}$0$ \\
\hline
\cellcolor{ForestGreen!40}  $h^1_1(x_{5})$ & $0$ & $0$ & $0$ & $0$ & $0$  & $0$ & $0$ & $0$ & $0$ & $0$ & \cellcolor{MidnightBlue!40}$0$ & \cellcolor{MidnightBlue!40}$0$ \\
\hline
\cellcolor{ForestGreen!40}  $h^2_1(x_{5})$ & $0$ & $0$ & $0$ & $0$ & $0$  & $0$ & $0$ & $0$ & $0$ & $0$ & \cellcolor{MidnightBlue!40}$0$ & \cellcolor{MidnightBlue!40}$0$ \\
\hline
\cellcolor{ForestGreen!40}  $h^1_2(x'_{1})$ & $0$ & $0$ & $0$ & $0$ & $0$ & $0$ & $0$ & $0$ & $0$ & $0$ & $0$ & $0$ \\
\hline
\cellcolor{ForestGreen!40}  $h^2_2(x'_{1})$ & $0$ & $0$ & $0$ & $0$ & $0$ & $0$ & $0$ & $0$ & $0$ & $0$ & $0$ & $0$ \\
\hline
\end{tabular}
}
\caption{Matrix  associated with the GGS chain complex of $(M_1,X_1)$.}
\label{table:2}
\end{table}

The differential $d^1_{6}$ determines the homotopical dynamical cancellations of $h^1_1(x_9)$ and  $h^1_0(x_{10})$. In fact, since $n(h^1_1(x_9),h^1_0(x_{10}))=+1$, one can cancel these generators together with $h^1_0(x_{11})$, giving rise to the new singularity $x'_{10}$  of type $\mathcal{W}_2$ and attracting nature, as we can see in $(M_2,X_2)$ in Figure \ref{fig:exemplo_final_2}. The matrix associated to the GGS chain complex of $(M_2,X_2)$ is given by Table \ref{table:3}.

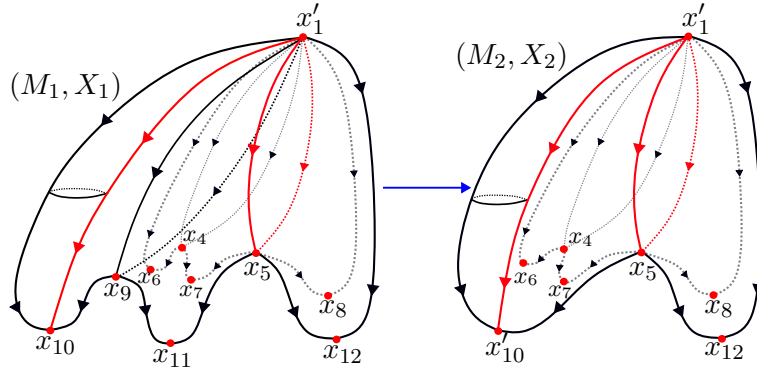
\begin{figure}[!ht]
\centering
\begin{overpic}[unit=1mm, scale=.25, width=10cm]{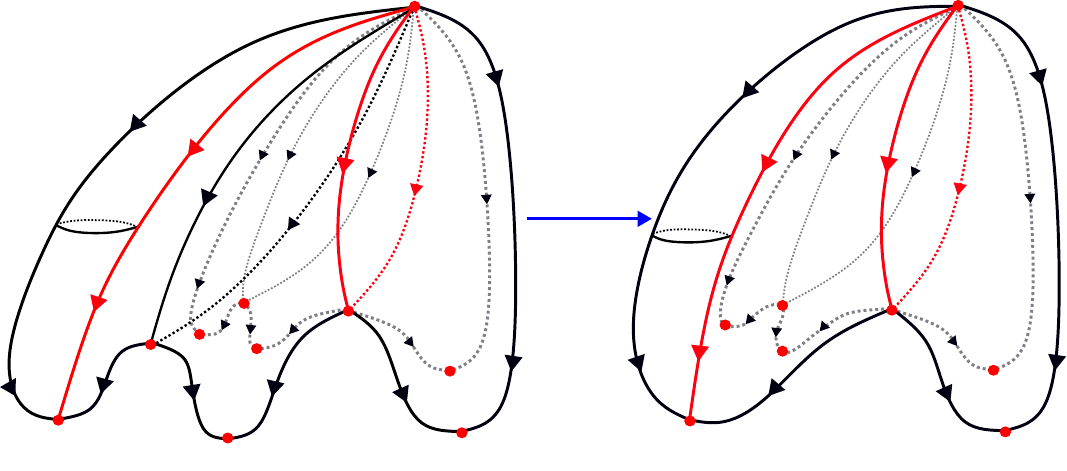}
        \put(0,36){$(M_1,X_1)$}
        \put(38,45){$x'_1$}
        \put(23,16.5){\footnotesize{$x_4$}}
        \put(31,12){$x_5$}
        \put(17,11){\footnotesize{$x_6$}}
        \put(22.5,9.3){\footnotesize{$x_7$}}
        \put(41,7){$x_8$}
        \put(12.5,9.3){$x_9$}
        \put(3.5,2.2){$x_{10}$}
        \put(19,0.5){$x_{11}$}
        \put(41.5,1){$x_{12}$}
        %%%%
        \put(59,40){$(M_2, X_2)$}
        \put(89,45){$x'_1$}
        \put(74,16.2){\footnotesize{$x_4$}}
        \put(82,12){$x_5$}
        \put(66.8,11){\footnotesize{$x_6$}}
        \put(71.3,9.3){\footnotesize{$x_7$}}
        \put(92,7){$x_8$}
        \put(63,1.5){$x'_{10}$}
        \put(92,1.5){$x_{12}$}
\end{overpic}
\caption{GGS manifold $M$ after first cancellation and after second cancellation.}
\label{fig:exemplo_final_2}
\end{figure}

\begin{table}
\centering
{
\begin{tabular}{|c|c|c|c|c|c|c|c|c|c|c|c|c| } 
\hline
& \cellcolor{ForestGreen!40} $h^1_0(x_{12})$ & \cellcolor{ForestGreen!40} $h^1_0(x_8)$ & \cellcolor{ForestGreen!40}  $h^1_0(x'_{10})$ & \cellcolor{ForestGreen!40} $h^1_0(x_{7})$  & \cellcolor{ForestGreen!40} $h^1_0(x_{8})$  & \cellcolor{ForestGreen!40} $h^1_1(x_{4})$ & \cellcolor{ForestGreen!40} $h^1_1(x_{5})$ & \cellcolor{ForestGreen!40} $h^2_1(x_5)$ & \cellcolor{ForestGreen!40} $h^1_2(x'_{1})$ &\cellcolor{ForestGreen!40}  $h^2_2(x'_1)$\\
\hline
\cellcolor{ForestGreen!40} $h^1_0(x_{12})$ & $0$ & $0$ & $0$ & $0$ & $0$ & \cellcolor{WildStrawberry!40}$0$ & \cellcolor{WildStrawberry!40}$0$ & \cellcolor{WildStrawberry!40}$+1$  & $0$ & $0$ \\
\hline
\cellcolor{ForestGreen!40} $h^1_0(x_8)$ & $0$ & $0$  & $0$ & $0$ & $0$  & \cellcolor{WildStrawberry!40}$0$ & \cellcolor{WildStrawberry!40}$+1$ & \cellcolor{WildStrawberry!40}$0$  & $0$ & $0$ \\
\hline
\cellcolor{ForestGreen!40} $h^1_0(x'_{10})$ & $0$ & $0$ & $0$ & $0$ & $0$  & \cellcolor{WildStrawberry!40}$0$ & \cellcolor{WildStrawberry!40}$0$ & \cellcolor{WildStrawberry!40}$-1$ & $0$ & $0$ \\
\hline
\cellcolor{ForestGreen!40} $h^1_0(x_{7})$ & $0$  & $0$ & $0$ & $0$ & $0$  & \cellcolor{WildStrawberry!40}$+1$ & \cellcolor{WildStrawberry!40}$-1$ & \cellcolor{WildStrawberry!40}$0$  & $0$ & $0$ \\
\hline
\cellcolor{ForestGreen!40} $h^1_0(x_{6})$ & $0$  & $0$ & $0$ & $0$ & $0$ & \cellcolor{WildStrawberry!40}\tikz[baseline=(current bounding box.center)]\node[draw, circle, fill=white, inner sep=0.3pt]{-1};& \cellcolor{WildStrawberry!40}$0$ & \cellcolor{WildStrawberry!40}$0$  & $0$ & $0$ \\
\hline
\cellcolor{ForestGreen!40} $h^1_1(x_{4})$ & $0$  & $0$ & $0$ & $0$  & $0$ & $0$ & $0$ & $0$  & \cellcolor{MidnightBlue!40}$0$ & \cellcolor{MidnightBlue!40}$0$ \\
\hline
\cellcolor{ForestGreen!40} $h^1_1(x_{5})$ & $0$  & $0$ & $0$ & $0$ & $0$ & $0$ & $0$ & $0$ & \cellcolor{MidnightBlue!40}$0$ & \cellcolor{MidnightBlue!40}$0$ \\
\hline
\cellcolor{ForestGreen!40} $h^2_1(x_{5})$ & $0$  & $0$ & $0$ & $0$  & $0$ & $0$ & $0$ & $0$  &  \cellcolor{MidnightBlue!40}$0$ &  \cellcolor{MidnightBlue!40}$0$ \\
\hline
\cellcolor{ForestGreen!40} $h^1_2(x'_{1})$ & $0$ & $0$ & $0$ & $0$ & $0$ & $0$ & $0$  & $0$ & $0$ & $0$ \\
\hline
\cellcolor{ForestGreen!40} $h^2_2(x'_{1})$ & $0$ & $0$ & $0$ & $0$ & $0$ & $0$ & $0$  & $0$ & $0$ & $0$\\
\hline
\end{tabular}
}
\caption{Matrix  associated with the GGS chain complex of $(M_2,X_2)$.}
\label{table:3}
\end{table}

\begin{figure}[!ht]
\centering
\begin{overpic}[unit=1mm, scale=.25, width=9cm]{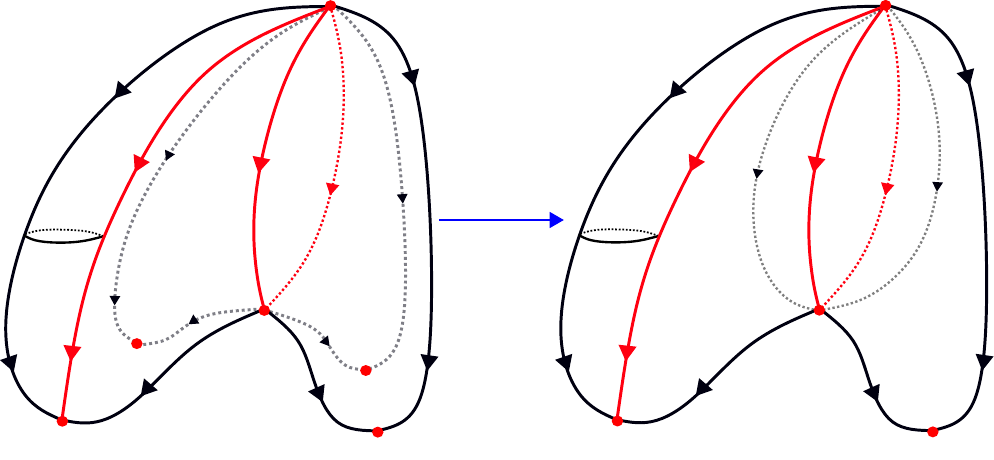}
        \put(0,41){$(M_3, X_3)$}
        \put(32,47){$x'_1$}
        \put(24.5,12){$x_5$}
        \put(10.4, 8.1){$x'_6$}
        \put(35,6){$x_8$}
        \put(4.5,0){$x'_{10}$}
        \put(36,0){$x_{12}$}
        %%%%
        \put(55,41){$(M_4, X_4)$}
        \put(88,47){$x'_1$}
        \put(80,11.4){$x'_5$}
        \put(60,0){$x'_{10}$}
        \put(92,0){$x_{12}$}
\end{overpic}
\caption{GGS manifold $M$ after third cancellation and after fourth cancellation.}
\label{fig:exemplo_final_3}
\end{figure}
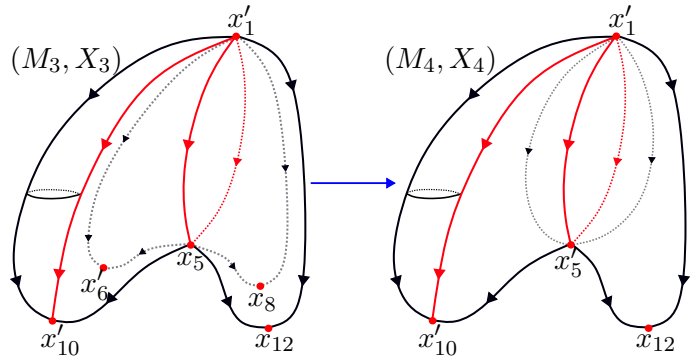

The differential $d^3_7$ determines the homotopical dynamical cancellation  of $h^1_1(x_4)$ and $h^1_0(x_{6})$, which can be done since $n(h^1_1(x_4),h^1_0(x_{6}))=-1$.  Cancelling these generators together with $h^1_0(x_7)$,  gives rise to $(M_3, X_3)$ which contains an attracting regular singularity  $x'_6$, as shown in Figure \ref{fig:exemplo_final_3}. The matrix associated with the corresponding GGS chain complex is given in Table \ref{table:4}.

\begin{table}[!ht]
\centering
%\resizebox{\textwidth}{!}{
\begin{tabular}{|c|c|c|c|c|c|c|c|c|c|c|c|c| } 
\hline
& \cellcolor{ForestGreen!40} $h^1_0(x_{12})$& \cellcolor{ForestGreen!40} $h^1_0(x_8)$ & \cellcolor{ForestGreen!40} $h^1_0(x'_{10})$  & \cellcolor{ForestGreen!40} $h^1_0(x'_{6})$  &  \cellcolor{ForestGreen!40} $h^1_1(x_{5})$ & \cellcolor{ForestGreen!40} $h^2_1(x_5)$ & \cellcolor{ForestGreen!40}  $h^1_2(x'_{1})$ & \cellcolor{ForestGreen!40} $h^2_2(x'_1)$\\
\hline
\cellcolor{ForestGreen!40} $h^1_0(x_{12})$ & $0$ & $0$ & $0$ & $0$ & \cellcolor{WildStrawberry!40}$0$ & \cellcolor{WildStrawberry!40}$+1$ & $0$ & $0$  \\
\hline
\cellcolor{ForestGreen!40} $h^1_0(x_8)$ & $0$ & $0$ & $0$ & $0$ & \cellcolor{WildStrawberry!40}$+1$  & \cellcolor{WildStrawberry!40}$0$ & $0$ & $0$ \\
\hline
\cellcolor{ForestGreen!40} $h^1_0(x'_{10})$ & $0$  & $0$ & $0$ & $0$ & \cellcolor{WildStrawberry!40}$0$  & \cellcolor{WildStrawberry!40}$-1$ & $0$ & $0$  \\
\hline
\cellcolor{ForestGreen!40} $h^1_0(x'_{6})$ & $0$  & $0$ & $0$ & $0$ & \cellcolor{WildStrawberry!40}\tikz[baseline=(current bounding box.center)]\node[draw, circle, fill=white, inner sep=0.3pt]{-1}; & \cellcolor{WildStrawberry!40}$0$ & $0$ & $0$   \\
\hline
\cellcolor{ForestGreen!40} $h^1_1(x_{5})$ & $0$  & $0$ & $0$ & $0$ & $0$ & $0$ & \cellcolor{MidnightBlue!40}$0$ & \cellcolor{MidnightBlue!40}$0$  \\
\hline
\cellcolor{ForestGreen!40} $h^2_1(x_{5})$ & $0$  & $0$ & $0$ & $0$  & $0$ & $0$ & \cellcolor{MidnightBlue!40}$0$ & \cellcolor{MidnightBlue!40}$0$   \\
\hline
\cellcolor{ForestGreen!40} $h^1_2(x'_{1})$ & $0$ & $0$ & $0$ & $0$ & $0$ & $0$ & $0$  & $0$  \\
\hline
\cellcolor{ForestGreen!40} $h^2_2(x'_{1})$ & $0$ & $0$ & $0$ & $0$ & $0$ & $0$ & $0$  & $0$ \\
\hline
\end{tabular}
%}
\caption{Matrix  associated with the GGS chain complex of $(M_3,X_3)$.}
\label{table:4}
\end{table}

The differential $d^5_8$ determines the homotopical dynamical cancellation of $h^1_1(x_5)$ and $h^1_0(x'_{6})$, which can be performed since $n(h^1_1(x_5),h^1_0(x'_{6}))=-1$. Thus, cancelling these generators together with $h^1_0(x_8)$ gives rise to $(M_4, X_4)$ which contains a double saddle attractor at $x'_5$, as shown in Figure \ref{fig:exemplo_final_3}. The matrix  associated with the corresponding GGS complex is given by Table \ref{table:5}.

\begin{table}[!ht]
\centering
\begin{tabular}{|c|c|c|c|c|c|c|c|c|c|c|c|c|} 
\hline
& \cellcolor{ForestGreen!40} $h^1_0(x_{12})$ & \cellcolor{ForestGreen!40} $h^1_0(x'_{10})$  & \cellcolor{ForestGreen!40} $h^1_0(x'_{5})$  & \cellcolor{ForestGreen!40} $h^1_1(x'_5)$ & \cellcolor{ForestGreen!40} $h^1_2(x'_{1})$ & \cellcolor{ForestGreen!40} $h^2_2(x'_1)$\\
\hline
\cellcolor{ForestGreen!40} $h^1_0(x_{12})$ & $0$ & $0$ & $0$ & \cellcolor{WildStrawberry!40}$+1$ & $0$  & $0$\\
\hline
\cellcolor{ForestGreen!40} $h^1_0(x'_{10})$ & $0$  & $0$ & $0$ & \cellcolor{WildStrawberry!40}\tikz[baseline=(current bounding box.center)]\node[draw, circle, fill=white, inner sep=0.3pt]{-1}; & $0$  & $0$ \\
\hline
\cellcolor{ForestGreen!40} $h^1_0(x'_{5})$ & $0$  & $0$ & $0$ & \cellcolor{WildStrawberry!40}$0$ & $0$ & $0$  \\
\hline
\cellcolor{ForestGreen!40} $h^1_1(x'_{5})$ & $0$  & $0$ & $0$ & $0$  & \cellcolor{MidnightBlue!40}$0$ & \cellcolor{MidnightBlue!40}$0$  \\
\hline
\cellcolor{ForestGreen!40} $h^1_2(x'_{1})$ & $0$ & $0$ & $0$ & $0$ & $0$ & $0$  \\
\hline
\cellcolor{ForestGreen!40} $h^2_2(x'_{1})$ & $0$ & $0$ & $0$ & $0$ & $0$ & $0$ \\
\hline
\end{tabular}
\caption{Matrix  associated with the GGS chain complex of $(M_4,X_4)$.}
\label{table:5}
\end{table}

\begin{figure}[!ht]
\centering
\begin{overpic}[unit=1mm, scale=.25, width=9cm]{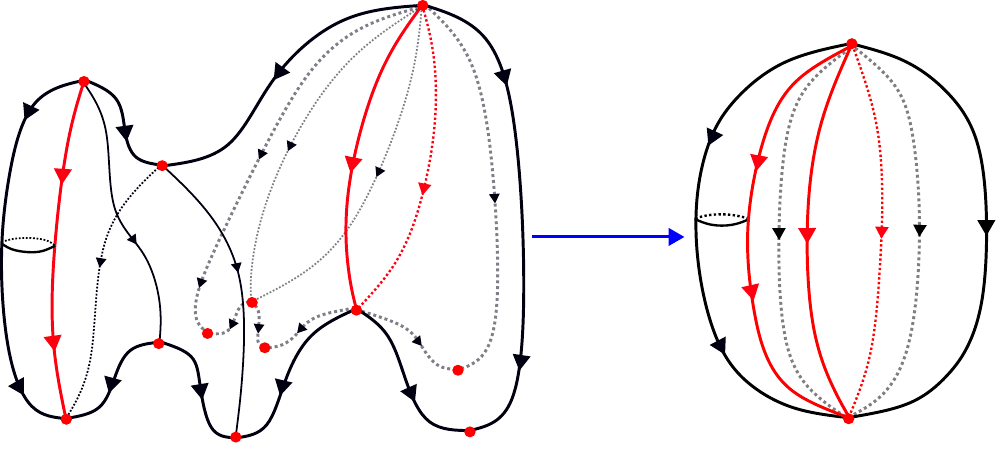}
        \put(0,45){$(M,X)$}
        \put(6.8,40.5){$x_1$}
        \put(41,48){$x_2$}
        \put(15,32){$x_3$}
        \put(25,18.5){\footnotesize{$x_4$}}
        \put(34,13){$x_5$}
        \put(19.5,11.5){\footnotesize{$x_6$}}
        \put(25,10){\footnotesize{$x_7$}}
        \put(45,7){$x_8$}
        \put(14,10){$x_9$}
        \put(4,2.5){$x_{10}$}
        \put(21,1){$x_{11}$}
        \put(45,1){$x_{12}$}
        %%%%
        \put(60,41){$(M_5,X_5)$}
        \put(84,44.2){$x'_{1}$}
        \put(84,1.5){$x'_{5}$}
    \end{overpic}
\caption{Initial GGS manifold and final manifold with core flow, after $5$ cancellations.}
\label{fig:exemplo_final_4}
\end{figure}
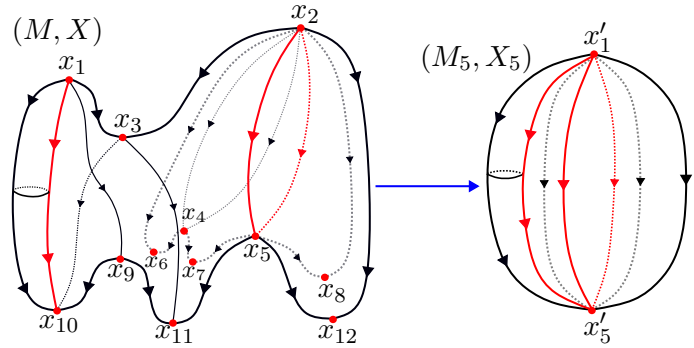

Finally, the differential $d^7_9$ determines the homotopical dynamical cancellation of 
$h^1_1(x'_5)$ and $h^1_0(x'_{10})$, which can be performed since $n(h^1_1(x'_5),h^1_0(x'_{10}))=-1$, Thus  cancelling these generators together with $h^1_0(x_{12})$, yields a GGS  manifold equipped with a core flow $(M_5,X_5)$, in which  $x'_1$ and $x'_5$ are of type $\mathcal{W}_2\mathcal{D}_2$ with repelling and attracting nature, respectively, as shown in Figure \ref{fig:exemplo_final_4}. The matrix  associated with the corresponding GGS chain complex is given in Table~\ref{table:6}.

\begin{table}[!ht]
\centering
%\resizebox{\textwidth}{!}{
\begin{tabular}{|c|c|c|c|c|c|c|c|c|c|c|c|c| } 
\hline
& \cellcolor{ForestGreen!40} $h^1_0(x'_5)$ & \cellcolor{ForestGreen!40} $h^2_0(x'_{5})$  & \cellcolor{ForestGreen!40} $h^1_2( x'_{1})$  & \cellcolor{ForestGreen!40} $h^2_2(x'_1)$\\
\hline
\cellcolor{ForestGreen!40} $h^1_0( x'_5)$ & $0$ & $0$ & $0$ & $0$\\
\hline
\cellcolor{ForestGreen!40} $h^2_0(x'_{5})$ & $0$  & $0$ & $0$ & $0$  \\
\hline
\cellcolor{ForestGreen!40} $h^1_2( x'_{1})$ & $0$  & $0$ & $0$ & $0$  \\
\hline
\cellcolor{ForestGreen!40} $h^2_2(x'_{1})$ & $0$  & $0$ & $0$ & $0$    \\
\hline
\end{tabular}
%}
\caption{Matrix associated with the GGS chain complex of $(M_5,X_5)$.}
\label{table:6}
\end{table}
\end{example}

% ---> COMANDO PARA PAUSAR O SUMÁRIO <---
\addtocontents{toc}{\protect\setcounter{tocdepth}{-1}}

\subsection*{Author contributions}
All authors contributed equally to the conception, development, and writing of the manuscript. All authors read and approved the final version of the paper.

\subsection*{Conflict of interest}
The authors declare no potential conflict of interests.

% ---> COMANDO PARA LIGAR O SUMÁRIO NOVAMENTE <---
% (Isso garante que as "References" continuem aparecendo no sumário)
\addtocontents{toc}{\protect\setcounter{tocdepth}{1}}

\bibliographystyle{amsplain} % Estilo padrão da AMS (numérico). Se preferir alfabético, use {amsalpha}
\bibliography{references}   % O nome do seu arquivo .bib (sem a extensão .bib)

\end{document}